\documentclass[a4paper,english,12pt]{article}

\usepackage{geometry}
\usepackage{pdfpages}
\usepackage[T1]{fontenc}
\usepackage[utf8]{inputenc}
\usepackage{babel}
\usepackage{amsmath}
\usepackage{amssymb}
\usepackage{amsthm}
\usepackage{textcomp}
\usepackage{enumitem}
\usepackage{dsfont}
\usepackage{mathrsfs}
\usepackage{cite}
\usepackage{color}
\usepackage{breqn}
\usepackage[colorlinks=true,urlcolor=blue,linkcolor=blue]{hyperref}

\newcommand{\R}{\mathbb{R}}

\newcommand{\Z}{\mathbb{Z}}
\newcommand{\N}{\mathbb{N}}
\newcommand{\C}{\mathbb{C}}

\newcommand{\T}{\mathbb{T}}

\newcommand{\id}{\textnormal{Id}}

\newcommand{\classeC}{\mathcal{C}}

\renewcommand{\Im}{\textnormal{Im}}
\renewcommand{\Re}{\textnormal{Re}}
\newcommand{\un}{\mathds{1}}

\newcommand{\half}{\frac{1}{2}}

\newcommand{\hilbert}{H}
\newcommand{\hamilton}{\mathcal{H}}

\newcommand{\intt}{\frac{1}{2\pi}\int_0^{2\pi}}

\renewcommand\d{\,{\mathrm d}}
\newcommand\e{\,{\mathrm e}}

\newcommand{\longrightarroww}[2] {\mathop{\longrightarrow}\limits_{#1}^{#2}}

\newcommand{\weaklim}[2] {\mathop{\rightharpoonup}\limits_{#1}^{#2}}

\newtheorem{mydef}{Definition}[section]
\newtheorem{thm}[mydef]{Theorem}
\newtheorem{lem}[mydef]{Lemma}
\newtheorem{prop}[mydef]{Proposition}
\newtheorem{cor}[mydef]{Corollary}
\newtheorem{rk}[mydef]{Remark}

\title{Long time behavior of solutions for a damped Benjamin-Ono equation.}
\author{Louise Gassot}
\date{}
\newcommand{\Addresses}{{
  \footnotesize
\noindent
\textsc{Département de mathématiques et applications, École normale supérieure, CNRS, PSL University, 75005 Paris, France
\\
 Université Paris-Saclay, CNRS, Laboratoire de mathématiques d’Orsay, 91405 Orsay, France}\par\nopagebreak
  \noindent
  \textit{E-mail address:} \texttt{louise.gassot@universite-paris-saclay.fr}
}}
\begin{document}
\maketitle
\abstract
{We consider the Benjamin-Ono equation on the torus with an additional damping term on the smallest Fourier modes ($\cos$ and $\sin$). We first prove global well-posedness of this equation in $L^2_{r,0}(\T)$. Then, we describe the weak limit points of the trajectories in $L^2_{r,0}(\T)$ when time goes to infinity, and show that these weak limit points are strong limit points. Finally, we prove the boundedness of higher-order Sobolev norms for this equation. Our key tool is the Birkhoff map for the Benjamin-Ono equation, that we use as an adapted nonlinear Fourier transform.
}

\tableofcontents

\section{Introduction}

In this paper, we consider the Benjamin-Ono equation on the torus~\cite{Benjamin1967},~\cite{Ono1977}
\begin{equation}\tag{BO}\label{eq:bo}
\partial_t u
	=H\partial_{xx}u-\partial_x(u^2)
\end{equation}
with an additional weak damping term on the smallest Fourier modes defined as follows. Fix a parameter $\alpha>0$, the damped Benjamin-Ono equation writes
\begin{equation}\tag{BO-$\alpha$}\label{eq:damped}
\partial_t u+\alpha(\langle u|\cos\rangle\cos+\langle u|\sin\rangle\sin)
	=H\partial_{xx}u-\partial_x(u^2).
\end{equation}
The operator $H$ is the Hilbert transform
\[
\hilbert f(x)=\sum_{n\in\Z}-i\,\mathrm{sgn}(n)\widehat{f}(n)\e^{inx},
	\quad f=\sum_{n\in\Z}\widehat{f}(n)\e^{inx},
	\quad \widehat{f}(n)=\intt f(x)\e^{-inx}\d x,
\]
with the convention that $\mathrm{sgn}(\pm n)=\pm 1$ if $n\geq 1$, and $\mathrm{sgn}(0)=0$.

Our aim is to investigate how this damping can affect the long-time behavior of trajectories. Note that our choice of damping breaks both the Hamiltonian and the integrable structures of the Benjamin-Ono equation, moreover, the damping term is not a small perturbation of the Benjamin-Ono equation. For a general damping term with such properties, the problem seems out of the scope of usual techniques.

Our main tool is the construction and study of a Birkhoff map for the Benjamin-Ono equation by Gérard, Kappeler and Topalov in~\cite{GerardKappeler2019,GerardKappelerTopalov2020,GerardKappelerTopalov2020-2} (see Gérard~\cite{Gerard2019} for a recent survey of these results), implying that the Benjamin-Ono equation is integrable in the strongest possible sense. This transformation should be seen as a nonlinear Fourier transform adapted to the Benjamin-Ono equation. The classical Fourier transform has proven to be a powerful tool for studying nonlinear partial differential equations, for instance with the development of pseudodifferential and paradifferential calculus. In the same spirit, we will see that the Birkhoff map transforms the damped equation~\eqref{eq:damped} into a system of ODEs on the nonlinear Fourier coefficients $(\zeta_n(u))_{n\geq 1}$, which enables us to study the qualitative properties of the solutions in infinite time.

The Birkhoff map writes
\begin{equation}\label{def:birkhoff}
\Phi:u\in L^2_{r,0}(\T)\mapsto (\zeta_n(u))_{n\geq 1}\in h^{\half}_+,
\end{equation}
where $L^2_{r,0}(\T)$ is the space of real-valued functions in $L^2(\T)$ with zero mean, and
\[
h^{\half}_+=\{\zeta=(\zeta_n)_{n\geq 1}\in\C^\N\mid \|\zeta\|_{h^{\half}_+}^2=\sum_{n=1}^{+\infty}n|\zeta_n|^2<+\infty\}.
\]
This map provides a system of coordinates in which the Benjamin-Ono equation can be solved by quadrature. Indeed, if $u$ is a solution to the Benjamin-Ono equation with initial data $u_0$, then there exist frequencies $\omega_n(u_0)=n^2-2\sum_{k\geq 1}\min(k,n)|\zeta_k(u_0)|^2$, $n\geq 1$, such that  for $t\in\R$ and $n\geq 1$,
\[
\zeta_n(u(t))=\zeta_n(u_0)e^{i\omega_n(u_0)t}.
\]

\subsection{Main results}


First, we establish the global well-posedness of the damped Benjamin-Ono equation~\eqref{eq:damped} in~$L^2_{r,0}(\T)$, and the weak sequential continuity of the solution map.

\begin{thm}[Global well-posedness]\label{thm:GWP}
For all $T>0$ and $u_0\in L^2_{r,0}(\T)$, there exists a unique solution of equation~\eqref{eq:damped} in the distribution sense in $\classeC([0,T],L^2_{r,0}(\T))$ with initial data $u_0$.

Moreover, the solution map is continuous and weakly sequentially continuous from $L^2_{r,0}(\T)$ to $\classeC([0,T],L^2_{r,0}(\T))$.
\end{thm}

The solutions are global thanks a Lyapunov functional controlling the $L^2$ norm
\begin{equation}\label{eq:normL2}
\frac{\d}{\d t}\|u(t)\|_{L^2(\T)}^2+2\alpha|\langle u(t)|e^{ix}\rangle|^2
	=0.
\end{equation}

In order to establish global well-posedness of equation~\eqref{eq:damped} in $L^2_{r,0}(\T)$, one could also apply the techniques used for the Benjamin-Ono equation from Molinet~\cite{Molinet2008} and Molinet, Pilod~\cite{MolinetPilod2012CauchyBO}. However, their techniques do not lead to the weak sequential continuity of the flow map, which is needed for describing the weak limit points of~$u(t)$ as $t\to+\infty$.
\newline

Then, we show that the weak limit points of trajectories as time goes to infinity follow the LaSalle principle. Moreover, we prove that the trajectories are relatively compact in $L^2_{r,0}(\T)$, in other words, all weak limit points are strong limit points.

\begin{thm}[Weak limit points]\label{thm:LaSalle}
Let $u$ be a solution to the damped Benjamin-Ono equation~\eqref{eq:damped} in $L^2_{r,0}(\T)$. Let $v_0$ be a limit point of the sequence $(u(t))_{t\geq 0}$ for the weak topology in $L^2_{r,0}(\T)$ as $t\to+\infty$. Then we have the following results.

\begin{enumerate}
\item The solution $v$ to the Benjamin-Ono equation~\eqref{eq:bo} with initial data $v_0$ satisfies:
\begin{equation}\label{eq:Lasalle}
\forall t\in\R,\quad\langle v(t)|e^{ix}\rangle=0.
\end{equation}

Moreover, condition~\eqref{eq:Lasalle} is equivalent to the fact that the weak limit point $v_0$ does not have two consecutive nonzero Birkhoff coordinates:
\[
\forall n\in\N, \quad \zeta_n(v_0)\zeta_{n+1}(v_0)=0,
\]
with the convention $\zeta_0(v_0)=1$.

\item Given a solution $u$, all the limit points for $u$ have the same actions: there exists a sequence $(\gamma_n^\infty)_{n\geq 1}$ such that for all limit point $v_0$ associated to $u$, for all $n\geq 1$,
\[
|\zeta_n(v_0)|^2=\gamma_n^{\infty}.
\]

\item The convergence is strong in $L^2_{r,0}(\T)$: if $v_0$ is a weak limit point associated to the sequence $(u(t_k))_k$, then
\[
\|u(t_k)-v_0\|_{L^2(\T)}\longrightarroww{k\to+\infty}{} 0.
\]
\end{enumerate}
\end{thm}

Gérard, Kappeler and Topalov proved that the trajectories of the undamped Benjamin-Ono equation~\eqref{eq:bo} are almost-periodic in $H^s_{r,0}(\T)$ for all $s>-\half$ (see~\cite{GerardKappeler2019}, Theorem 1.3 and~\cite{GerardKappelerTopalov2020-2}, Theorem 3 and Corollary 8).  
As a consequence, the solutions are recurrent (see~\cite{GerardKappeler2019}, Remark~1.4~(ii)): for all solutions $u$ with initial data $u_0\in L^2_{r,0}(\T)$, there exists a sequence $t_n\to+\infty$ such that $u(t_n)$ tends to $u_0$ in $L^2_{r,0}(\T)$. Both properties do not hold for the damped equation~\eqref{eq:damped} because of characterization~\eqref{eq:Lasalle}.

In~\cite{Tzvetkov2010,TzvetkovVisciglia2013,TzvetkovVisciglia2014, TzvetkovVisciglia2015,Deng2015,DengTzvetkovVisciglia2015}, the authors build a sequence $(\mu_n)_n$ of gaussian measures on $L^2(\T)$, invariant by the flow of the periodic Benjamin-Ono equation~\eqref{eq:bo}, and associated to the conservation laws for this equation. Each measure $\mu_n$ is concentrated on $H^s(\T)$ for $s<n-\half$, and satisfies $\mu_n(H^{n-\half}(\T))=0$. Formally, it is defined from the $n$-th conservation law $\hamilton_n$ as a renormalization of the formula $\d\mu_n=e^{-\hamilton_n(u)}\d u$. 
Sy~\cite{Sy2018} extends this construction to get a measure concentrated on $\classeC^{\infty}(\T)$. With this approach, it was already possible to prove that with probability $1$ with respect to the corresponding measure, the solutions of~\eqref{eq:bo} display the recurrence property mentionned above (see~\cite{TzvetkovVisciglia2014}, Corollary~1.3 and~\cite{Sy2018}, Corollary~1.2).

An equation similar to~\eqref{eq:damped} has been introduced and studied by Gérard and Grellier in~\cite{GerardGrellier2019} starting from the Szeg\H{o} equation. For $\alpha>0$, the damped Szeg\H{o} equation on the torus writes
\begin{equation}\tag{Sz-$\alpha$}\label{eq:damped_szego}
i\partial_t u+i\alpha\langle u|\un\rangle=\Pi(|u|^2u),
\end{equation}
where $\Pi$ is the Szeg\H{o} projector from $L^2(\T)$ onto the space 
\[L^2_+(\T)=\{u\in L^2(\T)\mid \forall n<0, \widehat{u}(n)=0\}.\] 
Characterization~\eqref{eq:Lasalle} is then similar for~\eqref{eq:damped} and~\eqref{eq:damped_szego}. Indeed, let $v_0$ be a weak limit point of a solution $(u(t))_{t\in\R}$ of~\eqref{eq:damped_szego} for the weak topology of the natural energy space $H^{\half}_+(\T):=\Pi(H^{\half}(\T))$, then the solution $v$ to the cubic Szeg\H{o} equation with initial data $v_0$ satisfies $\langle v(t)|\un\rangle=0$ for all times. As a consequence, the trajectories of~\eqref{eq:damped_szego} cannot be almost-periodic in $H^{\half}_+(\T)$, whereas almost-periodicity holds for the Szeg\H{o} equation~\cite{GerardGrellier2017}, similarly as the Benjamin-Ono case.

However, our result implies that the trajectories for~\eqref{eq:damped} are relatively compact in~$L^2_{r,0}(\T)$, whereas many initial data for~\eqref{eq:damped_szego} lead to trajectories which are not relatively compact in the natural energy space $H^{\half}_+(\T)$ (this fact is a consequence of the proof of Theorem 3.1 in~\cite{GerardGrellier2019}). One can also note that in the study of~\eqref{eq:damped_szego}, the authors strongly use the fact that one of the two known Lax pairs for the Szeg\H{o} equation still holds, whereas there is no known Lax pair for~\eqref{eq:damped}.
\newline


The last part of our paper is devoted to the boundedness for higher-order Sobolev norms.  For the Benjamin-Ono equation on the torus~\eqref{eq:bo}, the infinite number conservation laws from Nakamura~\cite{Nakamura1979} and Bock, Kruskal~\cite{BockKruskal1979} provide bounds on all the Sobolev norms $H^s$, $s\geq 0$. We prove that the $H^s$ Sobolev norms also stay bounded when $0\leq s<\frac{3}{2}$ for the damped Benjamin-Ono equation~\eqref{eq:damped}.
\begin{thm}[Higher-order Sobolev norms]\label{thm:higher_Hs}
Fix $0\leq s<\frac{3}{2}$. Let $u$ be a solution to the damped equation~\eqref{eq:damped} in $L^2_{r,0}(\T)$ such that the initial data $u_0$ belongs to $H^s(\T)$. Then there exists some constant $C_s(\|u_0\|_{H^s(\T)})>0$ such that for all $t\in\R_+$,
\[
\|u(t)\|_{H^s(\T)}\leq C_s(\|u_0\|_{H^s(\T)}).
\]
\end{thm}

The proof of Theorem~\ref{thm:higher_Hs} could be adapted to show that there is no growth of Sobolev norms in $H^s_{r,0}(\T)$ for any $s\geq 0$. However, our method becomes more and more technical as the exponent $s$ increases, this is why we chose to stop at $s=\frac{3}{2}$. See Remark~\ref{rk:higher_Hs} for more details.

The higher-order Sobolev norms for the Benjamin-Ono equation with and without damping have the same behavior, which is in contrast with the Szeg\H{o} equation. Indeed, Gérard and Grellier proved in~\cite{GerardGrellier2019} that for the damped Szeg\H{o} equation~\eqref{eq:damped_szego}, there exists a non empty open subset of $H^{\half}_+(\T)$ for which every solution $u$ with initial data in $H^s(\T)$, $s>\half$, has an exploding orbit in~$H^s(\T)$:
\[
\|u(t)\|_{H^s(\T)}\longrightarroww{t\to+\infty}{} +\infty.
\]
However, the undamped Szeg\H{o} equation exhibits weak turbulence~\cite{GerardGrellier2017}, in the sense that there exists a dense $G_\delta$ subset of initial data in $L^2_+\cap\classeC^\infty$ such that for every $s>\half$, the solution $u$ satisfies $\limsup_t\|u(t)\|_{H^s(\T)}=+\infty$, whereas $\liminf_t\|u(t)\|_{H^s(\T)}<+\infty$.

In order to prove Theorem~\ref{thm:higher_Hs}, we introduce Lyapunov functionals controlling the higher-order Sobolev norms in Birkhoff coordinates. Note that from~\cite{GerardKappelerTopalov2020}, Proposition 5 in Appendix~A, the Birkhoff map $\Phi$ and its inverse map bounded subsets of the Sobolev space $H^s_{r,0}(\T)$ to bounded subsets of sequences in 
\[h^{\half+s}_+=\{\zeta=(\zeta_n)_{n\geq 1}\in\C^\N\mid\|\zeta\|_{h^{\half+s}_+}^2= \sum_{n}n^{1+2s}|\zeta_n|^2<+\infty\}.
\]
This leads us to consider higher-order Lyapunov functionals in Birkhoff coordinates of the form
\[
P_s(u)
	=\sum_{n\geq 1}w_n\gamma_n(u),
\]
where we choose weights $w_n$ satisfying $w_n\approx n^{1+2s}$. The strategy of introducing Lyapunov functionals controlling higher-order Sobolev norms has been implemented for instance in~\cite{AlazardBresch2020} and~\cite{AlazardMeunierSmets2019} for various equations describing free surface flows in fluid dynamics. In our paper, this strategy is implemented for the first time using the Birkhoff coordinates of a close integrable system.

\subsection{Long-time behavior for equations of Benjamin-Ono type}

As a comparison, let us now recall how the long-time dynamics for the Benjamin-Ono equation would be affected by the change of geometry, the introduction of a perturbation or the addition of a diffusive damping term. We refer to Saut~\cite{Saut2018} for a general survey on the Benjamin-Ono equation.

\paragraph{Benjamin-Ono equation on the real line}
On the real line, one expects the dispersion to change the behavior of Benjamin-Ono solutions. Ifrim and Tataru~\cite{IfrimTataru2019} proved that for small localized initial data, there holds linear dispersive decay for large time scales. It is also conjectured that a solution associated to small initial data either is dispersive, or decomposes into the sum of a soliton and a dispersive part (soliton resolution conjecture).

The Benjamin-Ono equation on the line is also integrable. The Lax pair formulation is the same as on the torus, however, the inverse spectral theory, formally described in Ablowitz, Fokas~\cite{AblowitzFokas1983}, is only performed for sufficiently small and decaying initial data~\cite{CoifmanWickerhauser1990}. Wu~\cite{Wu2016} studied the spectrum of the Lax operator for the Benjamin-Ono equation, then solved the direct scattering problem for arbitrary sufficiently decaying initial data~\cite{Wu2017}. The recent construction of the Birkhoff map for multisolitons by Sun~\cite{Sun2020} constitutes a first step towards the soliton resolution conjecture.

\paragraph{Perturbations of the Benjamin-Ono equation on the torus}
Thanks to extensions of the KAM theory, one can investigate if the periodic, quasi periodic or almost periodic nature of trajectories is preserved by perturbation of an integrable equation. This kind of studies goes back to Kuksin~\cite{Kuksin1987} and Wayne~\cite{Wayne1990}, we refer to Craig~\cite{Craig2000} and Berti~\cite{Berti2019} for a detailed survey.

Concerning perturbations of the Benjamin-Ono equation in the neighborhood of the origin, we mention the following results. In~\cite{LiuYuan2011}, Liu and Yuan consider a class of unbounded Hamiltonian perturbations of the periodic Benjamin-Ono equation~\eqref{eq:bo}, construct KAM tori and deduce that there exist quasi-periodic trajectories. Mi and Zhang~\cite{MiZhang2014} extend this result to more general unbounded quasi-periodic Hamiltonian perturbations. 
In~\cite{Baldi2013}, Baldi proves that there exist periodic solutions for some non Hamiltonian reversible perturbations of a nonlinear cubic Benjamin-Ono equation
\[
\partial_t u+\hilbert\partial_{xx}u+\partial_x(u^3)=0.
\]
For generalized Benjamin-Ono equations of the form
\[
\partial_t u+\hilbert\partial_{xx}u+\partial_x(f(u))=0,
\]
Bernier and Grébert~\cite{BernierGrebert2020} establish approximations of trajectories on large time scales for generic initial data close to the origin.


In order to investigate further Hamiltonian perturbations of the Benjamin-Ono equation through the KAM theory and tackle trajectories which are not necessarily close to the origin, it would be useful to prove that the Birkhoff map $\Phi$ is real analytic. As a first step, Gérard, Kappeler et Topalov prove in~\cite{GerardKappelerTopalov2020-2} that the action map $u\mapsto (|\zeta_n(u)|^2)_n$ is real analytic from Sobolev spaces $H^s_{r,0}(\T)$ onto weighted $\ell^1$ spaces.

\paragraph{Damped Benjamin-Ono equations} Adding a damping of diffusive type can lead to a decay of solutions or the existence of a global attractor. A global attractor is a compact set, invariant by the flow of the equation, and attracting all trajectories uniformly on bounded sets of initial data.

Several physical damping terms can be chosen depending on the nature of dissipation, in the same way as for the Korteweg-de Vries equation, see Ott and Sudan~\cite{OttSudan1970}. Grimshaw, Smyth and Stepanyants  (see~\cite{GrimshawSmythStepanyants2018} and references inside) introduce and study numerically a general model of dispersive equations under the form
\[
\partial_t u+\alpha u\partial_x u+\beta H\partial_{xx}u+\delta\mathcal{D}[u]=0,
\]
where $\mathcal{D}(u)$ is a Fourier multiplier. For instance, the operator $\mathcal{D}$ can describe the Rayleigh dissipation $\mathcal{D}u=u$, the Reynolds dissipation $\mathcal{D}u=-\partial_{xx}u$ leading to the Benjamin-Ono-Burgers equation, the Landau damping $\mathcal{D}u=H\partial_x u$, the dissipation of
internal solitary waves over a rough bottom $\mathcal{D}u=|u|u$~\cite{Grimshaw2003}, etc. 

The decay to zero and an asymptotic profile of solutions as $t\to+\infty$ can be found in Dix~\cite{Dix1991} for Reynolds dissipation, and Bona, Luo~\cite{BonaLuo2011} for generalized Benjamin-Ono-Burgers equations (see also~\cite{MatsunoShchesnovichKamchatnovKraenkel2007} for dispersive shockwaves when the damping is small). For a fractional dissipation $\mathcal{D}u=|D|^qu$, $q>0$, the asymptotic behavior and the existence of global attractors is investigated in Dix~\cite{Dix1992} and Alarcon~\cite{Alarcon1993}.  Guo and Huo~\cite{GuoHuo2006} prove the existence of a global attractor in $L^2(\R)$, compact in $H^3(\R)$, for generalized weakly-damped KdV-Benjamin-Ono equations, under the form
\[
\partial_t u+\alpha \hilbert \partial_{xx}u+\beta\partial_{xxx}u+\mu\partial_xu+\lambda u+u\partial_xu=f.
\]
The existence of a global attractor in $H^1(\R)$ for the weakly-damped Benjamin-Ono equation had already been proven in~\cite{BolingWu1995}.


We also mention the works of Burq, Raugel and Schlag~\cite{BurqRaugelSchlag2015},~\cite{BurqRaugelSchlag2018} concerning weakly-damped Klein-Gordon equations.

Note that in equation~\eqref{eq:damped}, the damping term is weaker than the ones mentioned above. Indeed, there are many trajectories which do not go to zero as time goes to infinity, for instance, this is the case for the solutions satisfying condition~\eqref{eq:Lasalle} since the $L^2$ norm of these solutions is preserved. Numerical simulations implemented by Klein~\cite{Klein_numerics2020} tend to suggest that there are many other trajectories which do not go to zero, and that a global attractor may not exist (see related open questions in part~\ref{part:open_questions}).

\subsection{Equation in Birkhoff coordinates}
In the study of the damped Benjamin-Ono equation~\eqref{eq:damped}, our general strategy consists in applying the Birkhoff map $\Phi$ and transform this equation into a system of ODEs on the Birkhoff coordinates
\begin{equation}\label{eq:dot_zeta}
\frac{\d}{\d t}\zeta_n(u(t))
	=i\omega_n(\zeta(u(t)))\zeta_n(u(t))-\alpha\left(\langle u(t)|\cos\rangle \d\zeta_n[u(t)].\cos +\langle u(t)|\sin\rangle \d\zeta_n[u(t)].\sin\right),
\end{equation}
where we recall that $\omega_n(\zeta)=n^2-2\sum_{k\geq 1}\min(k,n)|\zeta_k|^2$ for $n\geq 1$.

Our key result is a simplification of the right-hand side of equation~\eqref{eq:dot_zeta}, in particular the terms $\d\zeta_n[u(t)].\cos$ and $\d\zeta_n[u(t)].\sin$, as a function of the Birkhoff coordinates $\zeta=(\zeta_n)_{n\geq 1}$, and of their actions $\gamma=(|\zeta_n|^2)_{n\geq 1}$.
Before stating the theorem, we introduce some notation for the spaces at stake. If $u\in L^2_{r,0}(\T)$, then the Birkhoff coordinates $\zeta$ belong to $h^\half_+\subset \ell^2_+$ , where~$\ell^p_+$ denotes the space of complex sequences
\[
\ell^p_+:=\{z=(z_n)_{n\geq 1}\in\C^{\N}\mid\|z\|_{\ell^p_+}^p=\sum_{n\geq 1}|z_n|^p<+\infty\}.
\]
In particular, the actions $\gamma$ belong to the space $\ell^1_+$.
\begin{thm}[Structure of the vector field]\label{thm:dzeta.cos}
There exist functions $p_n^*, q_n^*, A_{n,k}^*$ and $B_{n,k}^*$ such that for all $u\in L^2_{r,0}(\T)$, if $\zeta=\Phi(u)$, then 
\begin{equation*}
\d\zeta_n[u].\cos
	=p_n^*(\gamma)\zeta_{n-1}
	+q_n^*(\gamma)\zeta_{n+1}
	+\left(\sum_{k=0}^{+\infty}A_{n,k}^*(\gamma)\overline{\zeta_k}\zeta_{k+1}+B_{n,k}^*(\gamma)\zeta_k\overline{\zeta_{k+1}}\right)\zeta_n,
\end{equation*}
with the convention $\zeta_0=1$ and the notation $\gamma=(|\zeta_n|^2)_{n\geq 1}$.
The terms $p_n^*, q_n^*, A_{n,k}^*$ and $B_{n,k}^*$ are  $\classeC^1$ functions of the actions $\gamma\in\ell^1_+$, they are uniformly bounded and Lipschitz on finite balls: for all $\gamma\in \ell^1_+$ satisfying $\|\gamma\|_{\ell^1_+}\leq R$, for all $n,k\geq 0$,
\begin{equation*}
|p_n^*(\gamma)|+|q_n^*(\gamma)|+|A_{n,k}^*(\gamma)|+|B_{n,k}^*(\gamma)|\leq C(R),
\end{equation*}
moreover, for all $h\in \ell^1_+$,
\begin{equation*}
|\d p_n^*(\gamma).h|+|\d q_n^*(\gamma).h|+|\d A_{n,k}^*(\gamma).h|+|\d B_{n,k}^*(\gamma).h|\leq C(R)\|h\|_{\ell^1_+}.
\end{equation*}

The same holds for $\d\zeta_n[u].\sin$, which also writes under the form
\begin{equation*}
\d\zeta_n[u].\sin
	=p_n^*(\gamma)\zeta_{n-1}
	+q_n^*(\gamma)\zeta_{n+1}
	+\left(\sum_{k=0}^{+\infty}A_{n,k}^*(\gamma)\overline{\zeta_k}\zeta_{k+1}+B_{n,k}^*(\gamma)\zeta_k\overline{\zeta_{k+1}}\right)\zeta_n,
\end{equation*}
with similar estimates for the new terms $p_n^*,q_n^*, A_{n,k}^*$ and $B_{n,k}^*$.
\end{thm}
Precise formulas for $p_n^*$ and $q_n^*$ in the case of $\d\zeta_n[u].\cos$ are given by equalities~\eqref{def:pn} and~\eqref{def:qn} below.

Similarly, as proven in~\cite{GerardKappelerTopalov2020}, Lemma 13, there exist functions $a_k^*$ of the actions, which are uniformly bounded and Lipschitz on finite balls of $\ell^1_+$, such that
\begin{equation}\label{eq:<u|e^ix>}
\langle u|\cos\rangle-i\langle u|\sin\rangle
	=-\sum_{k\geq 0}a_k^*(\gamma)\overline{\zeta_k}\zeta_{k+1}.
\end{equation}

The decomposition of the vector field into blocks of the form
\begin{equation}\label{eq:A}
A(\zeta):=\sum_{k\geq 0}A_{k}^*(\gamma)\overline{\zeta_k}\zeta_{k+1},
\end{equation}
where $A_k^*$ are functions of the actions only, enables us to get bounds and Lipschitz estimates on finite balls, so that one can apply the Cauchy-Lipschitz theorem for ODEs. Thanks to this, we avoid going back and forth from $u$ to~$\zeta$ in the equation, and we do not need to use analicity results on the Birkhoff map (see~\cite{GerardKappelerTopalov2020-2} for results in this direction).

We now explain the main idea in the study of the long time asymptotics. Functions of the actions $\gamma$ are slowly varying, indeed, we will establish in inequality~\eqref{ineq:dot_gamma} below that if $\|u_0\|_{L^2(\T)}\leq R$, then for all $t\geq 0$,
\[
\sum_{n\geq 1}\left|\frac{\d}{\d t}\gamma_n(t)\right|
	\leq C(R)|\langle u(t)|e^{ix}\rangle|,
\]
where the right-hand side is square integrable over time thanks to the Lyapunov functional~\eqref{eq:normL2}. 
However, the Birkhoff coordinates $\zeta$ are highly oscillating over time, because of the phase factor $i\omega_n(\zeta)\zeta_n\approx in^2\zeta_n$ appearing in the formula~\eqref{eq:dot_zeta} of its time derivative. Moreover, interactions of different modes under the form $\overline{\zeta_k}\zeta_{k+1}\zeta_n\overline{\zeta_{n+1}}$, $k\neq n$, are also highly oscillating over time,  again because of the phase factor
\[i(-\omega_k(\zeta)+\omega_{k+1}(\zeta)+\omega_n(\zeta)-\omega_{n+1}(\zeta))\overline{\zeta_k}\zeta_{k+1}\zeta_n\overline{\zeta_{n+1}}
	\approx 2i(k-n)\overline{\zeta_k}\zeta_{k+1}\zeta_n\overline{\zeta_{n+1}}
\]
in the time derivative. By analogy with usual Fourier coefficients, such highly oscillating terms have a small time integral. We retrieve this fact here by performing an integration by parts.

Using these remarks, we start from the observation that $|\langle u|e^{ix}\rangle|^2$ is integrable over time, and expand the square modulus of formula~\eqref{eq:<u|e^ix>}. By discarding the interactions of different modes $\overline{\zeta_k}\zeta_{k+1}\zeta_n\overline{\zeta_{n+1}}$, $k\neq n$, we deduce that 
\[
\int_0^{+\infty}\sum_{n=0}^{+\infty}\gamma_n(t)\gamma_{n+1}(t)\d t<+\infty
\]
(see Proposition~\ref{prop:I(T,T')_n,p} for more details). Successive integration by parts allow us to go even further and prove that for higher and higher exponents s,
\[
\int_0^{+\infty}\sum_{n=0}^{+\infty}n^{2s}\gamma_n(t)\gamma_{n+1}(t)\d t<+\infty.
\]
Those latter integrals enable us to control the time evolution of our candidate higher-order Lyapunov functionals.

\subsection{Open questions}\label{part:open_questions}

Let us now mention some open questions about equation~\eqref{eq:damped}.

First, thanks to the Lyapunov functional~\eqref{eq:normL2} and Theorem~\ref{thm:LaSalle}, one knows that the map
\[u_0\in L^2_{r,0}(\T)\mapsto \lim_{t\to+\infty}\|u(t)\|_{L^2}^2\in\R_+\] is well-defined, upper semi-continuous, and that for every weak limit point $u_\infty$ of the sequence $(u(t))_{t\geq 0}$ as $t\to+\infty$, one has $\lim_{t\to+\infty}\|u(t)\|_{L^2}^2=\|u_{\infty}\|_{L^2}^2$. It would be interesting to have more information about this map, for instance, one could ask whether it is possible to have $u_\infty=0$ for some nonzero initial data $u_0$. According to numerical simulations implemented by Klein~\cite{Klein_numerics2020}, there exist families of initial data which seem not to decay to zero as time goes to infinity. For instance, this seems to be the case for a family of potentials with only one nonzero gap $\gamma_1$ at index $1$, given for some $r\in[0,1[$ by (see~\cite{GerardKappeler2019}, Appendix B)
\[
u_0(x)=\frac{1-r^2}{1-2\cos(x)r+r^2}-1.
\]

Finally, one knows from Theorem~\ref{thm:LaSalle}, point 2, that given an initial data $u_0$, all weak limit points $u_\infty$ have the same actions $(\gamma_n^{\infty})_n$, so that these limit points belong to the same Benjamin-Ono torus. One could ask if for some initial data $u_0$, it is possible to provide a precise description of the limiting torus. Again, the numerical simulations from Klein~\cite{Klein_numerics2020} suggest that even for potentials with only one nonzero gap $\gamma_1$ at index $1$, the weak limits could have more than one gap. This is linked to the structure of the equation in Theorem~\ref{thm:dzeta.cos}, since the terms $p_n^*$ and $q_n^*$ given by formulas~\eqref{def:pn} and~\eqref{def:qn} are nonzero and lead to the emergence of new modes in the time evolution. As a consequence, finite gap manifolds are not invariant by the flow of equation~\eqref{eq:damped}, whereas as a comparison, there exist finite dimensional invariant manifolds for equation~\eqref{eq:damped_szego} (see~\cite{GerardGrellier2019}).

\subsection{Plan of the paper}
First, in section~\ref{section:structure}, we write the damped Benjamin-Ono equation~\eqref{eq:damped} in Birkhoff coordinates as a system of ODEs, in particular, we prove Theorem~\ref{thm:dzeta.cos} about the decomposition of the terms $\d\zeta_n[u].\cos$ and $\d\zeta_n[u].\sin$.

Then, in section~\ref{section:flow_map}, we establish the global well-posedness of the damped Benjamin-Ono equation~\eqref{eq:damped} and the weak sequential continuity of the solution map (Theorem~\ref{thm:GWP}).

Section~\ref{section:limit_points} is devoted to the characterization of the weak limit points for the trajectories according to the LaSalle principle, and the fact that the convergence is actually strong in~$L^2_{r,0}(\T)$ (Theorem~\ref{thm:LaSalle}).

Finally, we bound the higher-order Sobolev norms in section~\ref{section:higher_Hs} by studying corresponding Lyapunov functionals in Birkhoff coordinates (Theorem~\ref{thm:higher_Hs}).

\paragraph{Acknowledgments} The author would like to thank her PhD advisor Patrick Gérard who introduced her to this problem and provided generous advice and encouragement. She also warmly thanks Christian Klein for valuable discussions and numerical simulations.

\section{Structure of the equation}\label{section:structure}

Let us consider the flow map for the damped Benjamin-Ono equation~\eqref{eq:damped} in Birkhoff coordinates. Fix $n\geq 1$, then the time derivative of the $n$-th Birkhoff coordinate $\zeta_n$ is
\begin{equation*}
\frac{\d}{\d t}\zeta_n(u(t))
	=\d\zeta_n[u(t)].\partial_tu(t).
\end{equation*}
Using the partial differential equation~\eqref{eq:damped} satisfied by $u$, this time derivative  becomes~\eqref{eq:dot_zeta}
\begin{equation*}
\frac{\d}{\d t}\zeta_n(u(t))
	=i\omega_n(\zeta(u(t)))\zeta_n(u(t))-\alpha\left(\langle u(t)|\cos\rangle \d\zeta_n[u(t)].\cos +\langle u(t)|\sin\rangle \d\zeta_n[u(t)].\sin\right),
\end{equation*}
with $\omega_n(\zeta)=n^2-2\sum_{k\geq 1}\min(k,n)|\zeta_k|^2$.

In this section, we simplify the different terms of the right-hand side of this equation, then study their Lipschitz properties.
Part~\ref{part:def_spectral} is devoted to the study of the terms which only depend on the actions $\gamma_n=|\zeta_n|^2$, $n\geq 1$, for which we establish bounds and Lipschitz properties on balls of finite radius in~$\ell^1_+$. In part~\ref{part:oscillating_terms}, we observe that blocks under the form~\eqref{eq:A}
\begin{equation*}
A(\zeta)=\sum_{k\geq 0}A_{k}^*(\gamma)\overline{\zeta_k}\zeta_{k+1},
\end{equation*}
where $A_k^*$ is a function of the actions $\gamma$ only, are also bounded and Lipschitz on finite balls as soon as the $A_k^*$ are. This is for instance the case for $\langle u|\cos\rangle-i\langle u|\sin\rangle$, which can be expressed under this form. In part~\ref{part:differential_birkhoff}, we establish a formula for the differential~$\d\zeta_n[u].h$ of the $n$-th Birkhoff coordinate at $u\in L^2_{r,0}(\T)$ applied to $h\in L^2_{r,0}(\T)$. Finally, in part~\ref{part:dzeta.cos}, we simplify this formula when $h=\cos,\sin$ and prove Theorem~\ref{thm:dzeta.cos}. Since $\d\zeta_n[u].\cos$ and $\d\zeta_n[u].\sin$ have an expression based on blocks of the above type~\eqref{eq:A}, we deduce bounds and Lipschitz properties on finite balls for these terms.

\subsection{Functions of the actions}\label{part:def_spectral}

In this part, we consider some spectral parameters which will appear all along our study of the damped Benjamin-Ono equation~\eqref{eq:damped}. These parameters are functions of the actions,
defined for $n\geq 1$ as
\(
\gamma_n(u)=|\zeta_n(u)|^2.
\)

First, let us recall some links between the Lax operator for the Benjamin-Ono equation (see~\cite{Nakamura1979},~\cite{BockKruskal1979}) and the actions. With the notation from~\cite{GerardKappeler2019}, Appendix A, the Lax operator $L_u$ writes
\(
L_u=D-T_u,
\)
where $ D=-i\partial_x$, and $T_u$ is the Toeplitz operator
\[
T_u:h\in L^2_+(\T)\mapsto \Pi(uh)\in L^2_+(\T).
\]
defined from the Szeg\H{o} projector $\Pi:L^2(\T)\to L^2_+(\T)$, where
\[
L^2_+(\T)=\{h\in L^2(\T)\mid \forall n<0,\quad \widehat{h}(n)=0\}.
\]

The eigenvalues $(\lambda_n)_{n\geq 0}$ of $L_u$ satisfy (see identity (3.13) in~\cite{GerardKappeler2019})
\begin{equation*}
\lambda_n(u)=n-\sum_{k=n+1}^{+\infty}\gamma_k(u), \quad n\geq 0.
\end{equation*}
Therefore, for all $n,p\in\N$, we have the inequalities
\[
|p-n|\leq |\lambda_p-\lambda_n|\leq |p-n|+\sum_{k\geq 1}\gamma_k.
\]

We now define the spectral parameters $\kappa_n$ and $\mu_n$ as in Corollary~3.4 and formula~(4.9) from~\cite{GerardKappeler2019}.
\begin{mydef}\label{def:kappa}
Let $u\in L^2_{r,0}(\T)$. We define
\[
\kappa_0(u)=\prod_{p=1}^{+\infty}\left(1-\frac{\gamma_p}{\lambda_p-\lambda_0}\right)
\]
and for $n\geq 1$,
\[
\kappa_n(u)=\frac{1}{\lambda_n-\lambda_0}\prod_{\substack{p=1 \\p\neq n}}^{+\infty}\left(1-\frac{\gamma_p}{\lambda_p-\lambda_n}\right).
\]
For $n\geq 1$, we also set
\begin{equation*}
\mu_n(u)
	=\left(1-\frac{\gamma_n}{\lambda_n-\lambda_0}\right)
	\prod_{\substack{p=1 \\p\neq n}}^{+\infty}
	\frac{\left(1-\frac{\gamma_p}{\lambda_p-\lambda_n}\right)}
	{\left(1-\frac{\gamma_p}{\lambda_p-\lambda_{n-1}-1}\right)}.
\end{equation*}
\end{mydef}

The parameters $\kappa_n$ and $\mu_n$ are bounded above and below.

\begin{lem}\label{lem:kappa}
For all $R>0$, there exists $C(R)>0$ such that for all $u\in L^2_{r,0}(\T)$ satisfying $\|\gamma(u)\|_{\ell^1_+}\leq R$ the following holds: for all $n\in\N$,
\[
\frac{1}{C(R)(n+1)}\leq \kappa_n(u)\leq \frac{C(R)}{n+1},
\]
moreover, for all $n\geq 1$,
\[
\frac{1}{C(R)}
	\leq \mu_n(u)
	\leq C(R).
\]
\end{lem}
\begin{proof}
$\bullet $ As a first step, we prove that for all $n\in\N$,
\[
\left|\sum_{\substack{p=1 \\p\neq n}}^{+\infty}\ln\left(1-\frac{\gamma_p}{\lambda_p-\lambda_n}\right)\right|
	\leq C(R).
\]
Let us first show that the sum of logarithms has a lower bound. For $p<n$, since $\lambda_p<\lambda_n$, we have $-\ln(1-\frac{\gamma_p}{\lambda_p-\lambda_n})\leq 0$. Consequently, we discard the negative terms in the following sum:
\begin{align*}
-\sum_{p\neq n}\ln\left(1-\frac{\gamma_p}{\lambda_p-\lambda_n}\right)
	\leq -\sum_{p\geq n+1}\ln\left(1-\frac{\gamma_p}{\lambda_p-\lambda_n}\right).
\end{align*}
For $p\geq n+1$, since by assumption $\gamma_p\in[0,R]$,
\begin{align*}
0
	\leq\frac{\gamma_p}{\lambda_p-\lambda_n}
	\leq \frac{\gamma_p}{1+\gamma_p}
	\leq \frac{R}{1+R},
\end{align*}
therefore one can use that $\ln$ is $C(R)$-Lipschitz on $[\frac{1}{1+R},1]$ to get
\begin{align*}
-\sum_{p\geq n+1}\ln\left(1-\frac{\gamma_p}{\lambda_p-\lambda_n}\right)
	&\leq C(R)\sum_{p\geq n+1}\frac{\gamma_p}{\lambda_p-\lambda_n}\\
	&\leq C(R)R.
\end{align*}

Similarly, we find an upper bound for the sum of logarithms. We first discard the negative terms in the sum
\begin{align*}
\sum_{p\neq n}\ln\left(1-\frac{\gamma_p}{\lambda_p-\lambda_n}\right)
	\leq \sum_{p\leq n-1}\ln\left(1+\frac{\gamma_p}{\lambda_n-\lambda_p}\right).
\end{align*}
Then, since $0\leq \frac{\gamma_p}{\lambda_n-\lambda_p}\leq R$ for $p\leq n-1$, and since $\ln$ is $C(R)$-Lipschitz on $[1,R+1]$, we get
\begin{align*}
\sum_{p\leq n-1}\ln\left(1+\frac{\gamma_p}{\lambda_n-\lambda_p}\right)
	&\leq C(R)\sum_{p\leq n-1}\frac{\gamma_p}{\lambda_n-\lambda_p}\\
	&\leq C(R)R.
\end{align*}

$\bullet$ Now, we prove the inequalities for $\kappa_n$, $n\in\N$. 
For $n\geq 1$, we have
\begin{align*}
|\ln(\lambda_n-\lambda_0)+\ln(\kappa_n)|
	&=\left|\sum_{p\neq n}\ln\left(1-\frac{\gamma_p}{\lambda_p-\lambda_n}\right)\right|\\
	&\leq C(R),
\end{align*}
where $|\lambda_n-\lambda_0-n|\leq C(R)$.
Similarly,
\begin{align*}
|\ln(\kappa_0)|
	&=\left|\sum_{p\geq 1}\ln\left(1-\frac{\gamma_p}{\lambda_p-\lambda_0}\right)\right|\\
	&\leq C(R).
\end{align*}
The proof is the same for $\mu_n$, $n\geq 1$.
\end{proof}

Since the formulas for $\kappa_n$ and $\mu_n$ only depend on the sequence of actions $\gamma(u)=(|\zeta_n(u)|^2)_n$, one can also consider these parameters as functions of the variable $\gamma\in\ell^1_+$ only. In order to avoid confusion, we will denote $\kappa^*_n(\gamma)$ and $\mu^*_n(\gamma)$ in this case. For all $u\in L^2_{r,0}(\T)$, we therefore have the link
\[
\kappa_n^*(\gamma(u))=\kappa_n(u)
\quad
\text{and}
\quad
\mu_n^*(\gamma(u))=\mu_n(u).
\]

We now estimate the differential of the spectral parameters $\kappa_n^*$ and $\mu_n^*$ with respect to the action variables. 

\begin{lem}\label{lem:dkappa}
Let $R>0$ Then there exists $C(R)>0$ such that for all $\gamma\in\ell^1_+$ satisfying $\|\gamma\|_{\ell^1_+}\leq R$ and for all $h\in\ell^1_+$,
\begin{equation*}
|\d\ln(\kappa_n^*)[\gamma].h|\leq C(R)\|h\|_{\ell^1_+},
	\quad n\geq 0
\end{equation*}
and
\begin{equation*}
|\d\ln(\mu_n^*)[\gamma].h|\leq C(R)\|h\|_{\ell^1_+},
	\quad n\geq 1.
\end{equation*}
\end{lem}

\begin{proof}
For instance in the case $n\geq 1$, the differential of $\ln(\kappa_n^*)$ at $\gamma$ is given by
\begin{multline*}
\d\ln(\kappa_n^*)[\gamma].h
	=-\frac{\d\lambda_n^*[\gamma].h-\d\lambda_0^*[\gamma].h}{\lambda_n-\lambda_0}\\
	+\sum_{p\neq n}\frac{1}{1-\frac{\gamma_p}{\lambda_p-\lambda_n}}\left(\frac{\d\lambda_p^*[\gamma].h-\d\lambda_n^*[\gamma].h}{(\lambda_p-\lambda_n)^2}\gamma_p-\frac{h_p}{\lambda_p-\lambda_n}\right),
\end{multline*}
where $h=(h_p)_{p\geq 1}\in\ell^1_+$. We write $\d\lambda_n^*$ in order to specify that we differentiate the eigenvalue $\lambda_n=n-\sum_{k\geq n+1}\gamma_k$ as a function of $\gamma$. In particular, for all $\gamma,h\in\ell^1_+$, we have
\[
|\d\lambda_n^*[\gamma].h|\leq \|h\|_{\ell^1_+}.
\]
The conclusion follows. We proceed similarly for  $\ln(\kappa_0^*)$ and for $\ln(\mu_n^*)$, $n\geq 1$.
\end{proof}

\subsection{Lipschitz properties of blocks}\label{part:oscillating_terms}

In this part, we start from bounds and Lipschitz properties on finite balls for general blocks of the form~\eqref{eq:A}
\begin{equation*}
A(\zeta)=\sum_{k\geq 0}A_{k}^*(\gamma)\overline{\zeta_k}\zeta_{k+1}.
\end{equation*}
We then prove that the terms $\langle u|\cos\rangle$ and $\langle u|\sin\rangle$ can be written under this form, then deduce bounds and Lipschitz estimates on finite balls.

\begin{prop}\label{prop:A(zeta)}
Let $A_{k}^*$, $k\geq 0$, be uniformly bounded and Lipschitz functions on finite balls of~$\ell^1_+$: for all $R>0$, there exists $C_0(R)>0$ such that for all $\gamma^{(1)},\gamma^{(2)}\in\ell^1_+$  satisfying $\|\gamma^{(1)}\|_{\ell^1_+}\leq R$ and $\|\gamma^{(2)}\|_{\ell^1_+}\leq R$, then for all $k\geq 0$,
\[
|A_k^*(\gamma^{(1)})|+|A_k^*(\gamma^{(2)})|\leq C_0(R)
\]
and
\[
|A_k^*(\gamma^{(1)})-A_k^*(\gamma^{(2)})|
	\leq C_0(R)\|\gamma^{(1)}-\gamma^{(2)}\|_{\ell^1_+}.
\]
For $\zeta\in\ell^2_+$  we write $\gamma=(|\zeta_n|^2)_{n\geq 1}$ and adopt the convention $\zeta_0=1$. Let us consider
\[
A(\zeta):=\sum_{k\geq 0}A_{k}^*(\gamma)\overline{\zeta_k}\zeta_{k+1}.
\]
Then $A$ is a bounded Lipschitz function of $\zeta$ on balls of finite radius in $\ell^{2}_+$: there exists $C(R)>0$ such that for all $\zeta^{(1)}, \zeta^{(2)}\in \ell^2_+$ satisfying $\|\zeta^{(1)}\|_{\ell^2_+}^2\leq R$ and $\|\zeta^{(2)}\|_{\ell^2_+}^2\leq R$, there holds
\[
|A(\zeta^{(1)})|+|A(\zeta^{(2)})|\leq C(R)
\]
and
\[
|A(\zeta^{(1)})-A(\zeta^{(2)})|
	\leq C(R)\|\zeta^{(1)}-\zeta^{(2)}\|_{\ell^2_+}.
\]
\end{prop}
The proof is quite direct and omitted here.
%

We now establish a formula of the above form~\eqref{eq:A} for $\langle u|\cos\rangle$ and $\langle u|\sin\rangle$.

Let $(f_p)_{p\geq 0}$ be the orthonormal basis of eigenfunctions for the Lax operator $L_u$ associated to the Benjamin-Ono equation in $L^2_{r,0}(\T)$ uniquely determined by the additional conditions $\langle \un|f_0\rangle>0$ and $\langle f_{p}|Sf_{p-1}\rangle >0$ for $p\geq 1$ (see Definition 2.8 in~\cite{GerardKappeler2019}). The formula for $\langle u|\cos\rangle$ and $\langle u|\sin\rangle$ depends on the matrix~$M$ of the adjoint operator $S^*=T_{e^{-ix}}$ of the shift operator $S:h\in L^2_+(\T)\mapsto e^{ix}h\in L^2_+(\T)$.

\begin{mydef}\label{def:M}
For all $n,p\geq 0$, let
\begin{align*}
M_{n,p}
	&=\langle f_p|Sf_n\rangle\\
	&=\begin{cases}
\frac{\sqrt{\mu_{n+1}}}{\kappa_{n+1}}\langle f_p|\un\rangle\overline{\langle f_{n+1}|\un\rangle}\frac{1}{\lambda_p-\lambda_n-1} & \text{ if } \zeta_{n+1}\neq 0
\\
\delta_{p,n+1} & \text{ if } \zeta_{n+1}=0
\end{cases}.
\end{align*}
In other words, since
\(
\zeta_n=\frac{\langle \un|f_n\rangle}{\sqrt{\kappa_n}}, \)
for $n\geq 0$, (see  equality (4.1) in~\cite{GerardKappeler2019}), we have
\begin{align*}
M_{n,p}
	&=\begin{cases}
\sqrt{\mu_{n+1}}\frac{\sqrt{\kappa_p}}{\sqrt{\kappa_{n+1}}}\overline{\zeta_p}\zeta_{n+1}\frac{1}{\lambda_p-\lambda_n-1}
& \text{ if } p\neq n+1
\\
\sqrt{\mu_{n+1}}
& \text{ if } p= n+1
\end{cases}.
\end{align*}
\end{mydef}
%
\begin{lem}\label{lem:<u|e^ix>}
Let $u\in L^2_{r,0}(\T)$. Then
\[
\langle u|e^{ix}\rangle
	=-\sum_{n\geq 0}\overline{\zeta_n}\zeta_{n+1}\sqrt{\frac{\kappa_n}{\kappa_{n+1}}}\sqrt{\mu_{n+1}}
\]
with the convention $\zeta_0=1$.
\end{lem}
This formula for $\langle u|\cos\rangle$ and $\langle u|\sin\rangle$ was already proven in~\cite{GerardKappelerTopalov2020}, Lemma 13, in the special case when $\gamma_n(u)\neq 0$ for all $n$. We justify here how the proof stays valid in the general case.

\begin{proof}
We note that $\langle u|e^{ix}\rangle=\langle \Pi u|e^{ix}\rangle$, then decompose $\Pi u$ and $e^{ix}$ in the orthonormal basis $(f_p)_{p\geq 0}$
\begin{equation*}
\langle u|e^{ix}\rangle
	=\sum_{p\geq 0}\langle \Pi u|f_p\rangle
\langle f_p|e^{ix}\rangle.
\end{equation*}
We now remark that $\langle f_p|e^{ix}\rangle=\langle S^*f_p|\un\rangle$, and decompose $S^*f_p$ and $\un$ along the orthonormal basis $(f_n)_{n\geq 0}$
\begin{align*}
\langle f_p|e^{ix}\rangle
	&=\sum_{n\geq 0}\langle S^*f_p|f_n\rangle\langle f_n|\un\rangle\\
	&=\sum_{n\geq 0}M_{n,p}\langle f_n|\un\rangle\\
	&=\sum_{\substack{n\geq 0 \\\gamma_{n+1}\neq 0}}\frac{\sqrt{\mu_{n+1}}}{\kappa_{n+1}}\langle f_p|\un\rangle\overline{\langle f_{n+1}|\un\rangle}\frac{1}{\lambda_p-\lambda_n-1}\langle f_n|\un\rangle.
\end{align*}
From the identity $\langle \Pi u|f_p\rangle=-\lambda_p\langle \un|f_p\rangle$, we get 
\begin{align*}
\langle u|e^{ix}\rangle
	&=\sum_{p\geq 0}-\lambda_p\langle \un|f_p\rangle\sum_{\substack{n\geq 0 \\\gamma_{n+1}\neq 0}}\frac{\sqrt{\mu_{n+1}}}{\kappa_{n+1}}\langle f_p|\un\rangle\overline{\langle f_{n+1}|\un\rangle}\frac{1}{\lambda_p-\lambda_n-1}\langle f_n|\un\rangle\\
	&=-\sum_{\substack{n\geq 0 \\\gamma_{n+1}\neq 0}}\frac{\sqrt{\mu_{n+1}}}{\kappa_{n+1}}\langle f_n|\un\rangle\overline{\langle f_{n+1}|\un\rangle}\sum_{p\geq 0}\frac{\lambda_p|\langle f_p|\un\rangle|^2}{\lambda_p-\lambda_n-1}.
\end{align*}
Recall the definition of the generating function $\hamilton_\lambda$ for $\lambda\in\C$ (formula (3.2) in~\cite{GerardKappeler2019})
\[
\hamilton_\lambda(u)=\sum_{p\geq0}\frac{|\langle \un|f_p\rangle|^2}{\lambda_p+\lambda},
\]
then if $\gamma_{n+1}\neq 0$, we have the identity
\begin{align*}
\sum_{p\geq 0}\frac{\lambda_p|\langle f_p|\un\rangle|^2}{\lambda_p-\lambda_n-1}
	&=(\lambda_n+1)\sum_{p\geq 0}\frac{|\langle f_p|\un\rangle|^2}{\lambda_p-\lambda_n-1}+\sum_{p\geq 0}|\langle f_p|\un\rangle|^2\\
	&=(\lambda_n+1)\hamilton_{-\lambda_n-1}(u)+1.
\end{align*}
Moreover, the formula from~\cite{GerardKappeler2019} (see Proposition 3.1 (i))
\[
\hamilton_{-\lambda_n-1}(u)
	=-\frac{1}{\lambda_n+1-\lambda_0}\prod_{p=1}^{+\infty}\left(1-\frac{\gamma_p}{\lambda_p-\lambda_n-1}\right)
\]
implies that $\hamilton_{-\lambda_n-1}(u)=0$ because of the term appearing in the product when $p=n+1$.
Using that $\zeta_n=\frac{\langle \un|f_n\rangle}{\sqrt{\kappa_n}}$, we conclude the proof of the lemma.
\end{proof}

Let us write
\[
\langle u|e^{ix}\rangle=\sum_{n\geq 0}-a_n^*(\gamma)\overline{\zeta_n}\zeta_{n+1},
\]
where
\begin{equation}\label{def:a_n}
a_n^*(\gamma)=\sqrt{\mu_{n+1}^*(\gamma)}\frac{\sqrt{\kappa_n^*(\gamma)}}{\sqrt{\kappa_{n+1}^*(\gamma)}}.
\end{equation}
Thanks to part~\ref{part:def_spectral}, we know that if  $\|\gamma\|_{\ell^1_+}\leq R$, then there exists $C(R)>0$ such that for all $n\geq 0$, 
\begin{equation}\label{ineq:a_n_bound}
\frac{1}{C(R)}\leq a_n^*(\gamma)\leq C(R)
\end{equation}
and for $h\in\ell^1_+$,
\begin{equation}\label{ineq:a_n_lipschitz}
|\d a_n^*[\gamma].h|\leq C(R)\|h\|_{\ell^1_+}.
\end{equation}
Proposition~\ref{prop:A(zeta)} implies that the maps $\zeta\in h^{\half}_+\mapsto \langle \Phi^{-1}(\zeta)|\cos\rangle$ and $\zeta\in h^{\half}_+\mapsto \langle \Phi^{-1}(\zeta)|\sin\rangle$ are bounded and Lipschitz on finite balls of $\ell^2_+$.

\begin{cor}\label{cor:<u|e^ix>_lip}
Fix $R>0$. Then there exists $C(R)>0$ such that for all  $\zeta^{(1)}, \zeta^{(2)}\in h^\half_+$ satisfying $\|\zeta^{(1)}\|_{\ell^2_+}\leq R$ and $\|\zeta^{(2)}\|_{\ell^2_+}\leq R$, writing $u^{(1)}=\Phi^{-1}(\zeta^{(1)})$ and $u^{(2)}=\Phi^{-1}(\zeta^{(2)})$, we have
\[
|\langle u^{(1)}|\cos\rangle|+|\langle u^{(1)}|\sin\rangle|+|\langle u^{(2)}|\cos\rangle|+|\langle u^{(2)}|\sin\rangle|\leq C(R)
\]
and
\[
|\langle u^{(1)}|\cos\rangle-\langle u^{(2)}|\cos\rangle|
	+|\langle u^{(1)}|\sin\rangle-\langle u^{(2)}|\sin\rangle|
	\leq C(R)\|\zeta^{(1)}-\zeta^{(2)}\|_{\ell^2_+}.
\]
\end{cor}

\subsection{Formula for the differential of the Birkhoff map}\label{part:differential_birkhoff}

In this part, we establish a formula for the differential of the normalized Birkhoff coordinates $\zeta_n$ (with $\zeta_0=1$), which are linked to the unnormalized Birkhoff coordinates $\langle \un|f_n\rangle$ through the formulas (see equality (4.1) in~\cite{GerardKappeler2019})
\[
\zeta_n=\frac{\langle \un|f_n\rangle}{\sqrt{\kappa_n}}, \quad n\geq 0.
\]
These coordinates are known to admit a differential thanks to Lemma~3.4 in~\cite{GerardKappeler2019}. The differential of $\zeta_n$ at $u\in L^2_{r,0}(\T)$ therefore satisfies: for all $h\in L^2_{r,0}(\T)$,
\begin{equation}\label{eq:dzeta}
\d\zeta_n[u].h
	=\d\langle \un|f_n\rangle[u].h\frac{1}{\sqrt{\kappa_n}}
	-\frac{1}{2}\zeta_n\d\ln(\kappa_n)[u].h.
\end{equation}
We prove the following decomposition of the differential of the unnormalized Birkhoff map.
\begin{prop}[Differential of the unnormalized Birkhoff map]\label{prop:decomposition_unnormalized_birkhoff}
Let $u,h\in L^2_{r,0}(\T)$, then for all $n\geq 0$, $\d\langle \un|f_n\rangle[u].h$ writes
\[
\d\langle \un|f_n\rangle[u].h
	=\delta_\parallel\langle \un|f_n\rangle
	+\delta_\perp\langle \un|f_n\rangle,
\]
where
\[
\delta_\parallel\langle \un|f_n\rangle
	=-i\Im (\langle\widetilde{\xi_n}|h\rangle) \langle\un|f_n\rangle,
\]
\[
\widetilde{\xi_n}=\frac{1}{\langle f_0|\un\rangle}\sum_{p=1}^{+\infty}\frac{\langle \un|f_p\rangle}{\lambda_p-\lambda_0}\overline{f_0}f_p
	-\sum_{k=1}^n \frac{1}{\langle f_k|Sf_{k-1}\rangle}\psi_k,
\]
\[
\psi_n
	=\sum_{\substack{p\geq 0\\p\neq n}}\frac{\langle f_p|Sf_{n-1}\rangle}{\lambda_p-\lambda_n}f_n\overline{f_p}
	-\sum_{\substack{p\geq 0\\p\neq n-1}}\frac{\langle Sf_p|f_{n}\rangle}{\lambda_p-\lambda_{n-1}}f_{n-1}\overline{f_p}
\]
and
\[
\delta_\perp\langle \un|f_n\rangle
	=\sum_{\substack{p\geq 0\\p\neq n}}\frac{\langle f_p\overline{f_n}|h \rangle}{\lambda_p-\lambda_n}\langle\un|f_p\rangle.
\]
\end{prop}

\begin{proof}
For $n\geq 0$, since $\d\langle \un|f_n\rangle[u].h=\langle \un |\d f_n[u].h\rangle$, we decompose the two terms inside the brackets along the orthonormal basis $(f_p)_{p\geq 0}$ and get
\[
\d\langle \un|f_n\rangle[u].h
	=\sum_{p\geq 0}\langle\un|f_p \rangle\langle f_p |\d f_n[u].h\rangle.
\]
We now compute $\langle\d f_n[u].h|f_p\rangle$ for $n,p\geq 0$.

For $n\geq 0$, recall that $f_n$ is the $L^2$-normalized eigenfunction of $L_u=D-T_u$ associated to the eigenvalue~$\lambda_n$, and that the family $(f_n)_n$ is uniquely determined by the additional conditions $\langle \un|f_0\rangle>0$ and $\langle f_{n}|Sf_{n-1}\rangle >0$ for $n\geq 1$. Therefore, we get by differentiating
\[
(L_u-\lambda_n)\d f_n[u].h=\d\lambda_n[u].hf_n+T_hf_n.
\]

For $p\neq n$, $f_p$ is orthogonal to $f_n$ and therefore
\[
\langle \d f_n[u].h | f_p\rangle=\frac{\langle f_n\overline{f_p}|h\rangle}{\lambda_p-\lambda_n}.
\]
This gives the formula for $\delta_\perp\langle \un|f_n\rangle$.

For $p=n$, we see that $\langle \d f_n[u].h | f_n\rangle$ is purely imaginary because $\|f_n\|_{L^2(\T)}^2=1$. Moreover, the conditions $\langle f_0|\un\rangle>0$ and $\langle f_n|Sf_{n-1}\rangle>0$ for $n\geq 1$ imply that
\[
\Im(\langle \d f_0[u].h|\un\rangle)=0
\]
and
\[
\Im(\langle \d f_n[u].h|Sf_{n-1}\rangle+\langle f_n|S\d f_{n-1}[u].h\rangle)=0
	,\quad n\geq 1.
\]
Decomposing $\un$ in the orthonormal basis $(f_p)_{p\geq 0}$, we get for $n=0$
\begin{align*}
0
	&=\Im\left(\langle \d f_0[u].h|f_0\rangle\langle f_0|\un\rangle+\sum_{p=1}^{+\infty}\langle \d f_0[u].h|f_p\rangle\langle f_p|\un\rangle\right)\\
	&=\Im\left(\langle \d f_0[u].h|f_0\rangle\langle f_0|\un\rangle+\sum_{p=1}^{+\infty}\frac{\langle f_0\overline{f_p}|h\rangle}{\lambda_p-\lambda_0}\langle f_p|\un\rangle\right),
\end{align*}
or
\begin{align*}
\langle \d f_0[u].h|f_0\rangle
	&=-i\Im\left(\frac{1}{\langle f_0|\un\rangle}\sum_{p=1}^{+\infty}\frac{\langle f_0\overline{f_p}|h\rangle}{\lambda_p-\lambda_0}\langle f_p|\un\rangle\right)\\
	&=i\Im(\langle \widetilde{\xi_0}|h\rangle).
\end{align*}
Similarly, for $n\geq 1$, decomposing $Sf_{n-1}$ and $S^*f_{n}$ in the orthonormal basis $(f_p)_{p\geq 0}$, we have
\begin{align*}
0
	&=\Im\left(\sum_{p\geq0}\langle \d f_n[u].h|f_p\rangle\langle f_p|Sf_{n-1}\rangle+\langle S^*f_n|f_p\rangle\langle f_p|\d f_{n-1}[u].h\rangle\right),
\end{align*}
and by isolating the index $p=n$ in the first term and $p=n-1$ in the second term, we get the formula
\begin{multline*}
\langle f_n|Sf_{n-1}\rangle \Im(\langle \d f_n[u].h|f_n\rangle-\langle \d f_{n-1}[u].h|f_{n-1}\rangle)\\
	=-\Im\left(\sum_{p\neq n}\langle \d f_n[u].h|f_p\rangle\langle f_p|Sf_{n-1}\rangle-\sum_{p\neq n-1}\overline{\langle f_n|Sf_p\rangle}\langle \d f_{n-1}[u].h|f_p\rangle\right),
\end{multline*}
leading to
\begin{multline*}
\langle f_n|Sf_{n-1}\rangle \Im(\langle \d f_n[u].h|f_n\rangle-\langle \d f_{n-1}[u].h|f_{n-1}\rangle)\\
	=-\Im\left(\sum_{p\neq n}\frac{\langle f_n\overline{f_p}|h\rangle}{\lambda_p-\lambda_n}\langle f_p|Sf_{n-1}\rangle-\sum_{p\neq n-1}\langle Sf_p|f_n\rangle\frac{\langle f_{n-1}\overline{f_p}|h\rangle}{\lambda_p-\lambda_{n-1}}\right).
\end{multline*}
We retrieve the expression of $\delta_\parallel\langle \un|f_n\rangle$.
\end{proof}

\subsection{Decomposition for \texorpdfstring{$\d\zeta_n[u].\cos$}{dzeta.cos} and \texorpdfstring{$\d\zeta_n[u].\sin$}{dzeta.sin} }\label{part:dzeta.cos}

In this part, we establish Theorem~\ref{thm:dzeta.cos}: we decompose $\d\zeta_n[u].\cos$ and $\d\zeta_n[u].\sin$ using blocks of the form~\eqref{eq:A} $\sum_{k\geq 0}A_k^*(\gamma)\overline{\zeta_k}\zeta_{k+1}$, where $A_k^*$ are functions of the actions only.

Let us now give the organization of this part. Recall that since $\zeta_n=\frac{\langle \un|f_n\rangle}{\sqrt{\kappa_n}}$, we have equality~\eqref{eq:dzeta}
\begin{align*}
\d\zeta_n[u].h
	&=\d\langle \un|f_n\rangle[u].h\frac{1}{\sqrt{\kappa_n}}
	-\frac{1}{2}\zeta_n\d\ln(\kappa_n)[u].h,
\end{align*}
where $\d\langle \un|f_n\rangle[u].h$ decomposes from Proposition~\ref{prop:decomposition_unnormalized_birkhoff} as
\[
\d\langle \un|f_n\rangle[u].h
	=\delta_\parallel\langle \un|f_n\rangle
	+\delta_\perp\langle \un|f_n\rangle.
\]
This leads us to study successively each part of the decomposition: $\d\ln(\kappa_n)[u].h$, $\delta_\parallel\langle \un|f_n\rangle$ and~$\delta_\perp\langle \un|f_n\rangle$, in the particular cases $h=\cos$ and $h=\sin$.  Then, in subpart~\ref{part:dzeta_lipschitz}, we use part~\ref{part:oscillating_terms} in order to deduce that $\d\zeta_n[u].\cos$ and $\d\zeta_n[u].\sin$ are bounded and Lipschitz on balls of finite radius.

\subsubsection{The case of \texorpdfstring{$\d\ln(\kappa_n)[u].h$}{d ln(kappa_n)}}

For $\zeta\in h^{\half}_+$, we denote $u=\Phi^{-1}\left(\zeta\right)$ and define
\[
\delta\kappa_n(\zeta):=\d\ln(\kappa_n)[u].\cos-i\d\ln(\kappa_n)[u].\sin,
\quad n\in\N.
\]
To avoid confusion, we shall precise that $\ln(\kappa_n)$ is a function of $u\in L^2_{r,0}(\T)$, and that $\d\ln(\kappa_n)[u].\cos$ denotes its differential with respect to the variable $u$ applied to $\cos$. However, $\delta\kappa_n$ is a function of the Birkhoff coordinates $\zeta$, since $u=\Phi^{-1}(\zeta)$ in the definition. 

\begin{lem}\label{lem:dkappa_lip}
There exist functions $A_{n,k}^*\in\classeC^1(\ell^1_+,\R)$, $n,k\in\N$, such that for all $n\geq 0$, for all $\zeta\in h^{\half}_+$,
\[
\delta{\kappa_n}(\zeta)=\sum_{k\geq 0}A_{n,k}^*(\gamma)\overline{\zeta_k} \zeta_{k+1}.
\]
Moreover, for all $R>0$, there exists $C(R)>0$ such that for all $\gamma\in \ell^1_+$ satisfying $\|\gamma\|_{\ell^1_+}\leq R$, for all $n,k\geq 0$,
\[
|A_{n,k}^*(\gamma)|\leq C(R)
\]
and for all $h\in \ell^1_+$,
\[
|\d A_{n,k}^*(\gamma).h|\leq C(R)\|h\|_{\ell^1_+}.
\]
\end{lem}

\begin{proof}
We first compute the differential of the eigenvalues $\lambda_n$, $ n\in\N$, with respect to the variable $u$ and applied to $h=\cos$ and $h=\sin$. From~\cite{GerardKappeler2019}, Corollary~5.3, we know that
\[
\d \lambda_n[u].h
	=-\langle |f_n|^2|h\rangle,
	\quad h\in L^2_{r,0}(\T),
\]
which leads to a formula for the differential of the moments $\gamma_n=\lambda_n-\lambda_{n-1}-1$ for $n\geq 1$
\begin{equation}\label{eq:dgamma_n}
\d\gamma_n[u].h
	=\langle |f_{n-1}|^2-|f_n|^2|h\rangle,
	\quad h\in L^2_{r,0}(\T).
\end{equation}
We now simplify
\begin{align*}
\d\gamma_n[u].\cos-i\d\gamma_n[u].\sin
	&=\langle |f_{n-1}|^2-|f_n|^2 | e^{ix}\rangle\\
	&=M_{n-1,n-1}-M_{n,n}.
\end{align*}
We denote
\[
m_n=M_{n,n}=-a_n^*(\gamma)\overline{\zeta_n}\zeta_{n+1},
\]
where we recall that $a_n^*$ is a function of $\gamma$ defined as
\[
a_n^*=\sqrt{\mu_{n+1}^*}\frac{\sqrt{\kappa_n^*}}{\sqrt{\kappa_{n+1}^*}}.
\]
In the rest of the proof, we drop the star exponent for $a_n^*$, $\lambda_n^*$, $\kappa_n^*$ and $\mu_n^*$ in order to avoid heaviness.
With this notation, we have
$$\d\gamma_n[u].\cos-i\d\gamma_n[u].\sin=m_{n-1}-m_n$$
and
$$\d\lambda_n[u].\cos-i\d\lambda_n[u].\sin=-m_n.$$

We are now ready to study the differential of $\ln(\kappa_n)$. Recall that for $n\geq 1$,
\[
\kappa_n(u)=\frac{1}{\lambda_n-\lambda_0}\prod_{\substack{p=1 \\p\neq n}}^{+\infty}\left(1-\frac{\gamma_p}{\lambda_p-\lambda_n}\right)
\]
and
\[
\kappa_0(u)=\prod_{p=1}^{+\infty}\left(1-\frac{\gamma_p}{\lambda_p-\lambda_0}\right).
\]
We get that for $n\geq 1$,
\begin{align*}
\delta{\kappa_n}
	&=\frac{m_n-m_0}{\lambda_n-\lambda_0}+\sum_{\substack{p\geq 1\\p\neq n}}\frac{1}{1-\frac{\gamma_p}{\lambda_p-\lambda_n}}\left(\frac{m_n-m_p}{(\lambda_p-\lambda_n)^2}\gamma_p-\frac{m_{p-1}-m_p}{\lambda_p-\lambda_n}\right)
\end{align*}
and similarly,
\begin{align*}
\delta{\kappa_0}
	&=\sum_{p\geq 1}\frac{1}{1-\frac{\gamma_p}{\lambda_p-\lambda_0}}\left(\frac{m_0-m_p}{(\lambda_p-\lambda_0)^2}\gamma_p-\frac{m_{p-1}-m_p}{\lambda_p-\lambda_0}\right).
\end{align*}

For $n\geq 1$, one can therefore write $\delta{\kappa_n}(\zeta)=\sum_kA_{n,k}^*(\gamma)\overline{\zeta_k} \zeta_{k+1}$, where for $k\not\in \{0;n\}$,
\[
A_{n,k}^*(\gamma)
	=\frac{1}{1-\frac{\gamma_k}{\lambda_k-\lambda_n}}\left(\frac{a_k}{(\lambda_k-\lambda_n)^2}\gamma_k-\frac{a_k}{\lambda_k-\lambda_n}\right)+\frac{1}{1-\frac{\gamma_{k+1}}{\lambda_{k+1}-\lambda_n}}\frac{a_k}{\lambda_{k+1}-\lambda_n},
\]
for $k=n$,
\[
A_{n,n}^*(\gamma)
	=-\frac{a_n}{\lambda_n-\lambda_0}
	+\left(\sum_{k\neq n}\frac{1}{1-\frac{\gamma_k}{\lambda_k-\lambda_n}}\frac{-a_n}{(\lambda_k-\lambda_n)^2}\gamma_k\right)
	+\frac{1}{1-\frac{\gamma_{n+1}}{\lambda_{n+1}-\lambda_n}}\frac{a_n}{\lambda_{n+1}-\lambda_n},
\]
and for $k=0$,
\[
A_{n,0}^*(\gamma)
	=\frac{a_0}{\lambda_n-\lambda_0}
	+\frac{1}{1-\frac{\gamma_{1}}{\lambda_{1}-\lambda_n}}\frac{a_0}{\lambda_{1}-\lambda_n}.
\]
In the case $n=0$, we have for $k\geq1$,
\[
A_{0,k}^*(\gamma)
	=\frac{1}{1-\frac{\gamma_k}{\lambda_k-\lambda_0}}\left(\frac{a_k}{(\lambda_k-\lambda_0)^2}\gamma_k-\frac{a_k}{\lambda_k-\lambda_0}\right)+\frac{1}{1-\frac{\gamma_{k+1}}{\lambda_{k+1}-\lambda_0}}\frac{a_k}{\lambda_{k+1}-\lambda_0}
\]
and for $k=0$,
\[
A_{0,0}^*(\gamma)
	=\left(\sum_{k\geq 1}\frac{1}{1-\frac{\gamma_k}{\lambda_k-\lambda_0}}\frac{-a_0}{(\lambda_k-\lambda_0)^2}\gamma_k\right)
	+\frac{1}{1-\frac{\gamma_{1}}{\lambda_{1}-\lambda_0}}\frac{a_0}{\lambda_{1}-\lambda_0}.
\]

We have seen in~\eqref{ineq:a_n_bound} and~\eqref{ineq:a_n_lipschitz} that if $\|\gamma\|_{\ell^1_+}\leq R$, then there exists $C(R)>0$ such that for all $n\in\N$,
\[
\frac{1}{C(R)}\leq a_n(\gamma)\leq C(R)
\]
and for all $h\in\ell^1_+$,
\[
|\d a_n(\gamma).h|\leq C(R)\|h\|_{\ell^1_+}.
\]
Moreover, the estimate $|\d\lambda_n^*(\gamma).h|\leq \|h\|_{\ell^1_+}$ holds for the differential of $\lambda_n^*$ with respect to~$\gamma$.
We deduce that for all $n,k\in\N$,
\[|A_{n,k}^*(\gamma)|\leq C(R)\]
and for all $h\in\ell^1_+$,
\[
\left|\d A_{n,k}^*(\gamma).h\right|\leq C(R)\|h\|_{\ell^1_+}.
\]
\end{proof}

\subsubsection{The case of \texorpdfstring{$\delta_\parallel\langle \un|f_n\rangle$}{delta_parallel}}

We now study, for $h=\cos$ and $h=\sin$, the term
\[
\delta_\parallel\langle \un|f_n\rangle
	=-i\Im (\langle\widetilde{\xi_n}|h\rangle) \langle\un|f_n\rangle,
\]
where we recall that
\[
\widetilde{\xi_n}=\frac{1}{\langle f_0|\un\rangle}\sum_{p=1}^{+\infty}\frac{\langle \un|f_p\rangle}{\lambda_p-\lambda_0}\overline{f_0}f_p
	-\sum_{k=1}^n \frac{1}{\langle f_k|Sf_{k-1}\rangle}\psi_k
\]
and
\[
\psi_k
	=\sum_{\substack{p\geq 0\\p\neq k}}\frac{\langle f_p|Sf_{k-1}\rangle}{\lambda_p-\lambda_k}f_k\overline{f_p}
	-\sum_{\substack{p\geq 0\\p\neq k-1}}\frac{\langle Sf_p|f_{k}\rangle}{\lambda_p-\lambda_{k-1}}f_{k-1}\overline{f_p}.
\]

\begin{lem}\label{lem:delta_parallel_lip}
For all $n\geq 0$, denote
\[
c_n^\pm:=\langle\widetilde{\xi_n}|e^{\pm ix}\rangle,
\]
then for $n,k\geq 0$, there exists $A^{\pm}_{n,k}\in\classeC^1(\ell^1_+,\R)$, such that for all $n\geq0$, for all $\zeta\in h^{\half}_+$,
 \[c_n^+(\zeta)=\sum_{k=0}^{+\infty}A^+_{n,k}(\gamma)\overline{\zeta_k}\zeta_{k+1}\quad \text{and} \quad c_n^-(\zeta)=\sum_{k=0}^{+\infty}A^-_{n,k}(\gamma)\zeta_k\overline{\zeta_{k+1}}.\]
Moreover, for all $R>0$, there exists $C(R)>0$ such that for all $\gamma\in\ell^1_+$ satisfying $\|\gamma\|_{\ell^1_+}\leq R$,
\[|A^\pm_{n,k}(\gamma)|\leq C(R)\] and for all $h\in\ell^1_+$,
\[
\left|\d A^\pm_{n,k}(\gamma).h\right|\leq C(R)\|h\|_{\ell^1_+}.
\]
\end{lem}

\begin{proof}
$\bullet$ By definition of $M_{n,p}=\langle f_p|Sf_n\rangle$, we have
\[
c_n^+
	=\frac{1}{\sqrt{\kappa_0}}\sum_{p=1}^{+\infty}\frac{\sqrt{\kappa_p}\zeta_p}{\lambda_p-\lambda_0}M_{0,p}
	-\sum_{k=1}^n \frac{1}{\sqrt{\mu_k}}\langle\psi_k|e^{ix}\rangle
\]
and
\begin{align*}
\langle \psi_k|e^{ix}\rangle
	&=\sum_{p\neq k}\frac{M_{k-1,p}}{\lambda_p-\lambda_k}M_{p,k}-\sum_{p\neq k-1}\frac{\overline{M_{p,k}}}{\lambda_p-\lambda_{k-1}}M_{p,k-1}.
\end{align*}

We replace the terms $M_{n,p}$ by their expressions (see Definition~\ref{def:M}), and make use of colors to emphasize the oscillating terms (of the form $\overline{\zeta_k}\zeta_{k+1}$ or $\zeta_k\overline{\zeta_{k+1}}$). Since $\zeta_0=1$, we have
\begin{align*}
c_n^+
	&=\sqrt{\frac{\mu_1\kappa_1}{\kappa_0}}\frac{{\color{teal} \overline{\zeta_0}\zeta_1}}{1+\gamma_1}+\sqrt{\frac{\mu_{1}}{\kappa_0\kappa_1}}{\color{teal} \overline{\zeta_0}\zeta_1}\sum_{p=2}^{+\infty}\frac{\kappa_p\overline{\zeta_p}\zeta_p}{(\lambda_p-\lambda_0)(\lambda_p-\lambda_0-1)}
	-\sum_{k=1}^n \frac{1}{\sqrt{\mu_k}}\langle\psi_k|e^{ix}\rangle.
\end{align*}
Moreover, by isolating the indices $p=k-1$ and $p=k-2$ in the first and second sum respectively in the formula for $\langle \psi_k|e^{ix}\rangle$, we get
\begin{multline*}
\langle \psi_k|e^{ix}\rangle
	=\frac{\mu_k}{1+\gamma_k}\sqrt{\frac{\kappa_{k-1}}{\kappa_k}}{\color{blue} \overline{\zeta_{k-1}}\zeta_k}\\
	+\overline{\zeta_k}\zeta_k\sum_{p\neq k-1,k}\sqrt{\mu_k\mu_{p+1}}\sqrt{\frac{
\kappa_p}{\kappa_{p+1}}}\frac{{\color{violet}\overline{\zeta_p}\zeta_{p+1}}}{(\lambda_p-\lambda_k)(\lambda_p-\lambda_{k-1}-1)(\lambda_k-\lambda_p-1)}\\
	+\frac{\mu_{k-1}}{(1+\gamma_{k-1})(1+\gamma_k+\gamma_{k-1})}\sqrt{\frac{\kappa_k}{\kappa_{k-1}}}{\color{blue}\zeta_k\overline{\zeta_{k-1}}}\\
	-{\color{blue}\zeta_k\overline{\zeta_{k-1}}}\sum_{p\neq k-2,k-1}\mu_{p+1}\frac{\sqrt{\kappa_k\kappa_{k-1}}}{\kappa_{p+1}(\lambda_p-\lambda_{k-1})}\frac{\zeta_{p+1}\overline{\zeta_{p+1}}}{(\lambda_k-\lambda_p-1)(\lambda_{k-1}-\lambda_p-1)}.
\end{multline*}

One can therefore write $c_n^+=\sum_{k=1}^{+\infty}A^{+}_{n,k-1}\overline{\zeta_{k-1}}\zeta_{k}$, with
\begin{align*}
A^+_{n,k-1}
=&{\color{teal}\un_{k=1}\left(\sqrt{\frac{\mu_1\kappa_1}{\kappa_0}}\frac{1}{1+\gamma_1}
	+\sqrt{\frac{\mu_{1}}{\kappa_0\kappa_1}}\sum_{p=2}^{+\infty}\frac{\kappa_p\gamma_p}{(\lambda_p-\lambda_0)(\lambda_p-\lambda_0-1)}\right)}\\
	&{\color{blue}+\un_{1\leq k\leq n}\Big(\frac{\sqrt{\mu_k}}{1+\gamma_k}\sqrt{\frac{\kappa_{k-1}}{\kappa_k}}}\\
	&{\color{blue}+\frac{\mu_{k-1}}{\sqrt{\mu_k}(1+\gamma_{k-1})(1+\gamma_k+\gamma_{k-1})}\sqrt{\frac{\kappa_k}{\kappa_{k-1}}}}\\
	&{\color{blue}-\sum_{p\neq k-2,k-1}\frac{\mu_{p+1}}{\sqrt{\mu_k}}\frac{\sqrt{\kappa_k\kappa_{k-1}}}{\kappa_{p+1}(\lambda_p-\lambda_{k-1})}\frac{\gamma_{p+1}}{(\lambda_k-\lambda_p-1)(\lambda_{k-1}-\lambda_p-1)}\Big)}\\
	&{\color{violet}+\sum_{\substack{1\leq l\leq n\\k-1\neq l-1,l}}\sqrt{\mu_k}\sqrt{\frac{
\kappa_{k-1}}{\kappa_{k}}}\frac{\gamma_l}{(\lambda_{k-1}-\lambda_l)(\lambda_{k-1}-\lambda_{l-1}-1)(\lambda_l-\lambda_{k-1}-1)}}.
\end{align*}

$\bullet $ We now prove bounds and Lipschitz estimates on finite balls of $\ell^1_+$.

We know that for $p\neq k$, $\frac{1}{|\lambda_p-\lambda_k|}\leq \frac{1}{|p-k|}$, and for $p\neq k+1$, $\frac{1}{|\lambda_p-\lambda_k-1|}\leq 1$. Moreover, $\mu_p\leq C(R)$ and $\kappa_p\leq\frac{C(R)}{p+1}$.

The first term is therefore bounded as
\[
\sqrt{\frac{\mu_1\kappa_1}{\kappa_0}}\frac{1}{1+\gamma_1}
	+\sqrt{\frac{\mu_{1}}{\kappa_0\kappa_1}}\sum_{p=2}^{+\infty}\frac{\kappa_p\gamma_p}{(\lambda_p-\lambda_0)(\lambda_p-\lambda_0-1)}
	\leq C(R)\left(1+\sum_{p=2}^{+\infty}\kappa_p\gamma_p\right)
	\leq C'(R).
\]
Then, for $1\leq k\leq n$, we have
\[
\frac{\sqrt{\mu_k}}{1+\gamma_k}\sqrt{\frac{\kappa_{k-1}}{\kappa_k}}\\
	+\frac{\mu_{k-1}}{\sqrt{\mu_k}(1+\gamma_{k-1})(1+\gamma_k+\gamma_{k-1})}\sqrt{\frac{\kappa_k}{\kappa_{k-1}}}
	\leq C(R)
\]
and since the following estimate can be obtained by comparing $(k-1)$ to $\frac{p}{2}$
\[
\frac{\sqrt{\kappa_k\kappa_{k-1}}}{\kappa_{p+1}(\lambda_p-\lambda_{k-1})}
	\leq C(R)\frac{p}{k|p-(k-1)|}
	\leq C'(R),
\]
we get
\begin{align*}
\left|\sum_{p\neq k-2,k-1}\frac{\mu_{p+1}}{\sqrt{\mu_k}}\frac{\sqrt{\kappa_k\kappa_{k-1}}}{\kappa_{p+1}(\lambda_p-\lambda_{k-1})}\frac{\gamma_{p+1}}{(\lambda_k-\lambda_p-1)(\lambda_{k-1}-\lambda_p-1)}\right|
	&\leq C(R)\sum_{p\neq k-2,k-1}\gamma_{p+1}\\
	&\leq C'(R).
\end{align*}
Finally,
\begin{align*}
\sum_{\substack{1\leq l\leq n\\k-1\neq l-1,l}}\sqrt{\mu_k}\sqrt{\frac{
\kappa_{k-1}}{\kappa_{k}}}\frac{\gamma_l}{(\lambda_{k-1}-\lambda_l)(\lambda_{k-1}-\lambda_{l-1}-1)(\lambda_l-\lambda_{k-1}-1)}
	&\leq C(R)\sum_{\substack{1\leq l\leq n\\k-1\neq l-1,l}}\gamma_l\\
	&\leq C'(R).
\end{align*}
We conclude that \[|A^+_{n,p}|\leq C(R).\]

We finally check that for all $h\in\ell^1_+$,
\[
\left|\d A^+_{n,p}(\gamma).h\right|\leq C(R)\|h\|_{\ell^1_+}.
\]
This is a consequence of the formulas, and the fact that for all $n\in\N$, the spectral parameters, seen as functions of the actions $\gamma$, satisfy $|\d\lambda_n^*(\gamma).h|\leq \|h\|_{\ell^1_+}$, $|\d\mu_n^*(\gamma).h|\leq C(R)\|h\|_{\ell^1_+}$ and $|(n+1)\d\kappa_n^*(\gamma).h|\leq C(R)\|h\|_{\ell^1_+}$ if $\|\gamma\|_{\ell^1_+}\leq R$.

$\bullet$ Similarly, we write the formula for $c_n^-$
\[
c_n^-
	=\frac{1}{\sqrt{\kappa_0}}\sum_{p=1}^{+\infty}\frac{\sqrt{\kappa_p}\zeta_p}{\lambda_p-\lambda_0}\overline{M_{p,0}}
	-\sum_{k=1}^n \frac{1}{\sqrt{\mu_k}}\langle\psi_k|e^{-ix}\rangle,
\]
where
\begin{align*}
\langle\psi_k|e^{-ix}\rangle
	&=\sum_{p\neq k}\frac{M_{k-1,p}}{\lambda_p-\lambda_k}\overline{M_{k,p}}-\sum_{p\neq k-1}\frac{\overline{M_{p,k}}}{\lambda_p-\lambda_{k-1}}\overline{M_{k-1,p}}.
\end{align*}
Using the expression of $M_{n,p}$, we get
\begin{align*}
c_n^-
	&=\sum_{p=1}^{+\infty}\sqrt{\mu_{p+1}}\sqrt{\frac{\kappa_p}{\kappa_{p+1}}}\frac{{\color{teal} \zeta_p\overline{\zeta_{p+1}}}}{(\lambda_p-\lambda_0)(\lambda_0-\lambda_p-1)}
	-\sum_{k=1}^n \frac{1}{\sqrt{\mu_k}}\langle\psi_k|e^{-ix}\rangle,
\end{align*}
where we isolate the terms $p=k+1$ and $p=k$ in the first and second term respectively in the formula for $\langle\psi_k|e^{-ix}\rangle$ 
\begin{multline*}
\langle\psi_k|e^{-ix}\rangle
	=\frac{\sqrt{\mu_k\mu_{k+1}}}{(1+\gamma_{k+1}+\gamma_k)(1+\gamma_{k+1})}\sqrt{\frac{\kappa_{k+1}}{\kappa_k}}{\color{blue} \overline{\zeta_{k+1}}\zeta_k}\\
	+{\color{blue} \zeta_k\overline{\zeta_{k+1}}}\sum_{p\neq k,k+1}\sqrt{\mu_k\mu_{k+1}}\frac{\kappa_p}{\sqrt{\kappa_k\kappa_{k+1}}(\lambda_p-\lambda_k)}\frac{\zeta_p\overline{\zeta_p}}{(\lambda_p-\lambda_{k-1}-1)(\lambda_p-\lambda_k-1)}\\
	+\frac{\sqrt{\mu_{k+1}\mu_k}}{1+\gamma_k}\sqrt{\frac{\kappa_k}{\kappa_{k+1}}}{\color{blue} \zeta_k\overline{\zeta_{k+1}}}\\
	-\overline{\zeta_k}\zeta_k\sum_{p\neq k-1,k}\sqrt{\mu_{p+1}\mu_k}\sqrt{\frac{\kappa_p}{\kappa_{p+1}}}\frac{{\color{violet} \overline{\zeta_{p+1}}\zeta_p}}{(\lambda_p-\lambda_{k-1})(\lambda_k-\lambda_p-1)(\lambda_p-\lambda_{k-1}-1)}.
\end{multline*}
One can therefore write $c_n^-=\sum_kA^-_{n,k}\zeta_k\overline{\zeta_{k+1}}$, with
\begin{align*}
A^-_{n,k}
	=&{\color{teal}\un_{k\neq 0}\sqrt{\mu_{k+1}}\sqrt{\frac{\kappa_k}{\kappa_{k+1}}}\frac{1}{(\lambda_k-\lambda_0)(\lambda_0-\lambda_k-1)}}\\
	&{\color{blue} +\un_{1\leq k\leq n}\Big(\frac{\sqrt{\mu_{k+1}}}{(1+\gamma_{k+1}+\gamma_k)(1+\gamma_{k+1})}\sqrt{\frac{\kappa_{k+1}}{\kappa_k}}}\\
	&{\color{blue} +\sum_{p\neq k,k+1}\sqrt{\mu_{k+1}}\frac{\kappa_p}{\sqrt{\kappa_k\kappa_{k+1}}(\lambda_p-\lambda_k)}\frac{\gamma_p}{(\lambda_p-\lambda_{k-1}-1)(\lambda_p-\lambda_k-1)}}\\
	&{\color{blue} +\frac{\sqrt{\mu_{k+1}}}{1+\gamma_k}\sqrt{\frac{\kappa_k}{\kappa_{k+1}}}\Big)}\\
	&{\color{violet} -\sum_{\substack{1\leq l\leq n\\k\neq l-1,l}}\sqrt{\mu_{k+1}}\sqrt{\frac{\kappa_k}{\kappa_{k+1}}}\frac{\gamma_l}{(\lambda_k-\lambda_{l-1})(\lambda_l-\lambda_k-1)(\lambda_k-\lambda_{l-1}-1)}}.
\end{align*}
The proof of the estimates
\[|A^-_{n,p}|\leq C(R)\]
and
\[
\left|\d A^-_{n,p}(\gamma).h\right|\leq C(R)\|h\|_{\ell^1_+}
\]
are similar to the terms $A_{n,p}^+$.
\end{proof}

\subsubsection{The case of \texorpdfstring{$\delta_\perp\langle \un|f_n\rangle$}{delta_perp}}

Finally, we focus on the term
\[
\delta_\perp\langle \un|f_n\rangle
	=\sum_{p\neq n}\frac{\langle f_p\overline{f_n}|h \rangle}{\lambda_p-\lambda_n}\langle\un|f_p\rangle
\]
when $h=\cos$ and $h=\sin$.

\begin{lem}\label{lem:delta_perp_lip}
For $n\geq 0$, let us write
\[
b_n^{\pm}
	:=\sum_{p\neq n}\frac{\langle f_p\overline{f_n}|e^{\pm ix} \rangle}{\lambda_p-\lambda_n}\langle\un|f_p\rangle.
\]
Then there exist $p_n^*, B_{n,k}^*,q_n^*\in\classeC^1(\ell^1_+,\R)$, $n,k\in\N$, such that for all $n\in\N$, for all $\zeta\in h^{\half}_+$,
\begin{equation*}
\frac{b_n^+}{\sqrt{\kappa_n}}
	=q_n^*(\gamma)\zeta_{n+1}
\end{equation*}
and
\begin{equation*}
\frac{b_n^-}{\sqrt{\kappa_n}}
	=p_n^*(\gamma)\zeta_{n-1}+\left(\sum_{k=0}^{+\infty}B_{n,k}^*(\gamma)\zeta_k\overline{\zeta_{k+1}}\right)\zeta_n.
\end{equation*}
The terms $p_n^*$, $B_{n,k}^*$ and $q_n^*$ are uniformly bounded and Lipschitz on finite balls: for all $R>0$, there exists $C(R)>0$ such that for all $\gamma\in \ell^1_+$ satisfying $\|\gamma\|_{\ell^1_+}\leq R$,  for all $n,k\geq 0$, we have
\begin{equation*}
|p_n^*(\gamma)|+|B_{n,k}^*(\gamma)|+|q_n^*(\gamma)|\leq C(R)
\end{equation*}
and for all $h\in \ell^1_+$,
\begin{equation*}
|\d p_n^*(\gamma).h|+|\d B_{n,k}^*(\gamma).h|+|\d q_n^*(\gamma).h|\leq C(R)\|h\|_{\ell^1_+}.
\end{equation*}
\end{lem}

\begin{proof}
We first simplify the formulas for $b_n^+$ and $b_n^-$: on the one hand,
\begin{align*}
b_n^+
	&= \sum_{p\neq n}\frac{M_{n,p}}{\lambda_p-\lambda_n}\sqrt{\kappa_p}\zeta_p\\
	&= \sqrt{\mu_{n+1}}\zeta_{n+1}\left(\frac{\sqrt{\kappa_{n+1}}}{1+\gamma_{n+1}}+ \frac{1}{\sqrt{\kappa_{n+1}}}\sum_{p\neq n, n+1}\frac{\kappa_p}{(\lambda_p-\lambda_n)(\lambda_p-\lambda_n-1)}\zeta_p\overline{\zeta_p}\right),
\end{align*}
and on the other hand,
\begin{align*}
b_n^-
	&= \sum_{p\neq n}\frac{\overline{M_{p,n}}}{\lambda_p-\lambda_n}\sqrt{\kappa_p}\zeta_p\\
	&=-\frac{\sqrt{\mu_n}\sqrt{\kappa_{n-1}}}{1+\gamma_n}\zeta_{n-1}
	+\sqrt{\kappa_n}\zeta_n \sum_{p\neq n,n-1}\frac{\sqrt{\mu_{p+1}}}{(\lambda_p-\lambda_n)(\lambda_n-\lambda_p-1)}\sqrt{\frac{\kappa_p}{\kappa_{p+1}}}\zeta_p\overline{\zeta_{p+1}}.
\end{align*}

We now define
\begin{equation}\label{def:pn}
p_n^*(\gamma)=-\frac{\sqrt{\mu_n}\sqrt{\kappa_{n-1}}}{\sqrt{\kappa_n}(1+\gamma_n)},
\end{equation}
\[
B_{n,k}^*(\gamma)=\un_{k\neq n,n-1}\frac{\sqrt{\mu_{k+1}}}{(\lambda_k-\lambda_n)(\lambda_n-\lambda_k-1)}\sqrt{\frac{\kappa_k}{\kappa_{k+1}}}
\]
and
\begin{equation}\label{def:qn}
q_n^*(\gamma)=\frac{\sqrt{\mu_{n+1}}}{\sqrt{\kappa_n}}\left(\frac{\sqrt{\kappa_{n+1}}}{1+\gamma_{n+1}}+ \frac{1}{\sqrt{\kappa_{n+1}}}\sum_{p\neq n, n+1}\frac{\kappa_p\gamma_p}{(\lambda_p-\lambda_n)(\lambda_p-\lambda_n-1)}\right).
\end{equation}
The estimates on $p_n^*$ and $B_{n,k}^*$ follow from their definitions. Note that in order to bound the sum on the right in the expression of $q_n^*$, we use the inequality $|\lambda_p-\lambda_n|\geq |p-n|$ and the estimate
\[
\sum_{\substack{p\geq 1\\p\neq n}}\frac{1}{p|p-n|}\leq \frac{C}{n+1},
\]
which can be obtained when comparing $p$ to $\frac{n}{2}$.
\end{proof}

\begin{proof}[Proof of Theorem~\ref{thm:dzeta.cos}]
Starting from the expression~\eqref{eq:dzeta} of $\d\zeta_n[u].\cos$ and $\d\zeta_n[u].\sin$, we have the formulas
\[
\d\zeta_n[u].\cos=\zeta_n\left(-i\Im(\langle \widetilde{\xi_n}|\cos\rangle)-\frac{\d\ln(\kappa_n)[u].\cos}{2}\right)+\frac{b_n^++b_n^-}{2\sqrt{\kappa_n}}
\]
and
\[
\d\zeta_n[u].\sin=\zeta_n\left(-i\Im(\langle \widetilde{\xi_n}|\sin\rangle)-\frac{\d\ln(\kappa_n)[u].\sin}{2}\right)-\frac{b_n^+-b_n^-}{2i\sqrt{\kappa_n}}.
\]
The proof is now a direct consequence of Lemmas~\ref{lem:dkappa_lip},~\ref{lem:delta_parallel_lip} and~\ref{lem:delta_perp_lip}.
\end{proof}


\subsubsection{Lipschitz properties}\label{part:dzeta_lipschitz}

We now deduce from Theorem~\ref{thm:dzeta.cos} that the terms $\d\zeta_n[u].\cos$ and $\d\zeta_n[u].\sin$ are bounded and Lipschitz functions of $\zeta$ on finite balls for the norms
\[h^{\half-s}_+=\{(\zeta_n)_{n\geq 1}\in\C^\N\mid \sum_{n=1}^{+\infty}n^{1-2s}|\zeta_n|^2<+\infty\},
\quad 0\leq s<\half.
\]

\begin{cor}[Bounds and Lipschitz properties]\label{cor:dzeta.cos} Fix $R>0$. There exists $C(R)>0$ such that the following holds.

Let $s< -\half$. If $\zeta\in h^{\half}_+$ and $\|\zeta\|_{\ell^2_+}\leq R$, then writing $u=\Phi^{-1}(\zeta)$,
\[
\|(\d\zeta_n[u].\cos)_n\|_{h^{\half+s}_+}+\|(\d\zeta_n[u].\sin)_n\|_{h^{\half+s}_+}
	\leq C(R)\|\zeta\|_{h^{\half+s}_+}.
\]
Moreover, for all $\zeta^{(1)},\zeta^{(2)}\in h^{\half}_+$ satisfying $\|\zeta^{(1)}\|_{h^{\half+s}_+},\|\zeta^{(2)}\|_{h^{\half+s}_+}\leq R$, if $u^{(1)}=\Phi^{-1}(\zeta^{(1)})$ and $u^{(2)}=\Phi^{-1}(\zeta^{(2)})$, then
\[
\|(\d\zeta_n[u^{(1)}].\cos)_n-(\d \zeta_n[u^{(2)}].\cos)_n\|_{h^{\half+s}_+}
	\leq C(R)\|\zeta^{(1)}-\zeta^{(2)}\|_{h^{\half+s}_+}
\]
and
\[
\|(\d\zeta_n[u^{(1)}].\sin)_n-(\d \zeta_n[u^{(2)}].\sin)_n\|_{h^{\half+s}_+}
	\leq C(R)\|\zeta^{(1)}-\zeta^{(2)}\|_{h^{\half+s}_+}.
\]
\end{cor}

\begin{proof}
Let $\zeta\in h^{\half}_+$ such that $\|\zeta\|_{\ell^2_+}\leq R$. We write
\begin{equation*}
\d\zeta_n[u].\cos
	=p_n^*(\gamma)\zeta_{n-1}+\left(\sum_{k=0}^{+\infty}A_{n,k}^*(\gamma)\overline{\zeta_k}\zeta_{k+1}+B_{n,k}^*(\gamma)\zeta_k\overline{\zeta_{k+1}}\right)\zeta_n
	+q_n^*(\gamma)\zeta_{n+1},
\end{equation*}
where the terms $p_n^*$, $q_n^*$, $A_{n,k}^*$ and $B_{n,k}^*$ are bounded and Lipschitz on balls of finite radius in~$\ell^1_+$, and $\gamma=(|\zeta_n|^2)_n\in\ell^1_+$. We have $\|\gamma\|_{\ell^1_+}\leq R^2$, therefore there exists $C(R)>0$ such that
\begin{equation*}
|p_n^*(\gamma)|+|q_n^*(\gamma)|+|A_{n,k}^*(\gamma)|+|B_{n,k}^*(\gamma)|\leq C(R)
\end{equation*}
and for all $h\in \ell^1_+$,
\begin{equation*}
|\d p_n^*(\gamma).h|+|\d q_n^*(\gamma).h|+|\d A_{n,k}^*(\gamma).h|+|\d B_{n,k}^*(\gamma).h|\leq C(R)\|h\|_{\ell^1_+}.
\end{equation*}
We deduce that for all $n\geq 1$,
\[
|\d\zeta_n[u].\cos|
	\leq C(R)(|\zeta_{n-1}|+|\zeta_n|+|\zeta_{n+1}|).
\]
Let $s>-\half$  and assume $\|\zeta\|_{\ell^2_+}\leq R$, a summation leads to
\[
\sum_{n=1}^{+\infty}n^{1+2s}|\d\zeta_n[u].\cos|^2
	\leq C(R)\sum_{n=1}^{+\infty}n^{1+2s}|\zeta_n|^2
	= C(R)\|\zeta\|_{h^{\half+s}_+}^2.
	\]

Now, define, for $\zeta\in h^{\half}_+$,
\[
S_n(\zeta):=\d\zeta_n[\Phi^{-1}(\zeta)].\cos.
\]
We prove that the differential of $S_n$ is bounded on balls of finite radius of $h^{\half+s}_+$ for $s>-\half$. Fix $h\in\ell^2_+$ and denote $H=(H_n)_{n\geq 1}$ the differential of the map $\zeta\in\ell^2_+\mapsto (|\zeta_n|^2)_{n\geq1}\in\ell^1_+$ at $\zeta$ applied to $h$. For all $n\geq 1$, we have $H_n=\zeta_n\overline{h_n}+\overline{\zeta_n}h_n$, therefore $H$ is in $\ell^1_+$ and
\[
\|H\|_{\ell^1_+}\leq 2\|\zeta\|_{\ell^2_+}\|h\|_{\ell^2_+}.
\]
The differential of $S_n$ at $\zeta$ applied to $h$ now writes
\begin{align*}
\d S_n(\zeta).h
	=&\d p_n^*(\gamma).H \zeta_{n-1}
	+\d q_n^*(\gamma).H\zeta_{n+1}
	+p_n^*(\gamma)h_{n-1}+q_n^*(\gamma)h_{n+1}\\
	&+\left(\sum_{k=0}^{+\infty}\d A_{n,k}^*(\gamma).H \overline{\zeta_k}\zeta_{k+1}+\d B_{n,k}^*(\gamma).H \zeta_k\overline{\zeta_{k+1}}\right)\zeta_n\\
	&+\left(\sum_{k=0}^{+\infty} A_{n,k}^*(\gamma)(\overline{h_k}\zeta_{k+1}+\overline{\zeta_k}h_{k+1})+ B_{n,k}^*(\gamma)(h_k\overline{\zeta_{k+1}}+\zeta_k\overline{h_{k+1}})\right)\zeta_n\\
	&+\left(\sum_{k=0}^{+\infty}A_{n,k}^*(\gamma)\overline{\zeta_k}\zeta_{k+1}+B_{n,k}^*(\gamma)\zeta_k\overline{\zeta_{k+1}}\right)h_n.
\end{align*}
When $\|\zeta\|_{h^{\half+s}_+}\leq R$, a summation leads to
\begin{align*}
\sum_{n=1}^{+\infty}n^{1-2s}|\d S_n(\zeta).h|^2
	&\leq C(R)\|H\|_{\ell^1_+}^2+C(R)\sum_{n=1}^{+\infty}n^{1+2s}|h_n|^2\\
	&\leq C'(R)\|h\|_{h^{\half+s}_+}^2.
\end{align*}
These arguments are also valid in the case of $\d\zeta_n[u].\sin$.
\end{proof}

\section{Flow map}\label{section:flow_map}

In this section, we prove global well-posedness for equation~\eqref{eq:damped} and establish the weak sequential continuity of the flow map.

Thanks to Theorem~\ref{thm:dzeta.cos}, we have simplified the system of ODEs satisfied by the Birkhoff coordinates $\zeta$. In order for the vector field to be Lipschitz along the second variable, we need to remove the oscillatory part in the frequencies $\omega_n=n^2-2\sum_{k\geq 1}\min(k,n)|\zeta_k|^2$. Therefore, we introduce a change of functions $z_n(t)=e^{-in^2t}\zeta_n(t)$, $n\geq1$: the new coordinates $z=(z_n)_{n\geq1}$ are solution to an ODE
\begin{equation}\label{eq:damped_z}
z'(t)=F(t,z(t)).
\end{equation}
The vector field $F=(F_n)_{n\geq 1}$ is defined  as
\begin{equation}\label{eq:def_F}
F_n(t,z)
	=e^{-in^2t}\left(i\widetilde{\omega}_n(z)e^{in^2t}z_n
	-\alpha \widetilde{Z}_n(t,z)\right),
\end{equation}
where $\widetilde{\omega}_n$ are the new frequencies
\[
\widetilde{\omega}_n(z)
	=-2\sum_{k=1}^{+\infty}\min(k,n)|z_k|^2
\]
and
\begin{equation}\label{def:Z_n}
\widetilde{Z}_n(t,z)
	=Z_n(\zeta(t,z))
	:=\langle u|\cos\rangle \d\zeta_n[u].\cos+\langle u|\sin\rangle \d\zeta_n[u].\sin,
\end{equation}
with
\[
\zeta_n(t,z)=e^{in^2t}z_n, \quad n\geq 1
\]
and
\[
u(t,z)=\Phi^{-1}\left(\zeta(t,z)\right).
\]

We apply the Cauchy-Lipschitz theorem for ordinary differential equations in part~\ref{part:lwp} to get local well-posedness for equation~\eqref{eq:damped}. In part~\ref{part:gwp}, global well-posedness follows from the existence of a Lyapunov functional controlling the $L^2$ norm of the solution and a compactness argument. Finally, we prove the weak sequential continuity of the flow map in part~\ref{part:weak_sequential}.

\subsection{Local well-posedness}\label{part:lwp}

In this part, we prove that the vector field $F$ is locally Lipschitz with respect to the second variable and apply the Cauchy-Lipschitz theorem for ODEs.

\begin{thm}\label{thm:F_loc_lipschitz}
The map $F$ is bounded and Lipschitz on finite balls in the following sense. Let $R>0$. Then there exists $C(R)>0$ such that if $\|z^{(1)}\|_{h^{\half}_+}\leq R$ and $\|z^{(2)}\|_{h^{\half}_+}\leq R$, then for all $t\in\R$,
\[
\|F(t,z^{(1)})\|_{h^{\half}_+}+\|F(t,z^{(2)})\|_{h^{\half}_+}
	\leq C(R),
\]
and
\[
\|F(t,z^{(2)})-F(t,z^{(1)})\|_{h^{\half}_+}
	\leq C(R)\|z^{(2)}-z^{(1)}\|_{h^{\half}_+}.
\]
Moreover, $F$ is weakly sequentially continuous with respect to the second variable.
\end{thm}

\begin{proof}
We start from the definition of $F$ from~\eqref{eq:def_F}:\[
F_n(t,z)
	=e^{-in^2t}\left(i\widetilde{\omega}_n(z)e^{in^2t}z_n
	-\alpha \widetilde{Z}_n(t,z)\right).
\]
We consider the terms $i\widetilde{\omega}_n(z)z_n$ and $\widetilde{Z}_n$ separately.

$\bullet$ First, the part $(t,z)\in\R\times h^{\half}_+\mapsto (i\widetilde{\omega}_n(z)z_n)_{n\geq 1}\in h^{\half}_+$ is bounded and Lipschitz on finite balls with respect to the variable $z$. Indeed, fix  $z^{(1)},z^{(2)}\in h^{\half}_+$ bounded by $R$ in $h^{\half}_+$. We have
\begin{equation*}
|\widetilde{\omega}_n(z^{(1)})|
	=2\sum_{k=1}^{+\infty}\min(k,n)|z^{(1)}_k|^2\\
	\leq 2R^2,
\end{equation*}
so that
\begin{equation*}
\|(\widetilde{\omega}_n(z^{(1)})z_n^{(1)})_{n\geq 1}\|_{h^{\half}_+}
	\leq 2R^3.
\end{equation*}
Moreover, thanks to Cauchy-Schwarz' inequality,
\begin{align*}
|\widetilde{\omega}_n(z^{(2)})-\widetilde{\omega}_n(z^{(1)})|
	&=2\left|\sum_{k=1}^{+\infty}\min(k,n)\left(|z^{(2)}_k|^2-|z^{(1)}_k|^2\right)\right|\\
	&\leq C(R)\|z^{(2)}-z^{(1)}\|_{h^{\half}_+},
\end{align*}
so that
\begin{equation*}
\|(\widetilde{\omega}_n(z^{(2)})z_n^{(2)})_{n\geq 1}-(\widetilde{\omega}_n(z^{(1)})z_n^{(1)})_{n\geq 1}\|_{h^{\half}_+}
	\leq C'(R)\|z^{(2)}-z^{(1)}\|_{h^{\half}_+}.
\end{equation*}

Let us now establish the weak sequential continuity. Let $z^{(k)}$, $k\geq 1$, be a sequence of elements of $h^{\half}_+$ weakly convergent to some $z\in h^{\half}_+$. By compactness, this sequence is strongly convergent in the space $\ell^2_+$. Therefore, for all $n\geq 1$,
\[
|\widetilde{\omega}_n(z^{(k)})-\widetilde{\omega}_n(z)|
	\leq 2n\sum_{p=1}^{+\infty}\left||z^{(k)}_p|^2-|z_p|^2\right|
\]
converges to zero as $k$ goes to infinity. We conclude that the sequence $(\widetilde{\omega}_n(z^{(k)})z_n^{(k)})_{n\geq 1}$ is weakly convergent to $(\widetilde{\omega}_n(z)z_n)_{n\geq 1}$ in $h^{\half}_+$ as $k$ goes to infinity.

$\bullet$ We now prove that $\widetilde{Z}=(\widetilde{Z}_n)_n$ defines a bounded and Lipschitz map on finite balls from $h^{\half-s}_+$ to itself for all $s\in [0,\half)$. 
The introduction of a variable Sobolev space $h^{\half-s}_+$ in the Lipschitz properties will be useful when considering the weak sequential continuity.

From Corollary~\ref{cor:<u|e^ix>_lip}, we know that $\zeta\in h^{\half}_+\mapsto \langle u|\cos\rangle\in\R$ and $\zeta\in h^{\half}_+\mapsto \langle u|\sin\rangle\in\R$, where $u=\Phi^{-1}(\zeta)$, define real-valued maps which are bounded by $C(R)$, and Lipschitz with respect to the $\ell^2_+$ norm: if $\|\zeta^{(1)}\|_{\ell^2_+}\leq R$ and $\|\zeta^{(2)}\|_{\ell^2_+}\leq R$, then
\[
|\langle u^{(1)}|\cos\rangle|+|\langle u^{(1)}|\sin\rangle|+|\langle u^{(2)}|\cos\rangle|+|\langle u^{(2)}|\sin\rangle|\leq C(R)
\]
and
\[
|\langle u^{(1)}|\cos\rangle-\langle u^{(2)}|\cos\rangle|
	+|\langle u^{(1)}|\sin\rangle-\langle u^{(2)}|\sin\rangle|
	\leq C(R)\|\zeta^{(1)}-\zeta^{(2)}\|_{\ell^2_+}.
\]

From Corollary~\ref{cor:dzeta.cos}, we also know that the terms $\d\zeta_n[u].\cos$ and $\d\zeta_n[u].\sin$, when restricted to a ball of radius $R$, are bounded by $C(R)$ and Lipschitz with respect to the $h^{\half-s}_+$ norm for all $s\in[0,\half)$: if $\|\zeta^{(1)}\|_{h^{\half-s}_+}\leq R$ and $\|\zeta^{(2)}\|_{h^{\half-s}_+}\leq R$, then
\[
\|(\d\zeta_n[u^{(1)}].\cos)_n\|_{h^{\half-s}_+}+\|(\d \zeta_n[u^{(1)}].\sin)_n\|_{h^{\half-s}_+}
	\leq C(R),
\]
\[\|(\d\zeta_n[u^{(2)}].\cos)_n\|_{h^{\half-s}_+}+\|(\d \zeta_n[u^{(2)}].\sin)_n\|_{h^{\half-s}_+}
	\leq C(R),
\]
\[
\|(\d\zeta_n[u^{(1)}].\cos)_n-(\d \zeta_n[u^{(2)}].\cos)_n\|_{h^{\half-s}_+}
	\leq C(R)\|\zeta^{(1)}-\zeta^{(2)}\|_{h^{\half-s}_+}
\]
and
\[
\|(\d\zeta_n[u^{(1)}].\sin)_n-(\d \zeta_n[u^{(2)}].\sin)_n\|_{h^{\half-s}_+}
	\leq C(R)\|\zeta^{(1)}-\zeta^{(2)}\|_{h^{\half-s}_+}.
\]

We deduce that $Z=(Z_n)_n$ is bounded and Lipschitz with respect to the $h^{\half-s}_+$ norm when $s\in[0,\half)$: if $\|\zeta^{(1)}\|_{h^{\half}_+}\leq R$ and $\|\zeta^{(2)}\|_{h^{\half}_+}\leq R$, then
\[
\|Z(\zeta^{(1)})\|_{h^{\half-s}_+}+\|Z(\zeta^{(2)})\|_{h^{\half-s}_+}
	\leq C(R)
\]
and
\[
\|Z(\zeta^{(1)})-Z(\zeta^{(2)})\|_{h^{\half-s}_+}
	\leq C(R)\|\zeta^{(1)}-\zeta^{(2)}\|_{h^{\half-s}_+}.
\]
Now, fix $t\in\R$ and $z^{(1)},z^{(2)}\in h^{\half}_+$ such that $\|z^{(1)}\|_{h^{\half}_+}\leq R$ and $\|z^{(2)}\|_{h^{\half}_+}\leq R$. Then by definition of $\zeta^{(1)}$ and $\zeta^{(2)}$ as $\zeta_n^{(k)}(t,z)=e^{in^2t}z_n^{(k)}$, for $n\geq 1$ and $k=1,2$, we see that $\|\zeta^{(k)}\|_{h^{\half}_+}=\|z^{(k)}\|_{h^{\half}_+}$, moreover, for $s\in[0,\half)$, $\|\zeta^{(1)}-\zeta^{(2)}\|_{h^{\half-s}_+}=\|z^{(1)}-z^{(2)}\|_{h^{\half-s}_+}$. We have therefore proven that
\[
\|\widetilde{Z}(z^{(1)})\|_{h^{\half-s}_+}+\|\widetilde{Z}(z^{(2)})\|_{h^{\half-s}_+}
	\leq C(R)
\]
and
\begin{equation}\label{ineq:Ptilde_Lipschitz}
\|\widetilde{Z}(z^{(1)})-\widetilde{Z}(z^{(2)})\|_{h^{\half-s}_+}
	\leq C(R)\|z^{(1)}-z^{(2)}\|_{h^{\half-s}_+}.
\end{equation}

We now prove that $\widetilde{Z}$ is weakly sequentially continuous with respect to~$z$. Indeed, let $z^{(k)}$, $k\geq 1$, be a sequence of elements of $h^{\half}_+$ weakly convergent to some $z\in h^{\half}_+$. We fix $s\in(0,\half)$ and use the fact that the embedding $h^{\half}_+\hookrightarrow h^{\half-s}_+$ is compact: from the Rellich theorem, $z^{(k)}$ is strongly convergent to $z$ in $h^{\half-s}_+$. But inequality~\eqref{ineq:Ptilde_Lipschitz} now implies that $\widetilde{Z}(z^{(k)})$ is strongly convergent to $\widetilde{Z}(z)$ in $h^{\half-s}_+$, and therefore weakly convergent in~$h^{\half}_+$.
\end{proof}

We deduce from the Cauchy-Lipschitz theorem for ODEs that the initial value problem for the damped Benjamin-Ono equation in Birkhoff coordinates~\eqref{eq:damped_z} is locally well-posed in $h^{\half}_+$.

\begin{cor}\label{cor:LWP}
For all $z(0)\in h^{\half}_+$, there exists $T>0$ such that equation~\eqref{eq:damped_z}
\[
z'(t)=F(t,z(t))
\]
with initial data $z(0)$ admits a unique solution $z\in \classeC^1([0,T],h^{\half}_+)$. Moreover, the solution map is continuous from $h^{\half}_+$ to $\classeC^1([0,T],h^{\half}_+)$.
\end{cor}

Applying the inverse of the Birkhoff map, we conclude that the damped Benjamin-Ono equation~\eqref{eq:damped} is locally well-posed $L^2_{r,0}(\T)$: for all initial data $u_0\in L^2_{r,0}(\T)$, there exists $T>0$ such that equation~\eqref{eq:damped} admits a unique solution $u\in \classeC([0,T],L^2_{r,0}(\T))$ with initial data $u_0$ in the distribution sense; moreover, the solution map is continuous. Indeed, thanks to formula (4.5) in Lemma 4.2 in~\cite{GerardKappeler2019}, one can prove that the inverse Birkhoff map is admits a differential from~$h^{\half}_+$ to $L^2_{r,0}(\T)$.

\subsection{Global well-posedness}\label{part:gwp}

Thanks to a Lyapunov functional controlling the $L^2$ norm, we now establish global well-posedness for the damped Benjamin-Ono equation~\eqref{eq:damped} in $L^2_{r,0}(\T)$.
\begin{prop}[Lyapunov functional]\label{prop:normL2}
Let $u_0\in L^2_{r,0}(\T)$, and $([0,T^*),u)$ be the corresponding maximal solution. Then for all $t\in [0,T^*)$, 
\[
\frac{\d}{\d t}\|u(t)\|_{L^2(\T)}^2+2\alpha|\langle u(t)|e^{ix}\rangle|^2
	=0.
\]
Moreover,
\[
2\alpha\int_{0}^{T^*}|\langle u(t)|e^{ix}\rangle|^2\d t
	\leq \|u_0\|_{L^2(\T)}^2
\]
and if $T^*=+\infty$, then $|\langle u(t)|e^{ix}\rangle|$ tends to $0$ as $t$ goes to $+\infty$.
\end{prop}

\begin{proof}
Using the equation,
\begin{align*}
\frac{\d}{\d t}\|u(t)\|_{L^2(\T)}^2
	&=2\Re\left[\langle \partial_t u|u(t)\rangle\right]\\
	&=-2\alpha (|\langle u(t)|\cos\rangle|^2+|\langle u(t)|\sin\rangle|^2)+2\Re\left[\langle \partial_x(H\partial_x u-u^2)|u(t)\rangle\right]\\
	&=-2\alpha |\langle u(t)|e^{ix}\rangle|^2.
\end{align*}

It remains to show that $|\langle u(t)|e^{ix}\rangle|$ tends to $0$ as $t$ goes to $+\infty$. Taking the derivative,
\begin{multline*}
\frac{\d}{\d t}|\langle u(t)|\cos\rangle|^2
	=2\Re\left[\langle \partial_t u|\cos\rangle\langle\cos|u(t)\rangle\right]\\
	=-2\alpha\Re\left[\left(\langle u(t)|\cos\rangle\langle\cos|\cos\rangle+\langle u(t)|\sin\rangle\langle\sin|\cos\rangle\right)\langle\cos|u(t)\rangle\right]\\
	+2\Re\left[\langle H\partial_{xx} u-\partial_x(u^2)|\cos\rangle\langle\cos|u(t)\rangle\right]\\
	=-\alpha|\langle u(t)|\cos\rangle|^2+2\Re\left[(-\langle u(t)|H\partial_{xx}\cos\rangle+\langle u^2(t)|\partial_x\cos\rangle)\langle\cos|u(t)\rangle\right].
\end{multline*}
Since $H\cos=\sin$, we deduce that
\begin{align*}
\frac{\d}{\d t}|\langle u(t)|\cos\rangle|^2
	&=-\alpha|\langle u(t)|\cos\rangle|^2+2\Re\left[(\langle u(t)|\sin\rangle-\langle u^2(t)|\sin\rangle)\langle\cos|u(t)\rangle\right].
\end{align*}
But
\[
|\langle u(t)|\cos\rangle|
	\leq \|u\|_{L^2(\T)}
	\leq \|u_0\|_{L^2(\T)},
\]
\[
|\langle u(t)|\sin\rangle|
	\leq \|u_0\|_{L^2(\T)}
\]
and
\[
|\langle u^2(t)|\sin\rangle|
		\leq \|u\|_{L^2(\T)}^2
		\leq \|u_0\|_{L^2(\T)}^2,
\]
therefore $\frac{\d}{\d t}|\langle u(t)|\cos\rangle|^2$ is bounded, and the same can be proven for $\frac{\d}{\d t}|\langle u(t)|\sin\rangle|^2$. This observation combined with the fact that $|\langle u(t)|e^{ix}\rangle|$ is square integrable implies that $|\langle u(t)|e^{ix}\rangle|$ tends to zero as time goes to infinity if $T^*=+\infty$.
\end{proof}

\begin{prop}\label{prop:GWP}
For all $z(0)\in h^{\half}_+$, the maximal solution $t\mapsto z(t)\in h^{\half}_+$ is global.
\end{prop}

The proof of this proposition relies on the fact that, given an initial condition $z(0)\in h^{\half}_+$, it is possible to construct a weighted space $\ell(w)$ which contains $z(0)$ and which compactly embeds into $h^{\half}_+$. With a Gronwall type argument, we then prove that the maximal solution $t\mapsto z(t)$ is bounded in $\ell(w)$. Because of the compactness of the embedding $\ell(w)\hookrightarrow h^{\half}_+$ and the local well-posedness of the Cauchy problem, this ensures that the maximal solution is global.

\begin{lem}\label{lem:astuce_somme}
Let $(x_n)_{n\geq 1}\in(\R_+)^{\N}$ be a sequence of non-negative real numbers such that the series with general term $x_n$ is convergent: $\sum_{n\geq 1}x_n<+\infty$. Then there exist positive weights $(w_n)_{n\geq1}$ such that $w_n\to+\infty$ and
\[
\sum_{n\geq1}w_nx_n<+\infty.
\]
Moreover, one can assume that for all $n\geq 1$, $1\leq \frac{w_{n+1}}{w_n}\leq 2$.
\end{lem}

\begin{proof}
Let $r_n:=\sum_{p>n}x_p$ and define
\[
w_n:=\frac{1}{2^{-n}+\sqrt{r_n}+\sqrt{r_{n-1}}}.
\]
Then $w_n\to+\infty$, and
\begin{align*}
w_nx_n
	&=\frac{r_{n-1}-r_{n}}{2^{-n}+\sqrt{r_n}+\sqrt{r_{n-1}}}\\
	&\leq \sqrt{r_{n-1}}-\sqrt{r_{n}}.
\end{align*}
The general term of the upper bound defines a telescopic sum, so that the series $\sum_{n\geq 1} w_nx_n$ is convergent.

We now need to ensure that for all $n\geq 1$, $1\leq \frac{w_{n+1}}{w_n}\leq 2$. Let $\widetilde{w}_0:=w_0$ and define by induction
\[
\widetilde{w}_n:=\min(w_n,2\widetilde{w}_{n-1}), \quad n\geq 1.
\]
We check that $\widetilde{w}_n$ satisfies the required assumptions.
\begin{itemize}
\item Since $\widetilde{w}_n\leq w_n$, we know that $\sum_{n\geq 1}\widetilde{w}_nx_n<+\infty$.
\item By definition, the sequence $(w_n)_n$ is increasing, therefore
\[
\widetilde{w}_n\geq \min(w_{n-1},2\widetilde{w}_{n-1})\geq \widetilde{w}_{n-1}.
\]
Moreover, the other side of the inequality is immediate
\(\widetilde{w}_n\leq 2\widetilde{w}_{n-1}\): we have proven that $1\leq \frac{\widetilde{w}_n}{\widetilde{w}_{n-1}}\leq 2$.
\item We now prove that $(\widetilde{w}_n)_n$ tends to infinity. Let $(n_k)_{k\geq 1}$ be the increasing sequence of indexes $n$ for which $\widetilde{w}_n=w_n$ (this sequence might be finite).

If $(n_k)_{k\geq 1}$ is finite, then there exists $N$ such that for all $n\geq N$,
\[
\widetilde{w}_n:=2\widetilde{w}_{n-1}.
\]
The weights being positive, this implies that $(\widetilde{w}_n)_n$ tends to infinity.

Otherwise, since the sequence $(w_n)_n$ goes to infinity, so does the sequence $(w_{n_k})_k$ as $k$ goes to infinity. The sequence $(\widetilde{w}_n)_n$ being increasing, we deduce that $(\widetilde{w}_n)_n$ also goes to infinity.
\end{itemize}
\end{proof}

We now fix a sequence of positive weights $w=(w_n)_{n\geq 1}$ and define the weighted space
\[
\ell(w):=\Big\{z\in h^{\half}_+ \mid \|z\|_{\ell(w)}^2=\sum_{n\geq 1}nw_n|z_n|^2<+\infty\Big\}. 
\]
\begin{lem}\label{lem:astuce_gronwall}
Assume that $w_n\to+\infty$ and $1\leq \frac{w_{n+1}}{w_n}\leq 2$ for all $n\geq 1$. Fix $z(0)\in \ell(w)$. Then the maximal solution $t\mapsto z(t)\in h^{\half}_+$ is global.
\end{lem}

\begin{proof}
Let $I$ be the maximal time interval on which the solution $t\mapsto z(t)$ is defined. On this interval, we consider the functional $f:t\mapsto \|z(t)\|_{\ell(w)}^2$.

We prove that $f$ stays bounded on the interval $I$. Indeed,
\begin{align*}
\frac{\d}{\d t}f(t)
	&=\sum_{n\geq 1}nw_n\frac{\d}{\d t}\gamma_n.
\end{align*}
For $n\geq 1$, the time derivative of $\gamma_n$ is
\begin{align*}
\frac{\d}{\d t}\gamma_n(u(t))
	&=\d\gamma_n[u(t)].\partial_t u\\
	&=-\alpha\langle u(t)|\cos\rangle \d\gamma_n[u(t)].\cos -\alpha\langle u(t)|\sin\rangle \d\gamma_n[u(t)].\sin.
\end{align*}
From the expression~\eqref{eq:dgamma_n} of the differential of $\gamma_n$, we have
\begin{align*}
\d\gamma_n[u(t)].\cos
	=\Re(\langle |f_{n-1}|^2-|f_n|^2| e^{ix}\rangle)
\end{align*}
and
\begin{align*}
\d\gamma_n[u(t)].\sin
	=-\Im(\langle |f_{n-1}|^2-|f_n|^2 | e^{ix}\rangle),
\end{align*}
where
\begin{equation*}
\langle |f_{n-1}|^2-|f_n|^2 | e^{ix}\rangle
	=-\sqrt{\mu_{n}}\frac{\sqrt{\kappa_{n-1}}}{\sqrt{\kappa_n}}\overline{\zeta_{n-1}}\zeta_n+\sqrt{\mu_{n+1}}\frac{\sqrt{\kappa_n}}{\sqrt{\kappa_{n+1}}}\overline{\zeta_n}\zeta_{n+1}.
\end{equation*}
Since $\mu_n$ and $\frac{\kappa_n}{\kappa_{n+1}}$ are bounded by $C(R)$, we see that actually
\begin{equation}\label{ineq:gamma_n_dot}
\left|\frac{\d}{\d t}\gamma_n\right|
	\leq C(R)|\langle u(t)|e^{ix}\rangle|\left(|\zeta_{n-1}\zeta_n|+|\zeta_n\zeta_{n+1}|\right)
\end{equation}
so that
\begin{equation*}
\left|\frac{\d}{\d t}f(t)\right|
	\leq C(R)|\langle u(t)|e^{ix}\rangle|\sum_{n\geq 1}nw_n\left(|\zeta_{n-1}\zeta_n|+|\zeta_n\zeta_{n+1}|\right).
\end{equation*}
Using Cauchy-Schwarz' inequality, we deduce that
\begin{multline*}
\left|\frac{\d}{\d t}f(t)\right|
	\leq C(R)|\langle u(t)|e^{ix}\rangle|\left(\sum_{n\geq 1}nw_n|\zeta_{n}|^2\right)^{\half}\\
	\left(\left(\sum_{n\geq 1}nw_n|\zeta_{n-1}|^2\right)^{\half}
	+\left(\sum_{n\geq 1}nw_n|\zeta_{n+1}|^2\right)^{\half}\right).
\end{multline*}
For all $n\geq 1$, because $1\leq \frac{w_{n+1}}{w_n}\leq 2$, we have $1\leq \frac{(n+1)w_{n+1}}{nw_n}\leq 4$. Hence there exists $C(R)>0$ such that for all $t\in I$,
\begin{equation*}
\left|\frac{\d}{\d t}f(t)\right|
	\leq C(R)|\langle u(t)|e^{ix}\rangle|f(t).
\end{equation*}
Now, Gronwall's lemma implies that for all $t\in I$,
\begin{equation*}
|f(t)|
	\leq |f(0)|e^{C(R)\int_0^t|\langle u(t)|e^{ix}\rangle|\d t}.
\end{equation*}
Applying Proposition~\ref{prop:normL2}, we deduce that
\begin{equation*}
|f(t)|
	\leq |f(0)|e^{C'(R)\sqrt{t}}.
\end{equation*}

To conclude, for all $T>0$, $f$ stays on a bounded subset of $\ell(w)$ on the time interval $([0,T^*)\cap[0,T]$. Since the embedding $\ell(w)\hookrightarrow h^{\half}_+$ is compact because of the condition $w_n\to+\infty$, the solution $t\mapsto z(t)$ stays in a compact set of $h^{\half}_+$. This implies that there cannot be finite time blowup: $[0,T)\subset [0,T^*)$, therefore the solution is global.
\end{proof}

\begin{proof}[Proof of Proposition~\ref{prop:GWP}]
Let $z(0)\in h^{\half}_+$. Thanks to Lemma~\ref{lem:astuce_somme}, we construct a sequence $(w_n)_{n\geq 1}$ of positive weights such that $z(0)\in\ell(w)$, $w_n\to+\infty$ and $1\leq \frac{w_{n+1}}{w_n}\leq 2$ for all $n\geq 1$. Now Lemma~\ref{lem:astuce_gronwall} with the weights $(w_n)_n$ ensures that the maximal solution $t\mapsto z(t)\in h^{\half}_+$ is global.
\end{proof}


\subsection{Weak sequential continuity of the flow map} \label{part:weak_sequential}

In this part, we prove that the flow map for the damped Benjamin-Ono equation~\eqref{eq:damped} is weakly sequentially continuous. Note that in view of the weak sequential continuity of the Birkhoff map and its inverse, it is equivalent to establish the weak sequential continuity of the flow map for the damped Benjamin-Ono equation in Birkhoff coordinates~\eqref{eq:damped_z}.


\begin{prop}\label{prop:flow_weak_C0}
The flow map for the damped Benjamin-Ono equation~\eqref{eq:damped} is weakly sequentially continuous.

More precisely, let $u^{(k)}(0)\rightharpoonup u(0)$ in $L^2_{r,0}(\T)$. Then for all $T>0$, the sequence of solutions $u^{(k)}$ associated to the initial data $u^{(k)}(0)$ converges to the solution $u$ associated to the initial data $u_0$ in $\classeC_w([0,T],L^2_{r,0}(\T))$ (with the weak topology).
\end{prop}

\begin{proof}
We consider a sequence $u^{(k)}(0)$ weakly convergent to $u(0)$ in $L^2_{r,0}(\T)$.

For $k\in\N$, denote $z^{(k)}(0):=\Phi(u^{(k)}(0))$. By weak sequential continuity of the Birkhoff map $\Phi$ (see~\cite{GerardKappelerTopalov2020}, Remark 6 (iii)), we have $z^{(k)}(0)\rightharpoonup z(0):=\Phi(u(0))$ in $h^{\half}_+$, therefore this sequence is bounded in $h^{\half}_+$ by some $R>0$. But the solutions of the damped Benjamin-Ono equation have a decreasing norm in $L^2_{r,0}(\T)$ (Proposition~\ref{prop:normL2}) and from the Parseval formula $\|u\|_{L^2(\T)}^2=2\sum_{n\geq1}n|\zeta_n|^2$ (Remark 1.2 (i) in~\cite{GerardKappeler2019}), their Birkhoff coordinates have a decreasing norm in $h^{\half}_+$. Therefore for all $t\geq 0$, the sequence $(z^{(k)}(t))_k$ is also bounded in $h^{\half}_+$ by $R$. 

Recall from Theorem~\ref{thm:F_loc_lipschitz} that there exists $C(R)>0$ such that for all $t\geq 0$,
\begin{align*}
\|\frac{\d z^{(k)}(t)}{\d t}\|_{h^{\half}_+}
	= \|F(t,z^{(k)}(t))\|_{h^{\half}_+}
	\leq C(R).
\end{align*}
We conclude that the sequence $\|\frac{\d z^{(k)}(t)}{\d t}\|_{h^{\half}_+}$, $k\in\N$, is bounded in $h^{\half}_+$ uniformly in time.

Fix $T>0$. From Ascoli's theorem, we know that up to a subsequence, $z^{(k)}$ converges in $\classeC_w([0,T],h^{\half}_+)$ (with the weak topology) to a function $\widetilde{z}$.
In particular, since $\Phi^{-1}$ is weakly sequentially continuous , for all $t\in[0,T]$, the sequence $u^{(k)}(t)=\Phi^{-1}(z^{(k)}(t))$ is weakly convergent to $\widetilde{u}(t):=\Phi^{-1}(\widetilde{z}(t))$. But $F$ is sequentially weakly continuous with respect to the second variable (see Theorem~\ref{thm:F_loc_lipschitz}), therefore $F(t,z^{(k)}(t))$ is weakly convergent to $F(t,\widetilde{z}(t))$.

Passing to the limit in the equation, we conclude that $\widetilde{z}$ is a solution on $[0,T]$ to the original equation~\eqref{eq:damped_z} in the distribution sense with initial data $z(0)$. By uniqueness of such solutions (see Corollary~\ref{cor:LWP}), we deduce that $\widetilde{z}=z$.
\end{proof}

\section{Long time asymptotics}\label{section:limit_points}

In this section, we describe the weak limit points for the flow map (point 1 of Theorem~\ref{thm:LaSalle}) in part~\ref{part:LaSalle}, and prove that the convergence to these weak limit points is actually strong in $L^2_{r,0}(\T)$ (points 2 and 3 of Theorem~\ref{thm:LaSalle}) in part~\ref{part:strong_convergence}. In order to get the strong convergence, we show that the integral \(
\int_0^{+\infty}\sum_{n=0}^{+\infty}\gamma_n(t)\gamma_{n+1}(t)\d t
\) is finite in part~\ref{part:time_integrability_product}.

\subsection{Weak limit points of trajectories as \texorpdfstring{$t\to+\infty$}{t tends to infinity}}\label{part:LaSalle}

Using the LaSalle principle, we study the limit points of $(u(t))_{t\in\R}$ for the weak topology in $L^2_{r,0}(\T)$ as $t$ goes to $+\infty$ and prove Theorem~\ref{thm:LaSalle}.

\begin{prop}
Let $u$ be a solution of the damped equation~\eqref{eq:damped} with initial data $u_0\in L^2_{r,0}(\T)$. Then any weak limit $v_0$ of the sequence $(u(t))_{t\in\R}$ in $L^2_{r,0}(\T)$ as $t$ goes to $+\infty$ defines a solution $v$ to the Benjamin-Ono equation~\eqref{eq:bo} such that for all $t\geq 0$,
\[
\langle v(t)|e^{ix}\rangle=0.
\]
\end{prop}

\begin{proof}
Let $v_0$ be a weak limit of some sequence $(u(t_k))_{k\in\N}$ in $L^2_{r,0}(\T)$, where $t_k\to+\infty$, and let $v$ be the solution to the damped Benjamin-Ono equation~\eqref{eq:damped} with initial data~$v_0$. By weak sequential continuity of the flow in $L^2_{r,0}(\T)$ (Proposition~\ref{prop:flow_weak_C0}), for all $t\in\R$, we have
\[
u(t+t_k)\weaklim{k\to+\infty}{} v(t)
\]
and in particular
\[
\langle u(t+t_k)|e^{ix}\rangle
	\longrightarroww{k\to+\infty}{} \langle v(t)|e^{ix}\rangle.
\]
However, from Proposition~\ref{prop:normL2},  $\langle u(t+t_k)|e^{ix}\rangle$ tends to $0$ as $k$ goes to $+\infty$. We deduce that $\langle v(t)|e^{ix}\rangle=0$ for all $t\in\R$.
\end{proof}

\begin{prop}
An initial data $v_0\in L^2_{r,0}(\T)$ defines a solution $v$ to the Benjamin-Ono equation such that for all $t\in\R$,
\[
\langle v(t)|e^{ix}\rangle=0
\]
if and only if $v_0$ does not have two consecutive nonzero gaps:
\[
\forall n\in\N, \quad \zeta_n(v_0)\zeta_{n+1}(v_0)=0.
\]
\end{prop}

\begin{proof}
Assume that $v$ is a solution to the Benjamin-Ono equation with initial data $v_0$. Then the Birkhoff coordinates evolve as
\[
\zeta_n(v(t))=e^{i\omega_n(v_0)t}\zeta_n(v_0),
\]
where
\[
\omega_n(v_0)=n^2-2\sum_{p=1}^{+\infty}\min(p,n)\gamma_p.
\]
Applying Lemma~\ref{lem:<u|e^ix>}, we have
\[
\langle v(t)|e^{ix}\rangle
	=\sum_{n\geq 0} m_ne^{i(\omega_n(v_0)-\omega_{n+1}(v_0))t},
\]
where $m_n$ is constant along the flow of the Benjamin-Ono equation:
\[
m_n
	=M_{n,n}
	=-\overline{\zeta_n(v_0)}\zeta_{n+1}(v_0)\sqrt{\frac{\kappa_n(v_0)}{\kappa_{n+1}(v_0)}}\sqrt{\mu_{n+1}(v_0)}.
\]

One can see that if the condition $\overline{\zeta_n(v_0)}\zeta_{n+1}(v_0)=0$ is satisfied for all $n\in\N$, then all the coefficients $m_n$, $n\in\N$, vanish, therefore $\langle v(t)|e^{ix}\rangle=0$ for all $t\in\R$.

Conversely, assume that $\langle v(t)|e^{ix}\rangle=0$ for all $t\in\R$. Since
\[
\omega_{n+1}(v_0)-\omega_n(v_0)
	=2n+1-2\sum_{p=n+1}^{+\infty}\gamma_p,
\]
the sequence $(\omega_{n+1}(v_0)-\omega_n(v_0))_{n\in\N}$ is strictly increasing. Fix $n\in\N$, and choose $T>0$. Then
\begin{align*}
0
	&=\frac{1}{T}\int_0^{T}e^{i(\omega_{n+1}(v_0)-\omega_n(v_0))t}\langle v(t)|e^{ix}\rangle\d t\\
	&=m_n+\frac{1}{T}\sum_{p\neq n}m_p\frac{e^{i(\omega_{n+1}(v_0)-\omega_n(v_0)+\omega_p(v_0)-\omega_{p+1}(v_0))T}-1}{i(\omega_{n+1}(v_0)-\omega_n(v_0)+\omega_p(v_0)-\omega_{p+1}(v_0))}.
\end{align*}
Taking $T\to+\infty$ in this equality, we deduce that $m_n=0$.
Note that $\kappa_n>0$, $\kappa_{n+1}>0$ and $\mu_{n+1}>0$, therefore, we have $\overline{\zeta_n(v_0)}\zeta_{n+1}(v_0)=0$.
\end{proof}

\subsection{Time integrability for products of two consecutive modes}\label{part:time_integrability_product}

Let $u\in \classeC(\R_+,L^2_{r,0}(\T))$ be a solution to the damped Benjamin-Ono equation~\eqref{eq:damped}. We denote $\gamma_n(t):=\gamma_n(u(t))$ for all $n\geq 1$ and $t\in\R_+$ (with the convention $\gamma_0(t)=1$).

\begin{prop}\label{prop:I(T,T')_n,p}
Let $R:=\|u_0\|_{L^2(\T)}$. Then there exists $C(R)>0$ such that
\[
\int_0^{+\infty}\sum_{n=0}^{+\infty}\gamma_n(t)\gamma_{n+1}(t)\d t\leq C(R).
\]
Moreover, there exists a map $\varepsilon_u$ such that $\varepsilon_u(T)\to 0$ as $T\to+\infty$ and the following holds. For all $n\in\N$, fix $a_n\in\classeC^1(\R_+,\C)$ satisfying:
\begin{equation}\label{hyp:I(T,T')}
\forall t\geq 0,
\quad |a_n(t)|\leq 1
\quad \text{and}\quad
|\dot{a_n}(t)|\leq |\langle u(t)|e^{ix}\rangle|.
\end{equation}
Then for all $T\geq 0$, we have
\[
\int_T^{+\infty}\left|\sum_{n=0}^{+\infty}a_n(t)\zeta_n(t)\overline{\zeta_{n+1}(t)}\right|^2\d t\leq \varepsilon_u(T).
\]
\end{prop}

\begin{rk}\label{rk:I(T)_n,p}
Note that by homogeneity, if the maps $a_n\in\classeC^1(\R_+,\C)$ satisfy:
\begin{equation*}
\forall t\geq 0,
\quad |a_n(t)|\leq \frac{K}{n+1}
\quad \text{and}\quad
|\dot{a_n}(t)|\leq K|\langle u(t)|e^{ix}\rangle|,
\end{equation*}
then for all $T>0$,
\[
\int_T^{+\infty}\left|\sum_{n=0}^{+\infty}a_n(t)\zeta_n(t)\overline{\zeta_{n+1}(t)}\right|^2\d t\leq K^2\varepsilon_u(T).
\]
\end{rk}

Let $0\leq T<T'<+\infty$. We denote
\[
I(T,T')
	:=\int_T^{T'}\sum_{n=0}^{+\infty}\gamma_n(t)\gamma_{n+1}(t)\d t.
\]
For a family $a=(a_n)_n$ of maps satisfying assumption~\eqref{hyp:I(T,T')} of Proposition~\ref{prop:I(T,T')_n,p}, we also define
\[
J_a(T,T')
	=\int_T^{T'}\sum_{\substack{n,p\\n\neq p}}a_n(t)a_p(t)\overline{\zeta_n(t)}\zeta_{n+1}(t)\zeta_p(t)\overline{\zeta_{p+1}(t)}\d t.
\]
The integrals $I(T,T')$ and $J_a(T,T')$ are well defined since for all $t\in\R_+$, $\sum_nn\gamma_n(t)\leq R^2/2$ thanks to the Lyapunov functional (see Proposition~\ref{prop:normL2}) and the Parseval formula (see~\cite{GerardKappeler2019}, Remark 1.2(i)).

The key step in the proof of Proposition~\ref{prop:I(T,T')_n,p} is the following estimation of $J_a(T,T')$ depending on $I(T,T')$.
\begin{lem}\label{lem:J(T,T')}
Let $R=\|u_0\|_{L^2(\T)}$. There exist $C(R)>0$ and a map $\varepsilon_u$ with $\varepsilon_u(T)\to 0$ as $T\to+\infty$ such that the following holds. For all $0\leq T<T'<+\infty$ and for all family $a=(a_n)_n$ of maps satisfying assumption~\eqref{hyp:I(T,T')} of Proposition~\ref{prop:I(T,T')_n,p},
\[
|J_a(T,T')|\leq \varepsilon_u(T)(1+\sqrt{ I(T,T')}).
\]
Moreover, for $T=0$, one has $|\varepsilon_u(0)|\leq C(R).$
\end{lem}

Given Lemma~\ref{lem:J(T,T')}, the strategy of proof for Proposition~\ref{prop:I(T,T')_n,p} will be the following. We observe that for the choice of family~$a^*=(a^*_n)_n$ given by~\eqref{def:a_n}
\begin{equation*}
a_n^*(t)=\sqrt{\mu_{n+1}^*(\gamma(t))}\sqrt{\frac{\kappa_n^*(\gamma(t))}{\kappa_{n+1}^*(\gamma(t))}},
\end{equation*}
we have thanks to the Lyapunov functional (Proposition~\ref{prop:normL2})
\[
I(0,T')\leq C(R)+C(R)|J_{a^*}(0,T')|.
\]
We deduce from Lemma~\ref{lem:J(T,T')} that $I(0,T')$ is bounded independently of $T'$. As a consequence, a new application of Lemma~\ref{lem:J(T,T')} implies that for all family~$a=(a_n)_n$ satisfying~\eqref{hyp:I(T,T')}, $J_a(T,T')$ is bounded by some $\varepsilon_u(T)$, where this upper bound is independent of~$T'$ and of the choice of $a$.

\begin{proof}[Proof of Lemma~\ref{lem:J(T,T')}]
For $n,p\geq 0$, $n\neq p$, we denote $\eta_{n,p}:=\overline{\zeta_n}\zeta_{n+1}\zeta_p\overline{\zeta_{p+1}}$.

Fix a family $a=(a_n)_n$ of elements of $\classeC^1(\R_+,\C)$ satisfying assumption~\eqref{hyp:I(T,T')}. Since for all $t\in\R_+$, $\sum_nn\gamma_n(t)\leq R^2/2$, there exists $C(R)>0$ such that for all $t\in\R_+$, \[\sum_{\substack{n,p,\\n\neq p}}|a_n(t)a_p(t)\eta_{n,p}(t)|\leq C(R).\] Therefore, one can exchange the summation sign with the time integral:
\begin{equation}\label{ineq:J(T,T')-1}
J_a(T,T')
	=\sum_{\substack{n,p\\n\neq p}}\int_T^{T'}a_n(t)a_p(t)\eta_{n,p}(t)\d t.
\end{equation}
For each term in the series over the indexes $n$ and $p$, we perform an integration by parts by using the differential equation satisfied by $\eta_{n,p}$, which we will now establish.

Recall from~\eqref{eq:dot_zeta} that the time derivative of $\zeta_n$ is
\[
\frac{\d}{\d t}\zeta_n(t)
	=i\omega_n(t)\zeta_n(t)-\alpha Z_n(t)
\]
with
\[
\omega_n(t)=n^2-2\sum_{k=1}^{+\infty}\min(k,n)\gamma_k(t)
\]
and $Z_n$ defined in~\eqref{def:Z_n} as
\begin{equation*}
Z_n(t)
	=\langle u(t)|\cos\rangle \d\zeta_n[u(t)].\cos+\langle u(t)|\sin\rangle \d\zeta_n[u(t)].\sin.
\end{equation*}
Therefore, the time derivative of $\eta_{n,p}=\overline{\zeta_n}\zeta_{n+1}\zeta_p\overline{\zeta_{p+1}}$ writes
\begin{equation}\label{eq:dot_eta}
\frac{\d}{\d t}\eta_{n,p}(t)
	=i\Omega_{n,p}(t)\eta_{n,p}(t)
	-\alpha F_{n,p}(t)
\end{equation}
where
\[
\Omega_{n,p}
	=-\omega_n+\omega_{n+1}+\omega_p-\omega_{p+1}
\]
and
\begin{equation}\label{def:F_n,p}
F_{n,p}
	=\left(\overline{Z_n} \zeta_{n+1}+\overline{\zeta_n}Z_{n+1}\right)\zeta_p\overline{\zeta_{p+1}}
	+\left(Z_p \overline{\zeta_{p+1}}+\zeta_p\overline{Z_{p+1}}\right)\overline{\zeta_n}\zeta_{n+1}.
\end{equation}

Note that
\[
\Omega_{n,p}
	=2\left((n-p)-\sum_{k\geq n+1}\gamma_k+\sum_{k\geq p+1}\gamma_k\right),
\]
hence for $n\neq p$,
\[
|\Omega_{n,p}|\geq 2|n-p|>0.
\]
When $n\neq p$, we can therefore divide the differential equation satisfied by $\eta_{n,p}$ by the factor $i\Omega_{n,p}$  and get
\begin{align*}
\int_T^{T'}a_n(t)a_p(t)\eta_{n,p}(t)\d t
	&=\int_T^{T'}\frac{a_n(t)a_p(t)}{i\Omega_{n,p}(t)}\left(\frac{\d}{\d t}\eta_{n,p}(t)+\alpha F_{n,p}(t)\right)\d t.
\end{align*}
An integration by parts now leads to
\begin{multline}\label{eq:J(T,T')-IPP}
\int_T^{T'}a_n(t)a_p(t)\eta_{n,p}(t)\d t
	=\left[\frac{a_n(t)a_p(t)}{i\Omega_{n,p}(t)}\eta_{n,p}(t)\right]_T^{T'}-\int_T^{T'}\frac{\d}{\d t}\left(\frac{a_n(t)a_p(t)}{i\Omega_{n,p}(t)}\right)\eta_{n,p}(t)\d t
	\\+\alpha\int_T^{T'}\frac{ a_n(t)a_p(t) }{i\Omega_{n,p}(t)}F_{n,p}(t)\d t.
\end{multline}
In order to determine an upper bound for $\sum_{n,p,\,n\neq p}|\int_T^{T'}a_n(t)a_p(t)\eta_{n,p}(t)\d t|$, we consider the three terms in the right-hand side of this equality separately.

\begin{enumerate}
\item For all $t\geq 0$, we have
\[
\sum_{\substack{n,p,\\n\neq p}}\left|\frac{a_n(t)a_p(t)}{i\Omega_{n,p}(t)}\eta_{n,p}(t)\right|
	\leq \left(\sum_{n=0}^{+\infty}\gamma_n(t)\right)^2
	\leq R^4,
\]
therefore the series with general term  $\sum_{n,p,\,\substack{n\neq p}}\left[\frac{a_n(t)a_p(t)}{i\Omega_{n,p}(t)}\eta_{n,p}(t)\right]_T^{T'}$ is absolutely convergent and bounded by some constant $C(R)>0$ independent of $T$ and $T'$.

Moreover, for $N\in\N$, one can cut the sum between the indexes $n\leq N$ and the indexes $n>N$ to deduce
\begin{align*}
\sum_{\substack{n,p,\\n\neq p}}\left|\frac{a_n(t)a_p(t)}{i\Omega_{n,p}(t)}\eta_{n,p}(t)\right|
	&\leq \sum_{\substack{n,p,\, n\neq p\\n\leq N,\, p\leq N}}\left|\frac{a_n(t)a_p(t)}{i\Omega_{n,p}(t)}\eta_{n,p}(t)\right|
	+\hspace{-5pt}\sum_{\substack{n,p,\, n\neq p\\n>N\text{ or }p>N}}\left|\frac{a_n(t)a_p(t)}{i\Omega_{n,p}(t)}\eta_{n,p}(t)\right|\\
	&\leq \left(\sum_{n=0}^N\sqrt{\gamma_n(t)\gamma_{n+1}(t)}\right)^2+2\left(\sum_{n=N+1}^{+\infty}\gamma_n(t)\right)\left(\sum_{p=0}^{+\infty}\gamma_p(t)\right).
\end{align*}
Let $\varepsilon>0$.
Since
\[
\left(\sum_{n=N+1}^{+\infty}\gamma_n(t)\right)\left(\sum_{p=0}^{+\infty}\gamma_p(t)\right)
	\leq \frac{1}{N+1}\left(\sum_{n=N+1}^{+\infty}n\gamma_n(t)\right)\left(\sum_{p=0}^{+\infty}\gamma_p(t)\right)
	\leq \frac{C(R)}{N+1},
\]
we know that there exists $N=N(\varepsilon)\in\N$ such that for all $t\in\R_+$, 
\[
\left(\sum_{n=N+1}^{+\infty}\gamma_n(t)\right)\left(\sum_{p=0}^{+\infty}\gamma_p(t)\right)
	\leq \varepsilon.
	\]
Besides, for all $n$, we have $\gamma_n(t)\gamma_{n+1}(t)\to 0$ as $t\to+\infty$ by weak sequential continuity of the Birkhoff map and the description of the weak limit points in Birkhoff coordinates (see the first point of Theorem~\ref{thm:LaSalle}). Therefore, there exists $T_0=T_0(\varepsilon)$ such that for all $t\geq T_0$,
\[
\left(\sum_{n=0}^N\sqrt{\gamma_n(t)\gamma_{n+1}(t)}\right)^2
	\leq \varepsilon.
\]
We conclude that there exists $\varepsilon_u$ such that $\varepsilon_u(T)\to 0$ as $T\to+\infty$ and for all $T<T'$,
\[\sum_{\substack{n,p\\n\neq p}}\left|\left[\frac{a_n(t)a_p(t)}{i\Omega_{n,p}(t)}\eta_{n,p}(t)\right]_T^{T'}\right|
	\leq \varepsilon_u(T),
\]
moreover $\varepsilon_u$ is independent of the choice of family $a$.

\item Next, we develop the time derivative
\[
\frac{\d}{\d t}\left(\frac{a_n(t)a_p(t)}{i\Omega_{n,p}(t)}\right)
	=\frac{\dot{a_n}(t)a_p(t)+a_n(t)\dot{a_p}(t)}{i\Omega_{n,p}(t)}
	-\frac{a_n(t)a_p(t)\dot{\Omega}_{n,p}(t)}{i\Omega_{n,p}(t)^2}.
\]
From the definition of $\Omega_{n,p}$, we know that
\[
|\dot{\Omega}_{n,p}(t)|
	\leq 2\|\dot{\gamma}(t)\|_{\ell^1_+}.
\]
However, the time derivative of $\gamma_n$ is bounded from inequality~\eqref{ineq:gamma_n_dot} as
\[
|\dot{\gamma_n}(t)|
	\leq C(R)|\langle u(t)|e^{ix}\rangle|\left(|\zeta_{n-1}\zeta_n|+|\zeta_n\zeta_{n+1}|\right),\]
so that
\begin{equation}\label{ineq:dot_gamma}
\|\dot{\gamma}(t)\|_{\ell^1_+}
	\leq C(R)|\langle u(t)|e^{ix}\rangle|,
\end{equation}
and therefore
\[
|\dot{\Omega}_{n,p}(t)|
	\leq 2C(R)|\langle u(t)|e^{ix}\rangle|.
\]
Using the assumptions~\eqref{hyp:I(T,T')} on $a_n$, this implies that there exists $C(R)>0$ such that
\begin{equation}\label{eq:J(T,T')-term2}
\int_T^{T'}\left|\frac{\d}{\d t}\left(\frac{a_n(t)a_p(t)}{i\Omega_{n,p}(t)}\right)\eta_{n,p}(t)\right|\d t
	\leq C(R)\int_T^{T'}|\langle u(t)|e^{ix}\rangle|\frac{|\eta_{n,p}(t)|}{|n-p|}\d t.
\end{equation}
Applying the Cauchy-Schwarz' inequality, we deduce
\begin{multline*}
\sum_{\substack{n,p\\n\neq p}}\int_T^{T'}\left|\frac{\d}{\d t}\left(\frac{a_n(t)a_p(t)}{i\Omega_{n,p}(t)}\right)\eta_{n,p}(t)\right|\d t\\
	\leq C(R)\left(\sum_{\substack{n,p\\n\neq p}}\int_T^{T'}|\langle u(t)|e^{ix}\rangle|^2\frac{\gamma_n(t)}{|n-p|^2}\d t\right)^{\half}
	\left(\sum_{\substack{n,p\\n\neq p}}\int_T^{T'}\gamma_{n+1}(t)\gamma_p(t)\gamma_{p+1}(t)\d t\right)^{\half}.
\end{multline*}
But there exists $C>0$ such that for all $n$, $\sum_{p,\,p\neq n}\frac{1}{|n-p|^2}\leq C$, moreover, for all $t\geq0$, we have $\sum_{n}\gamma_n(t)\leq R^2/2$. We obtain
\begin{equation*}
\sum_{\substack{n,p\\n\neq p}}\int_T^{T'}\left|\frac{\d}{\d t}\left(\frac{a_n(t)a_p(t)}{i\Omega_{n,p}(t)}\right)\eta_{n,p}(t)\right|\d t
	\leq C(R)\left(\int_T^{+\infty}|\langle u(t)|e^{ix}\rangle|^2\d t\right)^{\half}\sqrt{I(T,T')}.
\end{equation*}
	
\item Finally, we prove that
\[
\sum_{\substack{n,p\\n\neq p}}\int_T^{T'}\left|\frac{ a_n(t)a_p(t) }{i\Omega_{n,p}(t)}F_{n,p}(t)\right|\d t
	\leq C(R)\left(\int_T^{+\infty}|\langle u(t)|e^{ix}\rangle|^2\d t\right)^{\half}\sqrt{I(T,T')}.
\]

Note that from the definition of $Z_n=\langle u(t)|\cos\rangle \d\zeta_n[u(t)].\cos+\langle u(t)|\sin\rangle \d\zeta_n[u(t)].\sin$ in formula~\eqref{def:Z_n}, and from Corollary~\ref{cor:dzeta.cos} about the bounds on $\d\zeta_n[u(t)].\cos$ and $\d\zeta_n[u(t)].\sin$, we have 
\begin{equation*}
\sum_{n\geq 0}n|Z_n(t)|^2
	\leq C(R)|\langle u(t)|e^{ix}\rangle|^2.
\end{equation*}
One can apply Cauchy-Schwarz' inequality to deduce
\begin{align*}
\sum_{\substack{n,p\\n\neq p}}\int_T^{T'}\left|\frac{ a_n(t)a_p(t) }{i\Omega_{n,p}(t)}\right|&|\overline{Z_n(t)} \zeta_{n+1}(t)\zeta_p(t)\overline{\zeta_{p+1}(t)}|\d t\\
	&\leq \left(\int_T^{T'}\sum_{\substack{n,p\\n\neq p}} \frac{|Z_n(t)|^2}{|n-p|^2}\d t\right)^{\half}\left(\int_T^{T'}\sum_{\substack{n,p\\n\neq p}}\gamma_{n+1}\gamma_p\gamma_{p+1}\d t\right)^{\half}\\
	&\leq C(R)\left(\int_T^{+\infty}|\langle u(t)|e^{ix}\rangle|^2\d t\right)^{\half}\sqrt{I(T,T')}.
\end{align*}
Using the same strategy to the three other terms composing $F_{n,p}$ (up to exchanging the roles of $n$ and $n+1$ and the roles of $n$ and $p$), we get the desired inequality.
\end{enumerate}

To conclude, we use the square integrability in time of $|\langle u(t)|e^{ix}\rangle|$ (Proposition~\ref{prop:normL2}) to get
\[\int_0^{+\infty}|\langle u(t)|e^{ix}\rangle|^2\d t\leq R^2.\] We have therefore proven that there exists $\varepsilon_u$ satisfying $\varepsilon_u(T)\to 0$ as $T\to+\infty$, and such that for all family $a=(a_n)_n$ satisfying the assumptions in~\eqref{hyp:I(T,T')},
\begin{equation*}
\sum_{\substack{n,p\\n\neq p}}\left|\int_T^{T'}a_n(t)a_p(t)\eta_{n,p}(t)\d t\right|
	\leq \varepsilon_u(T)(1+\sqrt{I(T,T')}),
\end{equation*}
moreover, one has $|\varepsilon_u(0)|\leq C(R)$. Plugging this into equality~\eqref{ineq:J(T,T')-1}, we deduce that
\begin{equation*}
|J_a(T,T')|\leq \varepsilon_u(T)(1+\sqrt{I(T,T')}).
\end{equation*}
\end{proof}

\begin{proof}[Proof of Proposition~\ref{prop:I(T,T')_n,p}]
We first establish a bound of $I(0,T')$. We start from the formula for $\langle u(t)|e^{ix}\rangle 
$ from Lemma~\ref{lem:<u|e^ix>}:
\[
\langle u(t)|e^{ix}\rangle 
	=-\sum_{n=0}^{+\infty}a_n^*(t)\overline{\zeta_n(t)}\zeta_{n+1}(t),
\]
where
\[
a_n^*(t)
	=\sqrt{\mu_{n+1}(t)}\sqrt{\frac{\kappa_n(t)}{\kappa_{n+1}(t)}}.
\]
Now, we expand
\[
|\langle u(t)|e^{ix}\rangle|^2
	=\sum_{n=0}^{+\infty}a_n^*(t)^2\gamma_n(t)\gamma_{n+1}(t)+\Re\left(\sum_{\substack{n,p\\n\neq p}}a_n^*(t)a_p^*(t)\overline{\zeta_n}(t)\zeta_{n+1}(t)\zeta_p(t)\overline{\zeta_{p+1}(t)}\right)
\]
and since $a_n^*\geq\frac{1}{C(R)}$ (see inequality~\eqref{ineq:a_n_bound}), we deduce that
\begin{equation*}
\frac{1}{C(R)^2}\sum_{n=0}^{+\infty}\gamma_n(t)\gamma_{n+1}(t)
	\leq |\langle u(t)|e^{ix}\rangle|^2-\Re\left(\sum_{\substack{n,p\\n\neq p}}a_n^*(t)a_p^*(t)\eta_{n,p}(t)\right).
\end{equation*}
In particular, if we denote $a^*=(a_n^*)_n$, an integration in time leads to the inequality
\begin{equation}\label{ineq:I(T,T')}
I(0,T')
	\leq \int_0^{T'}|\langle u(t)|e^{ix}\rangle|^2\d t+|J_{a^*}(0,T')|
	\leq C(R)+|J_{a^*}(0,T')|.
\end{equation}

We now use Lemma~\ref{lem:J(T,T')} applied to the family $a^*$.
In only remains to check that up to division by some constant $C(R)$, this family satisfies assumption~\eqref{hyp:I(T,T')} of Proposition~\ref{prop:I(T,T')_n,p}. Using inequalities~\eqref{ineq:a_n_bound} and~\eqref{ineq:a_n_lipschitz}, we have
\[
|a_n^*(t)|\leq C(R)
\]
and
\[
|\dot{a_n^*}(t)|
	\leq C(R)\|\dot{\gamma}(t)\|_{\ell^1_+},
\]
so that from inequality~\eqref{ineq:dot_gamma},
there exists $C(R)>0$ such that
\[
|\dot{a_n^*}(t)|
	\leq C(R)|\langle u(t)|e^{ix}\rangle|.
\]

Now, Remark~\ref{rk:I(T)_n,p} following Lemma~\ref{lem:J(T,T')} implies that
\[
|J_{a^*}(0,T')|\leq C_1(R)(1+\sqrt{I(0,T')}).
\]
Therefore, one can use the inequality $2xy\leq x^2+y^2$ on the second term of the right-hand side to get
\[
|J_{a^*}(0,T')|\leq C_2(R)+\frac{1}{2} I(0,T').
\]
Plugging this into inequality~\eqref{ineq:I(T,T')}, we deduce an inequality of the form
\[
I(0,T')\leq C_3(R)+\half I(0,T').
\]
We conclude that $I(0,T')$ is bounded by some constant $2C_3(R)$ independent of $T'$. In particular, one can pass to the limit $T'\to+\infty$ and deduce that for some $C(R)>0$,
\[
\int_0^{+\infty}\sum_{n=0}^{+\infty}\gamma_n(t)\gamma_{n+1}(t)\d t\leq C(R).
\]

To conclude, fix family $a$ satisfying assumption~\eqref{hyp:I(T,T')}. From Lemma~\ref{lem:J(T,T')} and the bound $I(T,T')\leq C(R)$, we deduce that there exists a map $\varepsilon_u$ independent of $a$ such that $\varepsilon_u(T)\to 0$ as $T\to+\infty$ and for all $0\leq T<T'<+\infty$,
\[
\left|J_a(T,T')\right|
	\leq \varepsilon_u(T).
\]
Finally, expanding
\begin{align*}
\int_T^{T'}\left|\sum_{n=0}^{+\infty}a_n(t)\zeta_n(t)\overline{\zeta_{n+1}(t)}\right|^2\d t
	&\leq\int_T^{T'}\sum_{n=0}^{+\infty}|a_n(t)|^2\gamma_n(t)\gamma_{n+1}(t)\d t+\left|\Re(J_a(T,T'))\right|\\
	&\leq I(T,T')+|J_a(T,T')|\\
	&\leq \varepsilon_u(T),
\end{align*}
we see that this integral is bounded independently of $T'$ and $a$, and the bound $\varepsilon_u(T)$ goes to $0$ as $T\to+\infty$. We deduce the second part of the proposition.
\end{proof}

\subsection{Strong relative compactness of trajectories as \texorpdfstring{$t\to+\infty$}{t tends to infinity}}\label{part:strong_convergence}
In this part, we prove points 2 and 3 of Theorem~\ref{thm:LaSalle}: let $u$ be a solution to the damped Benjamin-Ono equation~\eqref{eq:damped}, then there exists a sequence $(\gamma_n^{\infty})_{n\geq 1}\in \ell^1_+$ such that for every weak limit $u_\infty$ in $L^2_{r,0}(\T)$ of $(u(t))_{t\geq 0}$ as $t\to+\infty$, and for all $n\geq 1$, we have $\gamma_n(u_\infty)=\gamma_n^{\infty}$; moreover, if $u_\infty$ is a weak limit associated to a subsequence $(u(t_k))_{k\in\N}$ with $t_k\to+\infty$, then the convergence of $(u(t_k))_{k\in\N}$ to $u_\infty$ is strong in $L^2_{r,0}(\T)$.

We start with the proof of point 2. Fix $n\geq 1$ and recall inequality~\eqref{ineq:gamma_n_dot}
\[
|\dot{\gamma_n}(t)|
	\leq C(R)|\langle u(t)|e^{ix}\rangle|\left(|\zeta_{n-1}\zeta_n|+|\zeta_n\zeta_{n+1}|\right).
\]
Since $|\langle u(t)|e^{ix}\rangle|\in L^2(\R_+)$ (see Proposition~\ref{prop:normL2}) and $|\zeta_{n-1}\zeta_n|\in L^2(\R_+)$ (see Proposition~\ref{prop:I(T,T')_n,p}), we get that $\dot\gamma_n\in L^1(\R_+)$. Therefore, there exists $\gamma_n^{\infty}$ such that $\gamma_n(t)\to\gamma_n^{\infty}$ as $t\to+\infty$. Let now $u_\infty$ be a weak limit in $L^2_{r,0}(\T)$ of a subsequence $(u(t_k))_{k\in\N}$ with $t_k\to+\infty$. By weak sequential continuity of the Birkhoff map, we get that for all $n$, $\gamma_n(u(t_k))\to\gamma_n(u_\infty)$, and therefore $\gamma_n(u_\infty)=\gamma_n^\infty$.

We now prove that $\|u(t)\|_{L^2(\T)}\to\|u_\infty\|_{L^2(\T)}$ as $t\to+\infty$, so that the convergence of $(u(t_k))_{k\in\N}$ to $u_\infty$ is strong in $L^2_{r,0}(\T)$. In order to do so, we consider the generating function
\[
\hamilton_\mu(u)=\sum_{n=0}^{+\infty}\frac{\kappa_n\gamma_n}{\lambda_n+\mu}.
\]
The strategy is as follows. Let $R=\|\zeta(u_0)\|_{h^{\half}_+}$. Using the differential equation satisfied by $t\mapsto \hamilton_\mu(u(t))$, we prove that there exists a map $\varepsilon_u$ such that $\varepsilon_u(t)\to0$ as $t\to+\infty$ and for any weak limit $u_\infty$, for all $\mu\geq R^2+1$,
\[
|\hamilton_\mu(u(t))-\hamilton_\mu(u_\infty)|
	\leq C(R)\frac{\varepsilon_u(t)}{\mu^3}.
\]
Then, for fixed $t\in\R_+$, we compare the asymptotic expansion of $\hamilton_{\mu}(u(t))$ as $\mu\to+\infty$ with the norm $\|u(t)\|_{L^2(\T)}$. Finally, we combine these two steps to get the convergence of $\|u(t)\|_{L^2(\T)}$ to $\|u_\infty\|_{L^2(\T)}$ when $t\to+\infty$.

Fix $\mu\geq R^2+1$. 
For all $t\geq 0$, $\|\zeta(t)\|_{h^{\half}_+}\leq R$, we have $\mu+\lambda_0(t)\geq\mu-R^2>0$, therefore $\hamilton_\mu(u(t))$ is well-defined. Moreover, by weak sequential continuity of the generating function $v\in L^2_{r,0}(\T)\mapsto \hamilton_\mu(v)\in\R$ (see~\cite{GerardKappelerTopalov2020}, Lemma 7), we have $\hamilton_\mu(u(t_k))\to\hamilton_\mu(u_\infty)$ as $k\to+\infty$. We now  quantify the rate of convergence of $\hamilton_\mu(u(t_k))$ to $\hamilton_\mu(u_\infty)$ by estimating the time derivative of $\hamilton_\mu(u(t))$.

\begin{lem}\label{lem:dot_hamilton}
Let $R=\|\zeta(u_0)\|_{h^{\half}_+}$. Then there exists a map $\varepsilon_u$ with $\varepsilon_u(T)\to 0$ as $T\to+\infty$, such that the following holds. For all $T>0$ and for all $\mu\geq R^2+1$, we have
\[
\int_T^{+\infty}\left|\frac{\d}{\d t}\hamilton_{\mu}(u(t))
\right|\d t\leq \frac{\varepsilon_u(T)}{(\mu-R^2)^3}.
\]
\end{lem}

\begin{proof}
Using the expression $\hamilton_\mu(u)=\langle (L_u+\mu\id)^{-1}\un|\un\rangle$, the time derivative of $t\mapsto \hamilton_{\mu}(u(t))$ writes
\[
\frac{\d}{\d t}\hamilton_{\mu}(u(t))
	=\langle T_{\partial_tu}(w_{\mu}^{u(t)})|w_{\mu}^{u(t)}\rangle,
\]
where we denote $w_{\mu}^{u(t)}=(L_{u(t)}+\mu\id)^{-1}\un$. Since $\un\in L^2_+$, one can remove the projector $\Pi$ in the expression of $T_{\partial_t u}$, so that
\[
\frac{\d}{\d t}\hamilton_\mu(u(t))
	=\langle \partial_tu \cdot w_{\mu}^{u(t)}|w_{\mu}^{u(t)}\rangle.
\]
Write $\partial_tu=\partial_x \nabla\hamilton(u)-\alpha P_1(u)$, where $\hamilton(u)$ is the Hamiltonian for the Benjamin-Ono equation without damping~\eqref{eq:bo} and \[P_1(u)=\langle u|\cos\rangle\cos+\langle u|\sin\rangle\sin.
\]
Then
\[
\langle \partial_x \nabla\hamilton(u(t)) \cdot w_{\mu}^{u(t)}|w_{\mu}^{u(t)}\rangle
	=\frac{\d}{\d t}\hamilton_\mu(v(t)),
\]
where $v$ is the global solution to \eqref{eq:bo} satisfying $v(t)=u(t)$ at time $t$. Since the map $t'\mapsto \hamilton_\mu(v(t'))$ is constant, we deduce $\langle \partial_x \nabla\hamilton(u(t)) \cdot w_{\mu}^{u(t)}|w_{\mu}^{u(t)}\rangle=0$. Therefore,
\[
\frac{\d}{\d t}\hamilton_\mu(u(t))
	=-\alpha\langle P_1(u(t)) w_{\mu}^{u(t)}|w_{\mu}^{u(t)}\rangle.
\]
We factor $\langle P_1(u) w_\mu^u|w_\mu^u\rangle$ by using complex numbers
\begin{align*}
\langle P_1(u) w_\mu^u|w_\mu^u\rangle	&=\langle u|\cos\rangle\langle\cos w_\mu^u|w_\mu^u\rangle+\langle u|\sin\rangle\langle\sin w_\mu^u|w_\mu^u\rangle\\
	&=\Re(\langle u|e^{ix}\rangle\langle e^{ix}w_\mu^u|w_\mu^u\rangle).
\end{align*}

In order to complete the proof, it is now enough to show that $\langle e^{ix}w_\mu^u|w_\mu^u\rangle\in L^2(\R_+)$, and that there holds an estimate of the form
\begin{equation}\label{eq:estimate_<Sw|w>}
\left(\int_T^{+\infty}\left|\langle e^{ix}w_\mu^{u(t)}|w_\mu^{u(t)}\rangle\right|^2\d t\right)^{\half}
	\leq \frac{\varepsilon_u(T)}{(\mu-R^2)^3},
\end{equation}
with $\varepsilon_u(T)\to 0$ as $T\to+\infty$. 

First, note that for all $t\geq 0$, since $w_\mu^{u(t)}=(L_{u(t)}+\mu\id)^{-1}\un\in L^2_+$, we have the cancellation $\langle e^{ix}w_\mu^{u(t)}|\un\rangle=0$. Therefore one can write
\(
\langle e^{ix}w_\mu^u|w_\mu^u\rangle=\langle e^{ix}w_\mu^u|w_\mu^u-\un/\mu\rangle.
\)
Then, let us decompose $w_\mu^u$ and $w_\mu^u-\un/\mu$ along the basis of eigenfunctions $(f_n)_n$ of $L_u$ to get
\begin{align*}
\langle e^{ix}w_\mu^u|w_\mu^u\rangle
	&=\langle e^{ix}w_\mu^u|w_\mu^u-\un/\mu\rangle\\
	&=\sum_{n,p\in\N}\langle w_\mu^u|f_n\rangle\overline{\langle w_\mu^u-\un/\mu|f_p\rangle}\langle e^{ix}f_n|f_p\rangle.
\end{align*}
By definition of $w_\mu^u$ as $w_\mu^u=(L_{u}+\mu\id)^{-1}\un$, we have
\[
\langle w_\mu^u|f_n\rangle
	=\frac{\langle \un|f_n\rangle}{\lambda_n+\mu}	
	=\frac{\sqrt{\kappa_n}\zeta_n}{\lambda_n+\mu},
\]
and
\[
\langle w_\mu^u-\un/\mu|f_p\rangle
	=-\frac{\langle \un|f_p\rangle \lambda_p}{\mu(\lambda_p+\mu)}	
	=-\frac{\sqrt{\kappa_p}\zeta_p \lambda_p}{\mu(\lambda_p+\mu)}.
\]
From the formula $\langle e^{ix}f_n|f_p\rangle=\overline{M_{n,p}}$ and Definition~\ref{def:M} of $M_{n,p}$, we conclude
\begin{multline*}
-\mu\langle e^{ix}w_\mu^u|w_\mu^u\rangle
	=\sum_{ \substack{n,p\\p\neq n+1}}\frac{\kappa_p \lambda_p\sqrt{\kappa_n}}{\sqrt{\kappa_{n+1}}}\frac{\sqrt{\mu_{n+1}}\zeta_n\overline{\zeta_p}\zeta_p\overline{\zeta_{n+1}}}{(\lambda_p-\lambda_n-1)(\lambda_n+\mu)(\lambda_p+\mu)}\\
	+\sum_n \frac{\sqrt{\kappa_n\kappa_{n+1}} \lambda_{n+1} \sqrt{\mu_{n+1}}\zeta_n\overline{\zeta_{n+1}}}{(\lambda_n+\mu)(\lambda_{n+1}+\mu)}.
\end{multline*}
Let us consider the two sums separately.

On the one hand, the Cauchy-Schwarz' inequality on the sums over indexes $n$ implies
\begin{multline*}
\left|\sum_{ \substack{n,p,\, p\neq n+1}}\frac{\kappa_p \lambda_p\sqrt{\kappa_n}}{\sqrt{\kappa_{n+1}}}\frac{\sqrt{\mu_{n+1}}\zeta_n\overline{\zeta_p}\zeta_p\overline{\zeta_{n+1}}}{(\lambda_p-\lambda_n-1)(\lambda_n+\mu)(\lambda_p+\mu)}\right|\\
	\leq C(R)\sum_p \frac{\kappa_p\gamma_p \lambda_p}{\lambda_p+\mu}\left(\sum_n\gamma_n\gamma_{n+1}\right)^{\half}\left(\sum_{n,\, n\neq p-1} \frac{1}{(\lambda_p-\lambda_n-1)^2(\lambda_n+\mu)^2}\right)^{\half}.
\end{multline*}
For all $p\in\N$, we have $\sum_{n,\, n\neq p-1}\frac{1}{(\lambda_p-\lambda_n-1)^2}\leq C$, moreover, for all $n\in\N$,
\[
\frac{1}{\lambda_n+\mu}\leq \frac{1}{\mu-R^2}
\] 
and finally,
\[\sum_p \frac{\kappa_p\gamma_p \lambda_p}{\lambda_p+\mu}
	\leq \frac{C(R)}{\mu-R^2}.\]
Therefore,
\begin{equation*}
\left|\sum_{ \substack{n,p,\,p\neq n+1}}\frac{\kappa_p \lambda_p\sqrt{\kappa_n}}{\sqrt{\kappa_{n+1}}}\frac{\sqrt{\mu_{n+1}}\zeta_n\overline{\zeta_p}\zeta_p\overline{\zeta_{n+1}}}{(\lambda_p-\lambda_n-1)(\lambda_n+\mu)(\lambda_p+\mu)}\right|
	\leq \frac{C(R)}{(\mu-R^2)^2}\left(\sum_n\gamma_n\gamma_{n+1}\right)^{\half},
\end{equation*}
so that
\begin{multline*}
\int_T^{+\infty}\left|\sum_{ \substack{n,p,\,p\neq n+1}}\frac{\kappa_p \lambda_p\sqrt{\kappa_n}}{\sqrt{\kappa_{n+1}}}\frac{\sqrt{\mu_{n+1}}\zeta_n\overline{\zeta_p}\zeta_p\overline{\zeta_{n+1}}}{(\lambda_p-\lambda_n-1)(\lambda_n+\mu)(\lambda_p+\mu)}\right|^2\d t\\
	\leq \frac{C(R)}{(\mu-R^2)^4} \int_T^{+\infty}\sum_n\gamma_n(t)\gamma_{n+1}(t)\d t.
\end{multline*}

On the other hand, let us denote
\[
b_n
	:=\frac{\sqrt{\kappa_n\kappa_{n+1}} \lambda_{n+1} \sqrt{\mu_{n+1}}}{(\lambda_n+\mu)(\lambda_{n+1}+\mu)}.
\]
Then, there exists $C(R)>0$ such that for all $n\in\N$,
\[
|b_n|\leq \frac{C(R)}{(\mu-R^2)^2}.
\]
and
\[
|\dot{b_n}|\leq \frac{C(R)}{(\mu-R^2)^2}|\langle u|e^{ix}\rangle|.
\]
As a consequence, Remark~\ref{rk:I(T)_n,p} following Proposition~\ref{prop:I(T,T')_n,p}, applied to the family $b=(b_n)_n$, implies
\begin{align*}
\int_T^{+\infty}\left|\sum_n \frac{\sqrt{\kappa_n\kappa_{n+1}} \lambda_{n+1} \sqrt{\mu_{n+1}}\zeta_n\overline{\zeta_{n+1}}}{(\lambda_n+\mu)(\lambda_{n+1}+\mu)}\right|^2\d t
	&\leq \frac{\varepsilon_u(T)}{(\mu-R^2)^4}.
\end{align*}

To conclude, we have proven the existence of $\varepsilon_u(T)\to 0$ as $T\to+\infty$ such that for all $\mu\geq R^2+1$,
\[
\int_T^{+\infty}\left|\mu\langle e^{ix}w_{\mu}^{u(t)}|w_{\mu}^{u(t)}\rangle\right|^2\d t
	\leq \frac{\varepsilon_u(T)}{(\mu-R^2)^4}.
\]
This completes the proof of inequality~\eqref{eq:estimate_<Sw|w>}, and therefore of the lemma.
\end{proof}

We now study the asymptotic expansion of $\hamilton_\mu(u)$ when $\mu\to+\infty$.

\begin{lem}
Let $u\in L^2_{r,0}(\T)$. Then we have the following asymptotic expansion at order $3$ of the generating function when $\mu\to+\infty$:
\[
\hamilton_\mu(u)
	=\frac{1}{\mu}+\frac{1}{2\mu^3}\|u\|_{L^2(\T)}^2+o\left(\frac{1}{\mu^3}\right).
\]
\end{lem}
\begin{proof}
By using the identity $\frac{1}{1+x}=1-x+\frac{x^2}{1+x}$, we have
\[
\hamilton_\mu(u)
	=\frac{1}{\mu}\sum_{n=0}^{+\infty}\frac{\kappa_n\gamma_n}{\lambda_n/\mu+1}
	=\frac{1}{\mu}\sum_{n=0}^{+\infty}\kappa_n\gamma_n-\frac{1}{\mu^2}\sum_{n=0}^{+\infty}\lambda_n\kappa_n\gamma_n+\frac{1}{\mu^3}\sum_{n=0}^{+\infty}\frac{\lambda_n^2}{\lambda_n/\mu+1}\kappa_n\gamma_n.
\]
Remark that
\[\sum_{n=0}^{+\infty}\kappa_n\gamma_n=\sum_{n=0}^{+\infty}|\langle \un|f_n\rangle|^2=1\]
 and 
\[\sum_{n=0}^{+\infty}\lambda_n\kappa_n\gamma_n
	=\sum_{n=0}^{+\infty}\lambda_n|\langle \un|f_n\rangle|^2
	=-\langle \Pi u|\un\rangle
	=0.
\]
Consequently, the asymptotic development simplifies as
\[
\hamilton_\mu(u)
	=\frac{1}{\mu}+\frac{1}{\mu^3}\sum_{n=0}^{+\infty}\frac{\lambda_n^2}{\lambda_n/\mu+1}\kappa_n\gamma_n.
\]
For $\mu\geq -2\lambda_0(u),$ we have $\lambda_0/\mu+1\geq \frac{1}{2}$, so that for all $n$, $\frac{\lambda_n^2}{\lambda_n/\mu+1}\leq 2\lambda_n^2$. Moreover,
\[\sum_{n=0}^{+\infty}\lambda_n^2\kappa_n\gamma_n=\langle \Pi u|\Pi u\rangle=\half\|u\|_{L^2(\T)}^2.
\]
Therefore, one can pass to the limit $\mu\to+\infty$ in the summation sum and deduce
\[
\hamilton_\mu(u)
	=\frac{1}{\mu}+\frac{1}{2\mu^3}\|u\|_{L^2(\T)}^2+o\left(\frac{1}{\mu^3}\right).
\]
\end{proof}

We apply this lemma to
\[
\hamilton_{\mu}(u(T))=\frac{1}{\mu}+\frac{\|u(T)\|_{L^2(\T)}^2}{2\mu^3}+o_T\left(\frac{1}{\mu^3}\right)
\]
and
\[
\hamilton_{\mu}(u_\infty)=\frac{1}{\mu}+\frac{\|u_\infty\|_{L^2(\T)}^2}{2\mu^3}+o\left(\frac{1}{\mu^3}\right).
\]
Note that all the possible weak limits $u_\infty$ have the same norm in $L^2_{r,0}(\T)$, indeed, we know from point 2 in Theorem~\ref{thm:LaSalle} that for all $n\geq 1$, $\gamma_n(u_\infty)=\gamma_n^\infty$.

In light of Lemma~\ref{lem:dot_hamilton}, we deduce that for fixed $T>0$,
\[
\left|\hamilton_{\mu}(u(T))-\hamilton_{\mu}(u_\infty)\right|
	=\left|\frac{\|u(T)\|_{L^2(\T)}^2-\|u_\infty\|_{L^2(\T)}^2}{2\mu^3}+o_T\left(\frac{1}{\mu^3}\right)\right|
	\leq \frac{\varepsilon_u(T)}{(\mu-R^2)^3}.
\]
We now consider the limit $\mu\to+\infty$ and get
\[
\left|\|u(T)\|_{L^2(\T)}^2-\|u_\infty\|_{L^2(\T)}^2\right|
	\leq 2\varepsilon_u(T).
\]
Taking the limit $T\to+\infty$ in this inequality, we conclude that $\|u(T)\|_{L^2(\T)}\to\|u_\infty\|_{L^2(\T)}$. We have proven that for any weak limit $u_\infty$ of a subsequence $(u(t_k))_k$ in $L^2_{r,0}(\T)$, the convergence is strong in $L^2_{r,0}(\T)$.

\section{Higher-order Sobolev norms}\label{section:higher_Hs}

In this section, we consider a solution with higher regularity: we assume that the initial data $u_0$ belongs to $H^s_{r,0}(\T)$ for some exponent $s\geq 0$. The Birkhoff transformation and its inverse transformation map bounded subsets of  $H^s_{r,0}(\T)$ to bounded subsets of $h^{\half+s}_+$ (see~\cite{GerardKappelerTopalov2020-2}, Proposition 5 in Appendix A). Therefore, it is equivalent to study the image of the $H^s$ norm by this transformation and consider $\sum_{n\geq 0}n^{1+2s}\gamma_n$. This leads us to introduce a Lyapunov functional of the form $P_s=\sum_n w_n\gamma_n$, where $w_n\approx n^{1+2s}$, in order to prove Theorem~\ref{thm:higher_Hs}.

Note that the Cauchy problem for~\eqref{eq:damped} is locally well-posed in $H^s_{r,0}(\T)$, $s\geq 0$, by a simple adaptation of the proof in part~\ref{part:lwp} and by using Corollary~\ref{cor:dzeta.cos}.
\subsection{Formula for the derivative of Sobolev norms}

\begin{prop}
Fix $s\geq 0$, and let $u$ be a solution to the damped Benjamin-Ono equation~\eqref{eq:damped} with initial data $u_0\in H^{s}_{r,0}(\T)$. Define $c_n:=n^{2s}$ with $c_0=0$, $w_n:=\sum_{k=1}^{n-1}c_k$, so that we have $\frac{n^{1+2s}}{C}\leq w_n\leq Cn^{1+2s}$, and denote
\[
P_s(t):=
	\sum_{n\geq 1}w_n\gamma_n(t).
\]
Then
\[
\frac{\d}{\d t}P_s(t)
	=-\alpha\sum_{n=0}^{+\infty}c_na_n^*(t)^2\gamma_n(t)\gamma_{n+1}(t)-\frac{\alpha}{2}\sum_{\substack{n,p\geq0\\n \neq p}}(c_n+c_p)a_n^*(t)a_p^*(t)\eta_{n,p}(t),
\]
with the notation
$a_n^*=\sqrt{\mu_{n+1}}\frac{\sqrt{\kappa_n}}{\sqrt{\kappa_{n+1}}}$ and $\eta_{n,p}=\overline{\zeta_n}\zeta_{n+1}\zeta_p\overline{\zeta_{p+1}}$.
\end{prop}

\begin{proof}
We have already seen in the proof of Lemma~\ref{lem:astuce_gronwall} that the time derivative of $\gamma_n$ is
\[
\frac{\d}{\d t}\gamma_n(u(t))
	=-\alpha\left(\langle u(t)|\cos\rangle \d\gamma_n[u(t)].\cos
	+\langle u(t)|\sin\rangle \d\gamma_n[u(t)].\sin\right),
\]
where
\[
\d\gamma_n[u(t)].\cos-i\d\gamma_n[u(t)].\sin
	=m_{n-1}-m_n
\]
and
\[
m_n
	=-\sqrt{\mu_{n+1}}\frac{\sqrt{\kappa_n}}{\sqrt{\kappa_{n+1}}}\overline{\zeta_n}\zeta_{n+1}.
\]
Moreover, from Lemma~\ref{lem:<u|e^ix>}, the following formula holds
\[
\langle u(t)|\cos\rangle-i\langle u(t)|\sin\rangle
	=\sum_{p\geq 0}m_p.
\]
Therefore, we have
\begin{align*}
\frac{\d}{\d t}\gamma_n(u(t))
	&=-\alpha \left(\Re(m_{n-1}-m_n)\Re\left(\sum_{p\geq 0}m_p\right)+\Im(m_{n-1}-m_n)\Im\left(\sum_{p\geq 0}m_p\right)\right)\\
	&=-\alpha\Re\left((m_{n-1}-m_n)\overline{\sum_{p\geq 0}m_p}\right).
\end{align*}
In particular, a summation leads to
\[
\frac{\d}{\d t}P_s(t)
	=-\alpha\Re\left(\sum_{n\geq1}w_n(m_{n-1}-m_n)\overline{\sum_{p\geq 0}m_p}\right).
\]
We rewrite
\[
\sum_{n\geq1}w_n(m_{n-1}-m_n)
	=\sum_{n\geq0}c_nm_n
\]
with
\[
c_n=w_{n+1}-w_n=n^{2s}.
\]
Then
\[
\frac{\d}{\d t}P_s(t)
	=-\alpha\Re\left(\sum_{n,p\geq0}c_nm_n\overline{m_p}\right),
\]
and taking the real part leads to the formula.
\end{proof}

We now consider the decomposition
\begin{equation*}
\frac{\d}{\d t}P_s(t)
	=-\alpha\sum_{n=0}^{+\infty}c_na_n^*(t)^2\gamma_n(t)\gamma_{n+1}(t)-\frac{\alpha}{2}\sum_{\substack{n,p\geq0\\n\neq p}}(c_n+c_p)a_n^*(t)a_p^*(t)\eta_{n,p}(t).
\end{equation*}
We establish an improvement of Lemma~\ref{lem:J(T,T')} so as to see the second term in the right-hand side as a remainder term compared to the first term in the right-hand side.

In what follows, we drop the star exponent to the term $a_n^*$ because we are only going to consider the only family $a_n=\sqrt{\mu_{n+1}}\sqrt{\frac{\kappa_n}{\kappa_{n+1}}}$. For $T\geq 0$, let us define
\[
J^s(T)
	:=\int_0^{T}\sum_{\substack{n,p\\n\neq p}}(c_n+c_p)a_n(t)a_p(t)\eta_{n,p}(t)\d t.
\]
Assume that for all $t\in\R_+$, we have $\sum_p c_p\gamma_p(t)\leq C(R)$ (in particular, we already know that this condition is always satisfied if $s\leq \half$). Then one can see that $J^s(T)$ is well-defined for all $T\geq 0$.

Our aim is to estimate $J^s(T)$ depending on 
\[
I^s(T)
	:=\int_0^T\sum_{n=0}^{+\infty}c_n\gamma_n(t)\gamma_{n+1}(t)\d t.
\]
In the following lemma, we show that provided some upper bound on $J^s(T)$ is satisfied (assumption~\eqref{hyp:Js} below), then the higher-order Lyapunov functional $P_s$ is bounded.

\begin{lem}
Fix $s\geq 0$. Assume that for some $R> 0$, and for all $t\geq 0$, the Birkhoff coordinates of the solution are bounded in $h^{s}_+$:
\[\sum_n c_n\gamma_n(t)\leq R\]
(the constant $R$ depends on the choice of $s$ since $c_n=n^{2s}$). Also assume that for all $\varepsilon>0$, there exists $C_s(R,\varepsilon)>0$ such that for all $t\geq 0$,
\begin{equation}\label{hyp:Js}
J^s(t)\leq C_s(R,\varepsilon)+C_s(R,\varepsilon)\sqrt{I^s(t)}+\varepsilon I^s(t).
\end{equation}
If the initial data belongs to $h^{\half+s}_+$, i.\@e.\@ if $P_s(0)$ is finite, then $P_s$ is bounded: there exists $C_s'(R)$ such that for all $t\geq 0$,
\[
P_s(t)\leq C_s'(R)+P_s(0),
\]
moreover,
\[
\int_0^{+\infty}\sum_{n=0}^{+\infty}c_n\gamma_n(t)\gamma_{n+1}(t)\d t\leq C_s'(R)+P_s(0).
\]
\end{lem}

\begin{proof}
We use the formula for the time derivative of $P_s$
\[
\frac{\d}{\d t}P_s(t)
	=-\alpha\sum_{n=0}^{+\infty}c_na_n(t)^2\gamma_n(t)\gamma_{n+1}(t)-\frac{\alpha}{2}\sum_{\substack{n,p\geq0\\n \neq p}}(c_n+c_p)a_n(t)a_p(t)\eta_{n,p}(t).
\]
Since for all $n$, $a_n(t)\geq \frac{1}{C(R)}$, an integration in time leads to
\[
P_s(T)-P_s(0)
	\leq -\frac{\alpha}{C(R)^2}
I^s(T)+\frac{\alpha}{2}|J^s(T)|.
\]
We fix $\varepsilon:=\frac{1}{2C(R)^2}$. Using that by assumption,
\begin{align*}
|J^s(T)|
	&\leq C_s(R,\varepsilon)+C_s(R,\varepsilon)\sqrt{I^s(T)}+\frac{1}{2C(R)^2}I^s(T)\\
	&\leq C_s'(R,\varepsilon)+\frac{1}{C(R)^2}I^s(T),
\end{align*}
we deduce that
\[
P_s(T)+\frac{\alpha}{2C(R)^2}I^s(T)
	\leq P_s(0)+\frac{\alpha }{2} C_s'(R,\varepsilon).
\]
\end{proof}

\subsection{The case \texorpdfstring{$s<\half$}{s<1/2}}

In the parts that follow, we prove that assumption~\eqref{hyp:Js} on $J^s(T)$ is satisfied. In this purpose, we first remove the part of $J^s(T)$ that can be bounded with the same strategy as the proof of Lemma~\ref{lem:J(T,T')}. In the case $s<\half$, we obtain a bound on $J^s(T)$ itself, and we deduce that if $u_0\in H^s_{r,0}(\T)$, then the solution $u$ stays bounded in $H^s_{r,0}(\T)$.

\begin{lem}
Let $s\geq 0$. We assume that for all $t\geq 0$, $\sum_{n\geq 0}c_n\gamma_n(t)\leq R$ (where $c_n=n^{2s}$). Then there exists $C(R)>0$ such that the following holds. For all $T>0$,
\[
\left|J^s(T)-\alpha K^s(T)\right|
	\leq C(R),
\]
where 
\begin{multline}\label{eq:def_K^s}
K^s(T)
	=\sum_{\substack{n,p\\n\neq p}}\int_0^T\frac{ a_n(t)a_p(t) }{i\Omega_{n,p}(t)}
	\Big(c_p\left(\overline{Z_n(t)} \zeta_{n+1}(t)+\overline{\zeta_n(t)}Z_{n+1}(t)\right)\zeta_p(t)\overline{\zeta_{p+1}(t)}\\
	+c_n\left(Z_p(t) \overline{\zeta_{p+1}(t)}+\zeta_p(t)\overline{Z_{p+1}(t)}\right)\overline{\zeta_n(t)}\zeta_{n+1}(t)\Big)\d t
\end{multline}
and from~\eqref{def:Z_n},
\begin{equation*}
Z_n
	=\langle u|\cos\rangle \d\zeta_n[u].\cos+\langle u|\sin\rangle \d\zeta_n[u].\sin.
\end{equation*}

Moreover, if $s<\half$, there exists $C_s(R)>0$ such that
\[
\left|J^s(T)\right|
	\leq C_s(R)+C_s(R)\sqrt{ I^s(T)}.
\]
\end{lem}

\begin{proof}
We follow the strategy of proof of Lemma~\ref{lem:J(T,T')}. For $n,p\geq 0$, $n\neq p$, recall the notation  $\eta_{n,p}=\overline{\zeta_n}\zeta_{n+1}\zeta_p\overline{\zeta_{p+1}}$. Since for all $t\geq 0$, $\sum_nc_n\gamma_n(t)\leq R$, there exists $C(R)>0$ such that for all $t\in\R$, 
\[\sum_{\substack{n,p\\n\neq p}}|(c_n+c_p)a_n(t)a_p(t)\eta_{n,p}(t)|\leq C(R).\]
 Therefore, one can exchange the summation sign with the time integral:
\begin{equation*}
J^s(T)
	=\sum_{\substack{n,p\\n\neq p}}\int_0^T(c_n+c_p)a_n(t)a_p(t)\eta_{n,p}(t)\d t.
\end{equation*}
For each term in the series over the indexes $n$ and $p$, we use  equality~\eqref{eq:J(T,T')-IPP} from the proof of Lemma~\ref{lem:J(T,T')}, which came from the differential equation~\eqref{eq:dot_eta} satisfied by $\eta_{n,p}$:
\begin{multline*}
\int_0^Ta_n(t)a_p(t)\eta_{n,p}(t)\d t
	=\left[\frac{a_n(t)a_p(t)}{i\Omega_{n,p}(t)}\eta_{n,p}(t)\right]_0^{T}-\int_0^T\frac{\d}{\d t}\left(\frac{a_n(t)a_p(t)}{i\Omega_{n,p}(t)}\right)\eta_{n,p}(t)\d t
	\\+\alpha\int_0^T\frac{ a_n(t)a_p(t) }{i\Omega_{n,p}(t)}F_{n,p}(t)\d t,
\end{multline*}
with
\begin{equation*}
F_{n,p}
	=(\overline{Z_n} \zeta_{n+1}+\overline{\zeta_n}Z_{n+1})\zeta_p\overline{\zeta_{p+1}}
	+(Z_p \overline{\zeta_{p+1}}+\zeta_p\overline{Z_{p+1}})\overline{\zeta_n}\zeta_{n+1}.
\end{equation*}
We now consider the three terms in the right-hand side separately.

\begin{enumerate}
\item For all $t\geq 0$, we have
\[
\sum_{\substack{n,p\\n\neq p}}c_n\left|\frac{a_n(t)a_p(t)}{i\Omega_{n,p}(t)}\eta_{n,p}(t)\right|
	\leq C(R)\left(\sum_{n=0}^{+\infty}c_n\gamma_n(t)\right)\left(\sum_{p=0}^{+\infty}\gamma_p(t)\right)
	\leq C'(R).
\]
The upper bound is independent of $t$ by assumption. Since $n$ and $p$ play symmetric roles, we also have the estimate $\sum_{\substack{n,p,\,n\neq p}}c_p\left|\frac{a_n(t)a_p(t)}{i\Omega_{n,p}(t)}\eta_{n,p}(t)\right|\leq C'(R)$.

Therefore, the series with general term  $\sum_{n,p,\,\substack{n\neq p}}(c_n+c_p)\left[\frac{a_n(t)a_p(t)}{i\Omega_{n,p}(t)}\eta_{n,p}(t)\right]_0^{T}$ is absolutely convergent and bounded by some constant $C''(R)$.

\item Next, recall from inequality~\eqref{eq:J(T,T')-term2} that
\begin{align*}
\int_0^T\left|\frac{\d}{\d t}\left(\frac{a_n(t)a_p(t)}{i\Omega_{n,p}(t)}\right)\eta_{n,p}(t)\right|\d t
	&\leq C(R)\int_0^T|\langle u(t)|e^{ix}\rangle|\frac{|\eta_{n,p}(t)|}{|n-p|}\d t.
\end{align*}
Since $n$ and $p$ play symmetric roles in the above upper bound, we only estimate
\begin{multline*}
\sum_{\substack{n,p\\n\neq p}}c_n\int_0^T\left|\frac{\d}{\d t}\left(\frac{a_n(t)a_p(t)}{i\Omega_{n,p}(t)}\right)\eta_{n,p}(t)\right|\d t\\
	\leq C(R)\Big(\sum_{\substack{n,p\\n\neq p}}\int_0^T|\langle u(t)|e^{ix}\rangle|^2\frac{c_n\gamma_n(t)}{|n-p|^2}\d t\Big)^{\half}
	\Big(\sum_{\substack{n,p\\n\neq p}}\int_0^Tc_n\gamma_{n+1}(t)\gamma_p(t)\gamma_{p+1}(t)\d t\Big)^{\half}.
\end{multline*}
But there exists $C>0$ such that for all $n$, $\sum_{p,\, p\neq n}\frac{1}{|n-p|^2}\leq C$ and moreover, by assumption, $\sum_{n}c_n\gamma_n(t)\leq R$. We conclude that
\begin{align*}
\sum_{\substack{n,p\\n\neq p}}c_n\int_0^T\left|\frac{\d}{\d t}\left(\frac{a_n(t)a_p(t)}{i\Omega_{n,p}(t)}\right)\eta_{n,p}(t)\right|\d t
	&\leq C(R)\left(\int_0^{+\infty}|\langle u(t)|e^{ix}\rangle|^2\d t\right)^{\half}\sqrt{I^0(T)}.
\end{align*}

\item Finally, we prove that
\begin{multline*}
\sum_{\substack{n,p\\n\neq p}}c_n\int_0^T\left|\frac{ a_n(t)a_p(t) }{i\Omega_{n,p}(t)}\left(\overline{Z_n(t)} \zeta_{n+1}(t)+\overline{\zeta_n(t)}Z_{n+1}(t)\right)\zeta_p(t)\overline{\zeta_{p+1}(t)}\right|\d t\\
	\leq C_s(R)\left(\int_0^{+\infty}|\langle u(t)|e^{ix}\rangle|^2\d t\right)^{\half}\sqrt{I^0(T)}.
\end{multline*}
By symmetry, we would also get
\begin{multline*}
\sum_{\substack{n,p\\n\neq p}}c_p\int_0^T\left|\frac{ a_n(t)a_p(t) }{i\Omega_{n,p}(t)}\left(Z_p(t) \overline{\zeta_{p+1}(t)}+\zeta_p(t)\overline{Z_{p+1}(t)}\right)\overline{\zeta_n(t)}\zeta_{n+1}(t)\right|\d t\\
	\leq C_s(R)\left(\int_0^{+\infty}|\langle u(t)|e^{ix}\rangle|^2\d t\right)^{\half}\sqrt{I^0(T)}.
\end{multline*}

In this purpose, we use Corollary~\ref{cor:dzeta.cos}: if for all $t\geq 0$, we have $\sum_{n}c_n\gamma_n(t)\leq R$ (recall that $c_n=n^{2s}$), then for all $t\geq 0$,
\[
\|d\zeta_n[u].\cos\|_{h^{s}_+}+\|d\zeta_n[u].\sin\|_{h^{s}_+}\leq C(R).
\]

We deduce that $Z_n	=\langle u|\cos\rangle \d\zeta_n[u].\cos+\langle u|\sin\rangle \d\zeta_n[u].\sin$ satisfies
\begin{equation}\label{eq:bound_Zn}
\sum_{n\geq 0}c_n|Z_n(t)|^2
	\leq C(R)|\langle u(t)|e^{ix}\rangle|^2.
\end{equation}
Applying Cauchy-Schwarz' inequality, we conclude
\begin{align*}
\sum_{\substack{n,p\\n\neq p}}c_n\int_0^T&\left|\frac{ a_n(t)a_p(t) }{i\Omega_{n,p}(t)}\overline{Z_n(t)} \zeta_{n+1}(t)\zeta_p(t)\overline{\zeta_{p+1}(t)}\right|\d t\\
	&\leq C(R)\left(\int_0^T\sum_{\substack{n,p\\n\neq p}} \frac{c_n|Z_n(t)|^2}{|n-p|^2}\d t\right)^{\half}\left(\int_0^T\sum_{\substack{n,p\\n\neq p}}c_n\gamma_{n+1}\gamma_p\gamma_{p+1}\d t\right)^{\half}\\
	&\leq C'(R)\left(\int_0^{+\infty}|\langle u(t)|e^{ix}\rangle|^2\d t\right)^{\half}\sqrt{I^0(T)}.
\end{align*}
One can apply the same strategy to the other term (involving $\overline{\zeta_n}Z_{n+1}\zeta_p\overline{\zeta_{p+1}}$) and get the desired inequality.
\end{enumerate}

To conclude, since $\int_0^{+\infty}|\langle u(t)|e^{ix}\rangle|^2\d t\leq R^2$, we have proven that there exists $C(R)>0$ such that for all $T\geq 0$,
\begin{equation*}
|J^s(T)-\alpha K^s(T)|
	\leq C(R)+C(R)\sqrt{I^0(T)}\leq C'(R).
\end{equation*}

For $s<\half$, we copy the proof of point 3. and show that actually
\[
\sum_{\substack{n,p,\,n\neq p}}(c_n+c_p)\int_0^T\left|\frac{ a_n(t)a_p(t) }{i\Omega_{n,p}(t)}F_{n,p}(t)\right|\d t
	\leq C_s(R)\left(\int_0^{+\infty}|\langle u(t)|e^{ix}\rangle|^2\d t\right)^{\half}\sqrt{I^s(T)}.
\]
Indeed, applying Cauchy-Schwarz' inequality, we have
\begin{multline*}
\sum_{\substack{n,p,\,n\neq p}}(c_n+c_p)\int_0^T\left|\frac{ a_n(t)a_p(t) }{i\Omega_{n,p}(t)}\right||\overline{Z_n(t)} \zeta_{n+1}(t)\zeta_p(t)\overline{\zeta_{p+1}(t)}|\d t\\
	\leq\left(\int_0^T\sum_{\substack{n,p,\,n\neq p}} \frac{(c_n+c_p)|Z_n(t)|^2}{|n-p|^2}\d t\right)^{\half}\left(\int_0^T\sum_{\substack{n,p,\,n\neq p}}(c_n+c_p)\gamma_{n+1}\gamma_p\gamma_{p+1}\d t\right)^{\half}.
\end{multline*}
Since $s<\half$, we have an estimate of the form $\sum_{p,\,p\neq n}\frac{c_p}{|n-p|^2}\leq C_sc_n$ by comparing $n$ to $\frac{p}{2}$. This leads to
\begin{multline*}
\sum_{\substack{n,p,\,n\neq p}}(c_n+c_p)\int_0^T\left|\frac{ a_n(t)a_p(t) }{i\Omega_{n,p}(t)}\right||\overline{Z_n(t)} \zeta_{n+1}(t)\zeta_p(t)\overline{\zeta_{p+1}(t)}|\d t\\
	\leq C_s(R)\left(\int_0^{+\infty}|\langle u(t)|e^{ix}\rangle|^2\d t\right)^{\half}\sqrt{I^s(T)}.
\end{multline*}
One can apply the same strategy to the three other terms composing $F_{n,p}$ (see the definition~\eqref{def:F_n,p} of $F_{n,p}$) and get the desired result.

\end{proof}

\subsection{The case \texorpdfstring{$\half\leq s<\frac{3}{2}$}{1/2<=s<3/2}}

We now consider the part left to study when $s\geq \half$, which has been defined in~\eqref{eq:def_K^s}:
\begin{multline*}
K^s(T)
	=\sum_{\substack{n,p\\n\neq p}}\int_0^T\frac{ a_n(t)a_p(t) }{i\Omega_{n,p}(t)}
	\Big(c_p\left(\overline{Z_n(t)} \zeta_{n+1}(t)+\overline{\zeta_n(t)}Z_{n+1}(t)\right)\zeta_p(t)\overline{\zeta_{p+1}(t)}\\
	+c_n\left(Z_p(t) \overline{\zeta_{p+1}(t)}+\zeta_p(t)\overline{Z_{p+1}(t)}\right)\overline{\zeta_n(t)}\zeta_{n+1}(t)\Big)\d t.
\end{multline*}
In order to find an upper bound for $K^s(T)$, we apply Theorem~\ref{thm:dzeta.cos} and get a decomposition for the term $Z_n=\langle u|\cos\rangle \d\zeta_n[u].\cos+\langle u|\sin\rangle \d\zeta_n[u].\sin$. We can therefore write $K^s(T)$ as a sum of terms of the form 
\[\sum_{n,p}\int_0^Tc_p\frac{a_na_p}{i\Omega_{n,p}} G_{n,p} \d t,
\]
where up to exchanging the roles of $n$ and $p$, $G_{n,p}$ is in one of the following two cases.
\begin{itemize}
\item In the first case, we take the terms corresponding to $p_n^*$ or $q_n^*$ from Theorem~\ref{thm:dzeta.cos}
\[
G_{n,p}=\langle u|h_0\rangle r_n\overline{\zeta_{n+\sigma}}\zeta_{n+1}\zeta_p\overline{\zeta_{p+1}},
\]
where $\sigma\in\{-1, 1\}$, $h_0=e^{ix}$ or $e^{-ix}$, and $r_n$ satisfies the estimates 
\[|r_n|\leq C(R) \quad\text{and}\quad |\dot{r_n}|\leq C(R)|\langle u|e^{ix}\rangle|,
\]
up to exchanging the roles of the indexes $n$, $n+1$, $p$ and $p+1$: we can also have
\[
G_{n,p}=\langle u|h_0\rangle r_n\overline{\zeta_{n}}\zeta_{n+1+\sigma}\zeta_p\overline{\zeta_{p+1}}.
\]
\item In the second case, we consider the terms corresponding to $A_{n,k}^*$ or $B_{n,k}^*$ from Theorem~\ref{thm:dzeta.cos}
\[
G_{n,p}=\langle u|h_0\rangle \sum_{k}A_{n,k}\zeta_k\overline{\zeta_{k+1}}\overline{\zeta_n}\zeta_{n+1}\zeta_p\overline{\zeta_{p+1}},
\]
or
\[
G_{n,p}=\langle u|h_0\rangle \sum_{k}A_{n,k}\overline{\zeta_k}\zeta_{k+1}\overline{\zeta_n}\zeta_{n+1}\zeta_p\overline{\zeta_{p+1}},
\]
where $h_0=e^{ix}$ or $e^{-ix}$, and $A_{n,k}$ satisfies the estimates 
\[|A_{n,k}|\leq C(R) \quad\text{and}\quad |\dot{A_{n,k}}|\leq C(R)|\langle u|e^{ix}\rangle|,
\]
up to replacing $A_{n,k}$ by $A_{n+1,k}$, $B_{n,k}$ or $B_{n+1,k}$, which satisfy the same estimates.
\end{itemize}
For the sake of simplicity, we only treat the first expression of $G_{n,p}$ for each situation and assume that $h_0=e^{ix}$:
\begin{itemize}
\item either
\[
G_{n,p}=\langle u|e^{ix}\rangle r_n\overline{\zeta_{n+\sigma}}\zeta_{n+1}\zeta_p\overline{\zeta_{p+1}}
\]
\item or
\[
G_{n,p}=\langle u|e^{ix}\rangle \sum_{k}A_{n,k}\zeta_k\overline{\zeta_{k+1}}\overline{\zeta_n}\zeta_{n+1}\zeta_p\overline{\zeta_{p+1}},
\]
\end{itemize}
but the other cases are similar. 

In both cases, we decompose $\langle u|h_0\rangle$ thanks to the formula from Lemma~\ref{lem:<u|e^ix>}
\[
\langle u|e^{ix}\rangle=-\sum_q a_q\overline{\zeta_q}\zeta_{q+1}
\]
(for treating the case $h_0=e^{-ix}$, we would need to take the conjugate of this expression). Then the strategy follows the proof of Lemma~\ref{lem:J(T,T')}, where $\eta_{n,p}$ is replaced by a longer expression, typically a product $\eta_{n,p,q}^\sigma$ of three or four of terms of the form $\overline{\zeta_q}\zeta_{q+1}\overline{\zeta_{n+\sigma}}\zeta_{n+1}\zeta_p\overline{\zeta_{p+1}}$ or $\overline{\zeta_q}\zeta_{q+1}\zeta_k\overline{\zeta_{k+1}}\overline{\zeta_n}\zeta_{n+1}\zeta_p\overline{\zeta_{p+1}}$. In particular, we use the differential equation satisfied by the oscillating part and integrate by parts.

\subsubsection{Border terms}\label{subpart:5.3.1}

\begin{lem}
We consider a term in the first situation
\[
G_{n,p}=\sum_q a_qr_n\overline{\zeta_q}\zeta_{q+1}\overline{\zeta_{n+\sigma}}\zeta_{n+1}\zeta_p\overline{\zeta_{p+1}}
\]
such that $\sigma\in\{-1, 1\}$,  and for all $t\geq 0$, $|r_n(t)|\leq C(R)|$ and $|\dot{r_n}(t)|\leq C(R)|\langle u(t)|e^{ix}\rangle|$. Assume that $\half\leq s<\frac{3}{2}$ and for all $t\geq 0$, $\sum_nc_n\gamma_n(t)\leq R$ (where $c_n=n^{2s}$). Then for all $\varepsilon>0$, there exists $C_s(R,\varepsilon)$ such that for all $T\geq 0$,
\[
\left|\sum_{\substack{n,p\\n\neq p}}c_p\int_0^T\frac{a_na_p}{i\Omega_{n,p}} G_{n,p}\d t\right|
	\leq C_s(R,\varepsilon)+C_s(R,\varepsilon)\sqrt{I^s(T)}+\varepsilon I^s(T).
\]
\end{lem}

\begin{proof}
The idea is to decompose $G_{n,p}$ between  a part which has a small time derivative, and a part which oscillates rapidly for which we establish a differential equation. We have
\[
\frac{a_na_p}{i\Omega_{n,p}} G_{n,p}=\sum_q\frac{a_na_pa_qr_n}{i\Omega_{n,p}} \eta^{\sigma}_{n,p,q},
\]
where 
\[\eta^{\sigma}_{n,p,q}=\overline{\zeta_q}\zeta_{q+1}\overline{\zeta_{n+\sigma}}\zeta_{n+1}\zeta_p\overline{\zeta_{p+1}}
\]
(for the other possible forms for $G_{n,p}$, we would have a similar formula up to placing the $\sigma$ elsewhere in the product or taking the conjugate of $\overline{\zeta_q}\zeta_{q+1}$).

The oscillating part $\eta^{\sigma}_{n,p,q}$ satisfies the differential equation
\begin{equation}\label{eq:eta_npq}
\frac{\d}{\d t}\eta^{\sigma}_{n,p,q}=i\Omega^{\sigma}_{n,p,q}\eta^{\sigma}_{n,p,q}-\alpha H^{\sigma}_{n,p,q},
\end{equation}
where
\begin{multline*}
H_{n,p,q}=(\overline{Z_{q}}\zeta_{q+1}+\overline{\zeta_{q}}Z_{q+1})\overline{\zeta_{n+\sigma}}\zeta_{n+1}\zeta_p\overline{\zeta_{p+1}}
	+\overline{\zeta_q}\zeta_{q+1}(\overline{Z_{n+\sigma}}\zeta_{n+1}+\overline{\zeta_{n+\sigma}}Z_{n+1})\zeta_p\overline{\zeta_{p+1}}\\
	+\overline{\zeta_q}\zeta_{q+1}\overline{\zeta_{n+\sigma}}\zeta_{n+1}(Z_p\overline{\zeta_{p+1}}+\zeta_p\overline{Z_{p+1}}),
\end{multline*}
and
\begin{align*}
\Omega^{\sigma}_{n,p,q}
	=&-\omega_q+\omega_{q+1}-\omega_{n+\sigma}+\omega_{n+1}+\omega_p-\omega_{p+1}\\
	=&2((1-\sigma)n+q-p)
	+1-\sigma^2-2\sum_{k\geq q+1}\gamma_k\\
	&+2\sum_{k\geq n+\sigma+1}\min(k,n+\sigma)\gamma_k-2\sum_{k\geq n+2}\min(k,n+1)\gamma_k+2\sum_{k\geq p+1}\gamma_k.
\end{align*}
We write
\begin{align*}
\Omega^{\sigma}_{n,p,q}
	&=2\left((1-\sigma)n+q-p+\widetilde{\Omega}^{\sigma}_{n,p,q}\right),
\end{align*}
with $|\widetilde{\Omega}^{\sigma}_{n,p,q}|\leq C_0(R)$ and $|\frac{\d}{\d t}\widetilde{\Omega}^{\sigma}_{n,p,q}|\leq C_0(R)|\langle u|e^{ix}\rangle|$.
For all $n,p,q$ such that $|(1-\sigma)n+q-p|> C_0(R)+1$, one has $\Omega^{\sigma}_{n,p,q}\neq 0$, so that one can divide by $\Omega^{\sigma}_{n,p,q}$ in the differential equation satisfied by $\eta^{\sigma}_{n,p,q}$.

We split
\begin{multline*}
\sum_{\substack{n,p,\,n\neq p}}c_p\int_0^T\frac{a_na_p}{i\Omega_{n,p}} G_{n,p}\d t
	=\sum_{\substack{n,p,q,\,n\neq p\\|(1-\sigma)n+q-p|> C_0(R)+1}}c_p\int_0^T\frac{a_na_pa_qr_n}{i\Omega_{n,p}} \eta^{\sigma}_{n,p,q}\d t\\
	+\sum_{\substack{n,p,q,\,n\neq p\\|(1-\sigma)n+q-p|\leq C_0(R)+1}}c_p\int_0^T\frac{a_na_pa_qr_n}{i\Omega_{n,p}} \eta^{\sigma}_{n,p,q}\d t.
\end{multline*}

$\bullet$ We first show that for all $\varepsilon>0$, there exists $C_s(R,\varepsilon)$ such that the part with indexes $n,p,q$ satisfying $|(1-\sigma)n+q-p|\leq C_0(R)+1$ is bounded by $\varepsilon I^s(T)+ C_s(R,\varepsilon)\sqrt{I^s(T)}$. Indeed, since $|a_na_pa_qr_n|\leq C(R)$, we have
from the Cauchy-Schwarz' inequality
\begin{multline*}
\sum_{\substack{n,p,q,\,n\neq p\\|(1-\sigma)n+q-p|\leq C_0(R)+1}} c_p\left|\int_0^T\frac{a_na_pa_qr_n}{i\Omega_{n,p}} \eta^{\sigma}_{n,p,q}\d t\right|\\
	\leq  C(R)\left(\int_0^T\sum_{\substack{n,p,q,\,n\neq p\\|(1-\sigma)n+q-p|\leq C_0(R)+1}}c_p\gamma_q\gamma_{q+1}\gamma_{n+\sigma}\d t\right)^{\half}\\
	\left(\int_0^T\sum_{\substack{n,p,q,\,n\neq p\\|(1-\sigma)n+q-p|\leq C_0(R)+1}}\frac{c_p}{|n-p|^2}\gamma_{n+1}\gamma_p\gamma_{p+1}\d t\right)^{\half}.
\end{multline*}

For fixed $n$ and $q$, the possible indexes $p$ in the sum belong to an interval of length~$C(R)$. Moreover, since $p\leq C(R)+4n+q$, we have $c_p\leq C(R)\max(c_n,c_q)\leq C(R)c_nc_q$, so that
\[\sum_{\substack{n,p,q,\,n\neq p\\|(1-\sigma)n+q-p|\leq C_0(R)+1}}c_p\gamma_q\gamma_{q+1}\gamma_{n+\sigma}
	\leq C(R)\sum_{n}c_n\gamma_{n+\sigma}\sum_qc_q\gamma_q\gamma_{q+1}
	\leq  C'(R)\sum_qc_q\gamma_q\gamma_{q+1},
\]
and
\[
\int_0^T\sum_{\substack{n,p,q,\,n\neq p\\|(1-\sigma)n+q-p|\leq C_0(R)+1}}c_p\gamma_q\gamma_{q+1}\gamma_{n+\sigma}\d t
	\leq C'(R)I^s(T).\]
Now, fix $\varepsilon>0$. Then there exists $N_0=N_0(\varepsilon)$ such that for all $t\geq 0$, $\sum_{n\geq N_0}\gamma_{n+1}(t)\leq \varepsilon$ (recall that by assumption, for all $t\geq 0$, $\sum_nn^{2s}\gamma_n(t)\leq R$). Moreover, up to increasing $N_0$, we also have
\[\sum_{n,\,|n-p|\geq N_0}\frac{\gamma_{n+1}}{|n-p|^2}\leq C(R)\sum_{n,\,|n-p|\geq N_0}\frac{1}{|n-p|^2}\leq\varepsilon.\]
In the two cases $n\geq N_0$ or $|n-p|\geq N_0$, we deduce
\begin{align*}
\sum_{\substack{n,p,q,\,n\neq p\\ n\geq N_0 \text{ or } |n-p|\geq N_0\\|(1-\sigma)n+q-p|\leq C_0(R)+1}}\frac{c_p}{|n-p|^2}\gamma_{n+1}\gamma_p\gamma_{p+1}
	&\leq  C(R) \sum_{\substack{n,p,\,n\neq p\\ n\geq N_0 \text{ or } |n-p|\geq N_0}}\frac{c_p}{|n-p|^2}\gamma_{n+1}\gamma_p\gamma_{p+1}\\
	&\leq \varepsilon C'(R) \sum_pc_p\gamma_p\gamma_{p+1}.
\end{align*}
Otherwise, we have $n\leq N_0$ and $|n-p|\leq N_0$, therefore
\begin{align*}
\sum_{\substack{n,p,q,\,n\neq p\\ n\leq N_0 \text{ and } |n-p|\leq N_0\\|(1-\sigma)n+q-p|\leq C_0(R)+1}}\frac{c_p}{|n-p|^2}\gamma_{n+1}\gamma_p\gamma_{p+1}
	&\leq  C(R) \sum_{\substack{n,p,\,n\neq p\\ n\leq N_0 \text{ and } |n-p|\leq N_0}}\frac{c_p}{|n-p|^2}\gamma_{n+1}\gamma_p\gamma_{p+1}\\
	&\leq C'(R)N_0^{2s} \sum_p\gamma_p\gamma_{p+1}.
\end{align*}
Since $\int_0^T\sum_p\gamma_p\gamma_{p+1}\d t\leq C(R)$ (see Proposition~\ref{prop:I(T,T')_n,p}), we deduce that for any $\varepsilon>0$, there exists $C_s(R,\varepsilon)$ such that
\[
\int_0^T\sum_{\substack{n,p,q,\,n\neq p\\|(1-\sigma)n+q-p|\leq C_0(R)+1}}\frac{c_p}{|n-p|^2}\gamma_{n+1}\gamma_p\gamma_{p+1}\d t
	\leq \varepsilon C(R) I^s(T)+C_s(R,\varepsilon).
\]

Therefore we conclude that
\[
\sum_{\substack{n,p,q,\,n\neq p\\|(1-\sigma)n+q-p|\leq C_0(R)+1}}c_p\left|\int_0^T\frac{a_na_pa_qr_n}{i\Omega_{n,p}} \eta^{\sigma}_{n,p,q}\d t\right|\\
	\leq  C(R)\sqrt{I^s(T)}\sqrt{\varepsilon C(R) I^s(T)+C_s(R,\varepsilon)}.
\]
In other words, for all $\varepsilon>0$, there exists $C_s(R,\varepsilon)$ such that
\[\sum_{\substack{n,p,q,\,n\neq p\\|(1-\sigma)n+q-p|\leq C_0(R)+1}}c_p\left|\int_0^T\frac{a_na_pa_qr_n}{i\Omega_{n,p}} \eta^{\sigma}_{n,p,q}\d t\right|\\
	\leq \varepsilon I^s(T)+ C_s(R,\varepsilon)\sqrt{I^s(T)}.
\]

$\bullet$ We now tackle the indexes $n$, $p$ and $q$ such that $n\neq p$ and $|(1-\sigma)n+q-p|> C_0(R)+1$ (so that $\Omega^{\sigma}_{n,p,q}\neq0$). We use the differential equation~\eqref{eq:eta_npq} satisfied by $\eta^{\sigma}_{n,p,q}$, and get that
\[
\int_0^T\frac{a_na_p}{i\Omega_{n,p}} G_{n,p}\d t
	= -\int_0^T\frac{a_na_pa_qr_n}{\Omega_{n,p}\Omega^{\sigma}_{n,p,q}}\left(\frac{\d}{\d t}\eta^{\sigma}_{n,p,q}+\alpha H^{\sigma}_{n,p,q}\right)\d t.
\]
We perform an integration by parts for the first term in the right-hand side of this equality:
\begin{multline*}
\int_0^T\frac{a_na_p}{i\Omega_{n,p}}G_{n,p}\d t
	=\left[-\frac{a_na_pa_qr_n}{\Omega_{n,p}\Omega^{\sigma}_{n,p,q}}\eta^{\sigma}_{n,p,q}\right]_0^T
	+\int_0^T\frac{\d}{\d t}\left(\frac{a_na_pa_qr_n}{\Omega_{n,p}\Omega^{\sigma}_{n,p,q}}\right)\eta^{\sigma}_{n,p,q}\d t\\
	-\alpha\int_0^T\frac{a_na_pa_qr_n}{\Omega_{n,p}\Omega^{\sigma}_{n,p,q}} H^{\sigma}_{n,p,q}\d t.
\end{multline*}

It remains to study the summability properties for each of those three terms.

\begin{enumerate}
\item First, we see that since $\sum_pc_p\gamma_p(t)\leq R$ for all $t\geq0$, then
\begin{multline*}
\sum_{\substack{n,p,q,\,n\neq p\\|(1-\sigma)n+q-p|> C_0(R)+1}}c_p\left|\left[-\frac{a_na_pa_qr_n}{\Omega_{n,p}\Omega^{\sigma}_{n,p,q}}\eta^{\sigma}_{n,p,q}\right]_0^T\right|\\
	\leq C(R)\sum_{n,p,q}c_p\sqrt{\gamma_q\gamma_{q+1}\gamma_{n+\sigma}\gamma_{n+1}\gamma_p\gamma_{p+1}}
	\leq C'(R).
\end{multline*}

\item Then, we apply Cauchy-Schwarz' inequality to the second term
\begin{multline*}
\sum_{\substack{n,p,q,\,n\neq p\\ |(1-\sigma)n+q-p|> C_0(R)+1 }} c_p\int_0^T \left|\frac{\d}{\d t}\left(\frac{a_na_pa_qr_n}{\Omega_{n,p}\Omega^{\sigma}_{n,p,q}}\right)\eta^{\sigma}_{n,p,q}\right|\d t\\
	\leq C(R)\left(\int_0^T\sum_{\substack{n,p,q,\,n\neq p\\ |(1-\sigma)n+q-p|> C_0(R)+1 }} c_p\gamma_p\gamma_{n+\sigma}  \left|\frac{\d}{\d t}\left(\frac{a_na_pa_qr_n}{\Omega_{n,p}\Omega^{\sigma}_{n,p,q}}\right)\right|^2\d t\right)^{\half}\\
	\left(\int_0^T\sum_{\substack{n,p,q,\,n\neq p\\ |(1-\sigma)n+q-p|> C_0(R)+1 }}c_p\gamma_{p+1}\gamma_{n+1}\gamma_q\gamma_{q+1}\d t\right)^{\half}.
\end{multline*}

On the one hand, we have
\[
\int_0^T\sum_{n,p,q}c_p\gamma_{p+1}\gamma_{n+1}\gamma_q\gamma_{q+1}\d t
	\leq C(R)I^0(T).
\]

On the other hand, we establish a bound for $\left|\frac{\d}{\d t}\left(\frac{a_na_pa_qr_n}{\Omega_{n,p}\Omega^{\sigma}_{n,p,q}}\right)\right|$. We have already seen that
\[
\left|\frac{\d}{\d t}\left(\frac{a_na_pa_q}{i\Omega_{n,p}}\right)\right|
	\leq C(R) \frac{|\langle u|e^{ix}\rangle|}{|\Omega_{n,p}|}.
\]
Since we also have
\[
\left|\frac{\d}{\d t}\Omega^{\sigma}_{n,p,q}\right|\leq C(R)|\langle u|e^{ix}\rangle|
\]
and by assumption,
\[
\left|\frac{\d}{\d t} r_n\right|\leq C(R)|\langle u|e^{ix}\rangle|,
\]
we deduce
\[
\int_0^T c_p\gamma_p\gamma_{n+\sigma}\left|\frac{\d}{\d t}\left(\frac{a_na_pa_qr_n}{\Omega_{n,p}\Omega^{\sigma}_{n,p,q}}\right)\right|^2\d t
	\leq C(R)\int_0^T \frac{c_p\gamma_p\gamma_{n+\sigma}}{|\Omega_{n,p}\Omega^{\sigma}_{n,p,q}|^2}|\langle u|e^{ix}\rangle|^2\d t.
\]
But since $|\Omega^{\sigma}_{n,p,q}|\geq 2(|(1-\sigma)n+q-p|-C_0(R))\geq 1$, we can use that 
\[\sum_{\substack{q\\ |(1-\sigma)n+q-p|> C_0(R)+1 }}\frac{1}{|\Omega^{\sigma}_{n,p,q}|^2}\leq C(R),\]
and we get that the following series is convergent and bounded by $C(R)$:
\[
\sum_{\substack{n,p,q,\,n\neq p\\ |(1-\sigma)n+q-p|> C_0(R)+1 }}
	\frac{c_p\gamma_p\gamma_{n+\sigma}}{|\Omega_{n,p}\Omega^{\sigma}_{n,p,q}|^2}
		\leq C(R) \sum_{\substack{n,p\\n\neq p}} \frac{c_p\gamma_p\gamma_{n+\sigma} }{|n-p|^2}
		\leq C'(R).
\]
We deduce
\[
\int_0^T\sum_{\substack{n,p,q,\,n\neq p\\|(1-\sigma)n+q-p|> C_0(R)+1 }} c_p\gamma_p\gamma_{n+\sigma}  \left|\frac{\d}{\d t}\left(\frac{a_na_pa_qr_n}{\Omega_{n,p}\Omega^{\sigma}_{n,p,q}}\right)\right|^2\d t
	\leq C(R).
	\]

To conclude, we have proven that
\[
\sum_{\substack{n,p,q,\,n\neq p\\ |(1-\sigma)n+q-p|> C_0(R)+1 }} c_p\int_0^T \left|\frac{\d}{\d t}\left(\frac{a_na_pa_qr_n}{\Omega_{n,p}\Omega^{\sigma}_{n,p,q}}\right)\eta^{\sigma}_{n,p,q}\right|\d t
	\leq C(R)\sqrt{I^0(T)}.
\]

\item Finally, recall that 
\begin{multline*}
H^{\sigma}_{n,p,q}=(\overline{Z_{q}}\zeta_{q+1}+\overline{\zeta_{q}}Z_{q+1})\overline{\zeta_{n+\sigma}}\zeta_{n+1}\zeta_p\overline{\zeta_{p+1}}
	+\overline{\zeta_q}\zeta_{q+1}(\overline{Z_{n+\sigma}}\zeta_{n+1}+\overline{\zeta_{n+\sigma}}Z_{n+1})\zeta_p\overline{\zeta_{p+1}}\\
	+\overline{\zeta_q}\zeta_{q+1}\overline{\zeta_{n+\sigma}}\zeta_{n+1}(Z_p\overline{\zeta_{p+1}}+\zeta_p\overline{Z_{p+1}}).
\end{multline*}
We estimate for instance the term 
\begin{multline*}
\left|\sum_{\substack{n,p,q,\,n\neq p\\ |(1-\sigma)n+q-p|> C_0(R)+1 }} 
	c_p \int_0^T\frac{a_na_pa_qr_n}{\Omega_{n,p}\Omega^{\sigma}_{n,p,q}} \overline{Z_{q}}\zeta_{q+1}\overline{\zeta_{n+\sigma}}\zeta_{n+1}\zeta_p\overline{\zeta_{p+1}
}\d t\right|\\
	\leq C(R)
	\left(\sum_{\substack{n,p,q,\,n\neq p\\ |(1-\sigma)n+q-p|> C_0(R)+1 }} \int_0^T c_p\gamma_p\gamma_{p+1}\gamma_{q+1}\gamma_{n+\sigma}\d t\right)^{\half}\\
	\left(\sum_{\substack{n,p,q,\,n\neq p\\ |(1-\sigma)n+q-p|> C_0(R)+1 }}  \int_0^T\frac{c_p\gamma_{n+1}}{|\Omega_{n,p}\Omega^{\sigma}_{n,p,q}|^2} |Z_{q}|^2\d t\right)^{\half}.
\end{multline*}
On the one hand,
\[
\sum_{\substack{n,p,q,\,n\neq p\\ |(1-\sigma)n+q-p|> C_0(R)+1 }} \int_0^T c_p\gamma_p\gamma_{p+1}\gamma_{q+1}\gamma_{n+\sigma}\d t
	\leq C(R)I^s(T).
\]
On the other hand, we first estimate the sum over indexes $p$:
\begin{multline*}
\sum_{\substack{p,\, p\neq n\\|(1-\sigma)n+q-p|> C_0(R)+1 }}  \frac{c_p}{|\Omega_{n,p}\Omega^{\sigma}_{n,p,q}|^2}\\
	\leq C(R)\sum_{\substack{p,\, p\neq n\\ |(1-\sigma)n+q-p|> C_0(R)+1 }}  \frac{p^{2s}}{|n-p|^2(1+|(1-\sigma)n+q-p|)^2}.
\end{multline*}
When $p\geq 2n$ and $p\geq 2((1-\sigma)n+q)$, the general term of the series is bounded by $C_s\frac{p^{2s}}{p^4}$, and this defines a convergent series since $s<\frac32$. Otherwise, $p^{2s}\leq C_s\max(n,q)^{2s}$ and the series $\sum_{p, p\neq n}\frac{1}{|n-p|^2}$ is convergent. We deduce that
\[
\sum_{\substack{p,\, p\neq n\\|(1-\sigma)n+q-p|> C_0(R)+1 }}  \frac{c_p}{|\Omega_{n,p}\Omega^{\sigma}_{n,p,q}|^2}
	\leq C_s\max(n,q)^{2s},
\]
so that
\[
\sum_{\substack{n,p,q,\, n\neq p\\ |(1-\sigma)n+q-p|> C_0(R)+1 }} \frac{c_p\gamma_{n+1}}{|\Omega_{n,p}\Omega^{\sigma}_{n,p,q}|^2} |Z_{q}|^2
	\leq C_s\sum_{n,q} \max(n,q)^{2s}\gamma_{n+1}|Z_{q}|^2.
\]
But by assumption, $\sum_n n^{2s}\gamma_n\leq R$, moreover, we have seen in~\eqref{eq:bound_Zn} that 
\[
\sum_k c_k|Z_k|^2\leq C(R)|\langle u|e^{ix}\rangle|^2.
\]
Therefore, we get an estimate for this second term
\[
\sum_{\substack{n,p,q,\, n\neq p\\ |(1-\sigma)n+q-p|> C_0(R)+1 }}  \int_0^T\frac{c_p\gamma_{n+1}}{|\Omega_{n,p}\Omega^{\sigma}_{n,p,q}|^2} |Z_{q}|^2\d t
	\leq C_s(R).
\]

To conclude, have proven that
\begin{equation*}
\left|\sum_{\substack{n,p,q,\, n\neq p\\|(1-\sigma)n+q-p|> C_0(R)+1 }} 
	c_p \int_0^T\frac{a_na_pa_qr_n}{\Omega_{n,p}\Omega^{\sigma}_{n,p,q}} \overline{Z_{q}}\zeta_{q+1}\overline{\zeta_{n+\sigma}}\zeta_{n+1}\zeta_p\overline{\zeta_{p+1}
}\d t\right|
	\leq C_s(R)\sqrt{I^s(T)}.
\end{equation*}
This proof also works when exchanging the roles of $n$ and $q$ and with a small variant when exchanging the roles of $n$ and $p$.
\end{enumerate}
\end{proof}

\subsubsection{Central terms}\label{subpart:5.3.2}

\begin{lem}
Let us consider a term in the second situation
\[
G_{n,p}=\sum_{q,k}a_qA_{n,k}\eta_{n,p,q,k},
\]
with
\[
\eta_{n,p,q,k}
	=\overline{\zeta_q}\zeta_{q+1}\zeta_k\overline{\zeta_{k+1}}\overline{\zeta_n}\zeta_{n+1}\zeta_p\overline{\zeta_{p+1}},
\]
and such that for all $t\geq 0$, $|A_{n,k}(t)|\leq C(R)$ and $|\dot{A_{n,k}}(t)|\leq C(R)|\langle u|e^{ix}\rangle|$.
Fix $\half\leq s<\frac{3}{2}$, and assume that for all $t\geq 0$, $\sum_nc_n\gamma_n(t)\leq R$ (where $c_n=n^{2s}$). Then for all $\varepsilon>0$, there exists $C_s(R,\varepsilon)$ such that for all $T\geq 0$,
\[
\left|\sum_{\substack{n,p\\n\neq p}}c_p\int_0^T\frac{a_na_p}{i\Omega_{n,p}} G_{n,p}\d t\right|
	\leq C_s(R,\varepsilon)+C_s(R,\varepsilon)\sqrt{I^s(T)}+\varepsilon I^s(T).
\]
\end{lem}

\begin{proof}
The proof follows the proof in the above part \ref{subpart:5.3.1}. We first compute the differential equation satisfied by  $\eta_{n,p,q,k}$:
\begin{equation}\label{eq:eta_npqk}
\frac{\d}{\d t}\eta_{n,p,q,k}=i\Omega_{n,p,q,k}\eta_{n,p,q,k}-\alpha H_{n,p,q,k},
\end{equation}
where
\begin{align*}
\Omega_{n,p,q,k}
	&=-\omega_q+\omega_{q+1}+\omega_k-\omega_{k+1}-\omega_n+\omega_{n+1}+\omega_p-\omega_{p+1}\\
	&=2(q-k+n-p)-2\sum_{l\geq q+1}\gamma_l+2\sum_{l\geq k+1}\gamma_l-2\sum_{l\geq n+1}\gamma_l+2\sum_{l\geq p+1}\gamma_l\\
	&=2((q-k+n-p)+\widetilde{\Omega}_{n,p,q,k}),
\end{align*}
where $|\widetilde{\Omega}_{n,p,q,k}|\leq C_0(R)$, $|\frac{\d}{\d t}\widetilde{\Omega}_{n,p,q,k}|\leq C_0(R)|\langle u|e^{ix}\rangle|$, and
\begin{multline*}
H_{n,p,q,k}
	=(\overline{Z_{q}}\zeta_{q+1}+\overline{\zeta_{q}}Z_{q+1})\zeta_k\overline{\zeta_{k+1}}\overline{\zeta_n}\zeta_{n+1}\zeta_p\overline{\zeta_{p+1}}
	+\overline{\zeta_q}\zeta_{q+1}(\overline{Z_{k}}\zeta_{k+1}+\overline{\zeta_{k}}Z_{k+1})\overline{\zeta_n}\zeta_{n+1}\zeta_p\overline{\zeta_{p+1}}\\
	+\overline{\zeta_q}\zeta_{q+1}\zeta_k\overline{\zeta_{k+1}}(\overline{Z_n}\zeta_{n+1}+\overline{\zeta_n}Z_{n+1})\zeta_p\overline{\zeta_{p+1}}
	+\overline{\zeta_q}\zeta_{q+1}\zeta_k\overline{\zeta_{k+1}}\overline{\zeta_n}\zeta_{n+1}(Z_p\overline{\zeta_{p+1}}+\zeta_p\overline{Z_{p+1}}).
\end{multline*}
For all $n,p,q,k$ such that $|q-k+n-p|-C_0(R)\geq 1$, one has $\Omega_{n,p,q,k}\neq 0$, so that one can divide by $\Omega_{n,p,q,k}$ in the differential equation satisfied by $\eta_{n,p,q,k}$.

We split
\begin{multline*}
\sum_{\substack{n,p\\n\neq p}}c_p\int_0^T\frac{a_na_p}{i\Omega_{n,p}} G_{n,p}\d t
	=\sum_{\substack{n,p,q,k,\, n\neq p\\|q-k+n-p|>C_0(R)+1}}c_p\int_0^T\frac{a_na_pa_qA_{n,k}}{i\Omega_{n,p}} \eta_{n,p,q,k}\d t\\
	+\sum_{\substack{n,p,q,k,\, n\neq p\\|q-k+n-p|\leq C_0(R)+1}}c_p\int_0^T\frac{a_na_pa_qA_{n,k}}{i\Omega_{n,p}} \eta_{n,p,q,k}\d t.
\end{multline*}

$\bullet$ We first show that the part with indexes $n,p,q$ and $k$ such that $n\neq p$ and $|q-k+n-p|\leq C_0(R)+1$ is bounded by $\varepsilon I^s(T)+ C_s(R,\varepsilon)\sqrt{I^s(T)}$. Indeed, we have
\begin{multline*}
\sum_{\substack{n,p,q,k,\, n\neq p\\|q-k+n-p|\leq C_0(R)+1}}c_p\left|\int_0^T\frac{a_na_pa_qA_{n,k}}{i\Omega_{n,p}} \eta_{n,p,q,k}\d t\right|\\
	\leq C(R)\sum_{\substack{n,p,q,k,\, n\neq p\\|q-k+n-p|\leq C_0(R)+1}}\frac{c_p}{|n-p|}\int_0^T\sqrt{\gamma_q\gamma_{q+1}\gamma_k\gamma_{k+1}\gamma_n\gamma_{n+1}\gamma_p\gamma_{p+1}}\d t,
\end{multline*}
so that from the Cauchy-Schwarz' inequality,
\begin{multline*}
\sum_{\substack{n,p,q,k,\, n\neq p\\|q-k+n-p|\leq C_0(R)+1}}
	c_p\left|\int_0^T\frac{a_na_pa_qA_{n,k}}{i\Omega_{n,p}} \eta_{n,p,q,k}\d t\right|\\
	\leq  C(R)\left(\int_0^T\sum_{\substack{n,p,q,k,\, n\neq p\\|q-k+n-p|\leq C_0(R)+1}}c_p\gamma_q\gamma_{q+1}\gamma_k\gamma_n\d t\right)^{\half}\\
	\left(\int_0^T\sum_{\substack{n,p,q,k,\, n\neq p\\|q-k+n-p|\leq C_0(R)+1}}\frac{c_p}{|n-p|^2}\gamma_{k+1}\gamma_{n+1}\gamma_p\gamma_{p+1}\d t\right)^{\half}.
\end{multline*}

For fixed $n,q$ and $k$, the possible indexes $p$ lie in an interval of length $C(R)$. Moreover, since $p\leq C(R)+q+k+n$, we have $c_p\leq C(R)\max(c_n,c_q,c_k)\leq C(R)c_nc_qc_k$, so that
\[\sum_{\substack{n,p,q,k,\, n\neq p\\|q-k+n-p|\leq C_0(R)+1}}c_p\gamma_q\gamma_{q+1}\gamma_k\gamma_n
	\leq C(R)\sum_{k}c_k\gamma_k\sum_{n}c_n\gamma_n\sum_qc_q\gamma_q\gamma_{q+1}
\]
and
\[\int_0^T\sum_{\substack{n,p,q,k,\, n\neq p\\|q-k+n-p|\leq C_0(R)+1}}c_p\gamma_q\gamma_{q+1}\gamma_k\gamma_n
	\leq C'(R)I^s(T).
\]

Moreover, fix $\varepsilon>0$. Then there exists $N_0=N_0(\varepsilon)$ such that for all $t\geq0$, 
\[\sum_{n\geq N_0}\gamma_{n+1}(t)\leq \varepsilon
\]
and
\[\sum_{n,\,|n-p|\geq N_0}\frac{\gamma_{n+1}}{|n-p|^2}\leq C(R)\sum_{n,\,|n-p|\geq N_0}\frac{1}{|n-p|^2}\leq\varepsilon.\]
In the two cases $n\geq N_0$ or $|n-p|\geq N_0$, we deduce
\begin{align*}
\sum_{\substack{n,p,q,k,\, n\neq p\\ n\geq N_0 \text{ or } |n-p|\geq N_0\\|q-k+n-p|\leq C_0(R)+1}}\frac{c_p}{|n-p|^2}\gamma_{n+1}\gamma_{k+1}\gamma_p\gamma_{p+1}
	&\leq  C(R) \sum_{\substack{n,p,\, n\neq p\\ n\geq N_0 \text{ or } |n-p|\geq N_0}}\frac{c_p}{|n-p|^2}\gamma_{n+1}\gamma_p\gamma_{p+1}\\
	&\leq \varepsilon C(R) \sum_pc_p\gamma_p\gamma_{p+1}.
\end{align*}
Otherwise,
\begin{align*}
\sum_{\substack{n,p,q,k,\, n\neq p\\ n\leq N_0 \text{ and } |n-p|\leq N_0\\|q-k+n-p|\leq C_0(R)+1}}\frac{c_p}{|n-p|^2}\gamma_{n+1}\gamma_{k+1}\gamma_p\gamma_{p+1}
	&\leq  C(R) \sum_{\substack{n,p,\, n\neq p\\ n\leq N_0 \text{ and } |n-p|\leq N_0}}\frac{c_p}{|n-p|^2}\gamma_{n+1}\gamma_p\gamma_{p+1}\\
	&\leq C'(R)N_0^{2s} \sum_p\gamma_p\gamma_{p+1}.
\end{align*}
Since $\int_0^T\sum_p\gamma_p\gamma_{p+1}\d t\leq C(R)$, we deduce that for any $\varepsilon>0$, there exists $C_s(R,\varepsilon)$ such that
\[
\int_0^T\sum_{\substack{n,p,q,k,\, n\neq p\\|q-k+n-p|\leq C_0(R)+1}}\frac{c_p}{|n-p|^2}\gamma_{n+1}\gamma_{k+1}\gamma_p\gamma_{p+1}\d t
	\leq \varepsilon C(R) I^s(T)+C_s(R,\varepsilon).
\]

We conclude that for all $\varepsilon>0$, there exists $C_s(R,\varepsilon)$ such that for all $T\geq 0$,
\[
\sum_{\substack{n,p,q,k,\, n\neq p\\|q-k+n-p|\leq C_0(R)+1}}c_p\left|\int_0^T\frac{a_na_pa_qA_{n,k}}{i\Omega_{n,p}} \eta_{n,p,q,k}\d t\right|\\
	\leq  \varepsilon I^s(T)+C_s(R,\varepsilon)\sqrt{I^s(T)}.
\]


$\bullet$ We now tackle the indexes $n$, $p$, $q$ and $k$ such that $|q-k+n-p|>C_0(R)+1$, so that $\Omega_{n,p,q,k}\neq 0$. We use the differential equation~\eqref{eq:eta_npqk} satisfied by $\eta_{n,p,q,k}$, and get that if $|q-k+n-p|>C_0(R)+1$, then
\[
\int_0^T\frac{a_na_p}{i\Omega_{n,p}} G_{n,p}\d t
	= -\int_0^T\frac{a_na_pa_qA_{n,k}}{\Omega_{n,p}\Omega_{n,p,q,k}}\left(\frac{\d}{\d t}\eta_{n,p,q,k}+\alpha H_{n,p,q,k}\right)\d t.
\]
We perform an integration by parts for the first term in the right-hand side of this equality:
\begin{multline*}
\int_0^T\frac{a_na_p}{i\Omega_{n,p}}G_{n,p}\d t
	=\left[-\frac{a_na_pa_qA_{n,k}}{\Omega_{n,p}\Omega_{n,p,q,k}}\eta_{n,p,q,k}\right]_0^T
	+\int_0^T\frac{\d}{\d t}\left(\frac{a_na_pa_qA_{n,k}}{\Omega_{n,p}\Omega_{n,p,q,k}}\right)\eta_{n,p,q,k}\d t\\
	-\alpha\int_0^T\frac{a_na_pa_qA_{n,k}}{\Omega_{n,p}\Omega_{n,p,q,k}} H_{n,p,q,k}\d t.
\end{multline*}

It remains to study the summability properties for each of those three terms.

\begin{enumerate}
\item First, we see that thanks to the assumption $\sum_pc_p\gamma_p\leq C(R)$, we have
\begin{multline*}
\sum_{\substack{n,p,q,k\\|q-k+n-p|>C_0(R)+1}} c_p\left|\left[-\frac{a_na_pa_qA_{n,k}}{\Omega_{n,p}\Omega_{n,p,q,k}}\eta_{n,p,q,k}\right]_0^T\right|\\
	\leq C(R)\sum_{n,p,q,k}c_p\sqrt{\gamma_q\gamma_{q+1}\gamma_k\gamma_{k+1}\gamma_n\gamma_{n+1}\gamma_p\gamma_{p+1}}
	\leq C'(R).
\end{multline*}

\item Then, we apply Cauchy-Schwarz' inequality to the second term:
\begin{multline*}
\sum_{\substack{n,p,q,k\\|q-k+n-p|>C_0(R)+1}} c_p\int_0^T \left|\frac{\d}{\d t}\left(\frac{a_na_pa_qA_{n,k}}{\Omega_{n,p}\Omega_{n,p,q,k}}\right)\eta_{n,p,q,k}\right|\d t\\
	\leq C(R)\left(\int_0^T\sum_{\substack{n,p,q,k\\|q-k+n-p|>C_0(R)+1}} c_p\gamma_p\gamma_k\gamma_n  \left|\frac{\d}{\d t}\left(\frac{a_na_pa_qA_{n,k}}{\Omega_{n,p}\Omega_{n,p,q,k}}\right)\right|^2\d t\right)^{\half}\\
	\left(\int_0^T\sum_{\substack{n,p,q,k\\|q-k+n-p|>C_0(R)+1}}c_p\gamma_{p+1}\gamma_{k+1}\gamma_{n+1}\gamma_q\gamma_{q+1}\d t\right)^{\half}.
\end{multline*}

On the one hand, we have
\[
\int_0^T\sum_{n,p,q}c_p\gamma_{p+1}\gamma_{k+1}\gamma_{n+1}\gamma_q\gamma_{q+1}\d t
	\leq C(R)I^0(T).
\]

On the other hand, we know that
\[
\left|\frac{\d}{\d t}\left(\frac{a_na_pa_qA_{n,k}}{\Omega_{n,p}\Omega_{n,p,q,k}}\right)\right|
	\leq C(R)\frac{|\langle u|e^{ix}\rangle|}{|\Omega_{n,p}\Omega_{n,p,q,k}|}.
\]
But using that $\sum_{\substack{q,\,|q-k+n-p|>C_0(R)+1}}\frac{1}{|\Omega_{n,p,q,k}|^2}\leq C(R)$, we get that the following series is convergent and bounded by $C(R)$:
\[
\sum_{\substack{n,p,q,k,\,n\neq p\\|q-k+n-p|>C_0(R)+1}}
	\frac{c_p\gamma_p\gamma_k\gamma_n}{|\Omega_{n,p}\Omega_{n,p,q,k}|^2}
		\leq C(R) \sum_{\substack{n,p,k\\n\neq p}} \frac{c_p\gamma_p\gamma_k\gamma_n }{|n-p|^2}
		\leq C'(R).
\]
We deduce
\[
\sum_{\substack{n,p,q,k,\,n\neq p\\|q-k+n-p|>C_0(R)+1}} c_p\int_0^T \left|\frac{\d}{\d t}\left(\frac{a_na_pa_qA_{n,k}}{\Omega_{n,p}\Omega_{n,p,q,k}}\right)\eta_{n,p,q,k}\right|\d t
	\leq C(R)\sqrt{I^0(T)}.
\]

\item Finally, recall that 
\begin{multline*}
H_{n,p,q,k}
	=(\overline{Z_{q}}\zeta_{q+1}+\overline{\zeta_{q}}Z_{q+1})\zeta_k\overline{\zeta_{k+1}}\overline{\zeta_n}\zeta_{n+1}\zeta_p\overline{\zeta_{p+1}}\\
	+\overline{\zeta_q}\zeta_{q+1}(\overline{Z_{k}}\zeta_{k+1}+\overline{\zeta_{k}}Z_{k+1})\overline{\zeta_n}\zeta_{n+1}\zeta_p\overline{\zeta_{p+1}}\\
	+\overline{\zeta_q}\zeta_{q+1}\zeta_k\overline{\zeta_{k+1}}(\overline{Z_n}\zeta_{n+1}+\overline{\zeta_n}Z_{n+1})\zeta_p\overline{\zeta_{p+1}}\\
	+\overline{\zeta_q}\zeta_{q+1}\zeta_k\overline{\zeta_{k+1}}\overline{\zeta_n}\zeta_{n+1}(Z_p\overline{\zeta_{p+1}}+\zeta_p\overline{Z_{p+1}}).
\end{multline*}

We estimate for instance the term 
\begin{multline*}
\left|\sum_{\substack{n,p,q,k,\,n\neq p\\|q-k+n-p|>C_0(R)+1}} 
	c_p \int_0^T\frac{a_na_pa_qA_{n,k}}{\Omega_{n,p}\Omega_{n,p,q,k}} \overline{Z_{q}}\zeta_{q+1}\zeta_k\overline{\zeta_{k+1}}\overline{\zeta_n}\zeta_{n+1}\zeta_p\overline{\zeta_{p+1}
}\d t\right|\\
	\leq C(R)
	\left(\sum_{\substack{n,p,q,k,\,n\neq p\\|q-k+n-p|>C_0(R)+1}}\int_0^T c_p\gamma_p\gamma_{p+1}\gamma_{q+1}\gamma_{k+1}\gamma_n\d t\right)^{\half}\\
	\left(\sum_{\substack{n,p,q,k,\,n\neq p\\|q-k+n-p|>C_0(R)+1}} \int_0^T\frac{c_p\gamma_{n+1}\gamma_k}{|\Omega_{n,p}\Omega_{n,p,q,k}|^2} |Z_{q}|^2\d t\right)^{\half}.
\end{multline*}
On the one hand,
\[
\sum_{\substack{n,p,q,k,\,n\neq p\\|q-k+n-p|>C_0(R)+1}} \int_0^T c_p\gamma_p\gamma_{p+1}\gamma_{q+1}\gamma_{k+1}\gamma_n\d t
	\leq C(R)I^s(T).
\]
On the other hand, we first estimate the sum over indexes $p$:
\begin{multline*}
\sum_{\substack{p,\,p\neq n\\|q-k+n-p|>C_0(R)+1}}  \frac{c_p}{|\Omega_{n,p}\Omega_{n,p,q,k}|^2}\\
	\leq C(R)\sum_{\substack{p,\,p\neq n\\|q-k+n-p|>C_0(R)+1}}  \frac{p^{2s}}{|n-p|^2(1+|q-k+n-p|)^2}.
\end{multline*}
When $p\geq 2n$ and $p\geq 2(q-k+n)$, the general term inside the summation term is bounded by $C_s\frac{p^{2s}}{p^4}$, and this defines a convergent series since $s<\frac32$. Otherwise, we have $p^{2s}\leq \max(n,q,k)^{2s}$ and the series $\sum_{p, p\neq n}\frac{1}{|n-p|^2}$ is convergent. We deduce that
\[
\sum_{\substack{p,\,p\neq n\\|q-k+n-p|>C_0(R)+1}}  \frac{c_p}{|\Omega_{n,p}\Omega_{n,p,q,k}|^2}
	\leq C_s \max(n,q,k)^{2s},
\]
so that
\[
\sum_{\substack{n,p,q,k,\,n\neq p\\|q-k+n-p|>C_0(R)+1}}\frac{c_p\gamma_{n+1}\gamma_k}{|\Omega_{n,p}\Omega_{n,p,q,k}|^2} |Z_{q}|^2
	\leq C_s\sum_{n,q,k} \max(n,q,k)^{2s}\gamma_{n+1}\gamma_k|Z_{q}|^2.
\]
Since $\sum_l l^{2s}\gamma_l\leq C(R)$, and since from inequality~\eqref{eq:bound_Zn} we have
\[
\sum_l l^{2s}|Z_l|^2\leq C(R)|\langle u|e^{ix}\rangle|^2,
\]
we get the bound
\[
\sum_{\substack{n,p,q,k,\,n\neq p\\|q-k+n-p|>C_0(R)+1}}  \int_0^T\frac{c_p\gamma_{n+1}\gamma_k}{|\Omega_{n,p}\Omega_{n,p,q,k}|^2} |Z_{q}|^2\d t
	\leq C_s(R).
\]

To conclude, we have proven that
\begin{equation*}
\left|\sum_{\substack{n,p,q,k,\,n\neq p\\|q-k+n-p|>C_0(R)+1}} 
	\hspace{-10pt} c_p \int_0^T\frac{a_na_pa_qA_{n,k}}{\Omega_{n,p}\Omega_{n,p,q,k}} \overline{Z_{q}}\zeta_{q+1}\zeta_k\overline{\zeta_{k+1}}\overline{\zeta_n}\zeta_{n+1}\zeta_p\overline{\zeta_{p+1}
}\d t\right|
	\leq C_s(R)\sqrt{I^s(T)}.
\end{equation*}

This proof also works when exchanging the roles of $n$ and $q$, and with a small variant when exchanging the roles of $n$ and $p$.
\end{enumerate}
\end{proof}

\begin{rk}\label{rk:higher_Hs}
We have proven that $\sum_n n^{1+2s}\gamma_n$ is bounded for exponents $s$ satisfying $s<\frac{3}{2}$. To increase the range of exponents $s$, we would need to further decompose the term $Z_q$ appearing in point~3 of subparts~\ref{subpart:5.3.1} and~\ref{subpart:5.3.2}. This would lead us to consider finite products of the following form. Let $h_0=\cos$ or $h_0=\sin$, $N=(n_1,\dots,n_k)$ be a finite set of indexes, and $\Sigma=(\sigma_1,\dots,\sigma_k)$ be a finite set of indexes bounded by some finite constant $c(k)$. Then we would have to study
\[
G_{N,\Sigma}=\langle u|h_0\rangle A_{N,\Sigma}\zeta_{n_1}\overline{\zeta_{n_1+\sigma_1}}\dots\zeta_{n_k}\overline{\zeta_{n_k+\sigma_k}},
\]
provided the uniform estimates
\[
|A_{N,\Sigma}|\leq C(R)
\quad\text{and}\quad
|\dot{A}_{N,\Sigma}|\leq C(R)|\langle u|e^{ix}\rangle|.
\]
In particular, we should use the differential equation satisfied by this term and integrate by parts in the same way as before, following the idea of proof from Lemma~\ref{lem:J(T,T')}.
\end{rk}

\bibliography{/home/gassot/Documents/these/references/mybib.bib}{}
\bibliographystyle{abbrv}
\Addresses

\end{document}